\documentclass[reqno]{amsart}
\usepackage{bbm,amsmath, amsfonts, amsthm, amssymb}
\usepackage{bm}
\usepackage[margin=3cm]{geometry}
\usepackage{xcolor}
\usepackage{graphicx,tabularx}

\usepackage{mathtools}
\mathtoolsset{showonlyrefs} 
\graphicspath{{Figures/}}

\let\oldtocsection=\tocsection
\let\oldtocsubsection=\tocsubsection
\let\oldtocsubsubsection=\tocsubsubsection
\renewcommand{\tocsection}[2]{\hspace{0em}\oldtocsection{#1}{#2}}
\renewcommand{\tocsubsection}[2]{\hspace{1em}\oldtocsubsection{#1}{#2}}
\renewcommand{\tocsubsubsection}[2]{\hspace{2em}\oldtocsubsubsection{#1}{#2}}

\usepackage[textsize=tiny]{todonotes}
\definecolor{cadmiumgreen}{rgb}{0.0, 0.42, 0.24}

\newcommand{\DD}{\mathbb{D}}
\newcommand{\EE}{\mathbb{E}}
\newcommand{\NN}{\mathbb{N}}
\newcommand{\PP}{\mathbb{P}}
\newcommand{\RR}{\mathbb{R}}
\newcommand{\TT}{\mathbb{T}}
\newcommand{\D}{\mathrm{d}}
\newcommand{\dq}{\mathrm{d}q}
\newcommand{\ds}{\mathrm{d}s}
\newcommand{\dr}{\mathrm{d}r}
\newcommand{\dt}{\mathrm{d}t}
\newcommand{\du}{\mathrm{d}u}
\newcommand{\dv}{\mathrm{d}v}
\newcommand{\dx}{\mathrm{d}x}
\newcommand{\Df}{\mathrm{D}}
\newcommand{\E}{\mathrm{e}}
\newcommand{\Cc}{\mathcal{C}}

\newcommand{\Ee}{\mathcal{E}}
\newcommand{\Ff}{\mathcal{F}}
\newcommand{\Gg}{\mathcal{G}}

\newcommand{\Uu}{\mathcal{U}}

\newcommand{\Vv}{\mathcal{V}}

\newcommand{\Ww}{\mathcal{W}}
\newcommand{\Xx}{\mathcal{X}}
\newcommand{\ep}{\varepsilon}

\newcommand{\frA}{\mathfrak{A}}
\newcommand{\frB}{\mathfrak{B}}

\newcommand{\B}{\mathrm{B}}
\newcommand{\bdg}{\mathfrak{b}}
\newcommand{\nn}{[\![n]\!]}
\newcommand{\half}{\frac{1}{2}}
\newcommand{\halft}{\frac{3}{2}}
\newcommand{\halfs}{\frac{1}{6}}

\newcommand{\one}{\mathbbm{1}}

\newcommand{\abs}[1]{\left\lvert#1\right\rvert}
\newcommand{\norm}[1]{\left\lVert#1\right\rVert}

\newcommand{\bu}{\overline{u}}

\newcommand{\ww}{\bm{\omega}}

\newcommand{\wphi}{\widetilde{\phi}}
\newcommand{\wpsi}{\widetilde{\psi}}

\newcommand{\kphi}{{\kappa_{\phi}}}
\newcommand{\kpsi}{{\kappa_{\psi}}}

\newcommand{\Const}{\boldsymbol{\mathrm{C}}}

\newcommand{\XTb}{X_T^{t,\overline{X}_t}}
\newcommand{\Pf}{\mathfrak{P}}
\newcommand{\XXi}{\boldsymbol{\Xi}}
\newtheorem{theorem}{Theorem}[section]
\newtheorem{corollary}[theorem]{Corollary}
\newtheorem{lemma}[theorem]{Lemma}
\newtheorem{proposition}[theorem]{Proposition}

\theoremstyle{definition}
\newtheorem{definition}[theorem]{Definition}
\newtheorem{remark}[theorem]{Remark}
\newtheorem{assumption}[theorem]{Assumption}
\theoremstyle{plain}
\newtheorem{example}[theorem]{Example}
\numberwithin{equation}{section}

\definecolor{ocean}{rgb}{0,0.1,0.6}

\definecolor{imperialGreen}{RGB}{2,137,59}

\definecolor{imperialBlue}{RGB}{0, 62, 116}
\definecolor{imperialBrick}{RGB}{165,25,0}
\definecolor{imperialProcess}{RGB}{0,133,202}

\usepackage{hyperref}
\hypersetup{colorlinks=false, linkbordercolor=imperialProcess,
citebordercolor=imperialBrick}

\usepackage[foot]{amsaddr} 

\author{Ofelia Bonesini}
\address{Department of Mathematics, London School of Economics and Political Science}
\email{o.bonesini@lse.ac.uk}

\author{Antoine Jacquier}
\address{Department of Mathematics, Imperial College London, and the Alan Turing Institute}
\email{a.jacquier@imperial.ac.uk}

\author{Alexandre Pannier}
\address{LPSM, Université Paris Cité}
\email{pannier@lpsm.paris}

\title{Rough volatility, path-dependent PDEs and weak rates of convergence}
\date{\today}

\thanks{We would like to thank Christian Bayer for introducing us to the weak rate problem and an anonymous referee for bringing Lemma~\ref{lemma:Friz} to our attention.
The authors acknowledge financial support from the EPSRC grant EP/T032146/1.}
\subjclass[2020]{60G22, 35K10, 65C20, 91G20, 91G60}
\keywords{Rough volatility, path-dependent PDEs, weak rates, stochastic Volterra equations}

\allowdisplaybreaks

\newcolumntype{s}{>{\hsize=.6\hsize}X}
\newcolumntype{y}{>{\hsize=1.4\hsize}X}

\begin{document}

\begin{abstract}
    In the setting of stochastic Volterra equations, and in particular rough volatility models, we show that conditional expectations are the unique classical solutions to path-dependent PDEs. The latter arise from the functional Itô formula developed by [Viens, F., \& Zhang, J. (2019). A martingale approach for fractional Brownian motions and related path dependent PDEs. \emph{Ann. Appl. Probab.}]. We then leverage these tools to study weak rates of convergence for discretised stochastic integrals of smooth functions of a Riemann-Liouville fractional Brownian motion with Hurst parameter $H \in (0,\half)$. These integrals approximate log-stock prices in rough volatility models. We obtain the optimal weak error rates of order~$1$ if the test function is quadratic and of order~$(3H+\half)\wedge1$ if the test function is five times differentiable; in particular these conditions are independent of the value of~$H$.
    \end{abstract}

\maketitle

\tableofcontents


\section{Introduction}

\subsection{Motivation}
Until and unless softwares become capable of handling infinite quantities in finite time, numerical error analysis justifies the application of continuous-time models to discretised real-world applications. In this paper, we consider a time-grid of $(N+1)$ equally spaced points~$(t_i:=\frac{iT}{N})_{i=0}^{N}$ and we study the convergence, as $\Delta=\frac{T}{N}$ goes to zero, of the Euler approximation
\begin{equation}\label{eq:EulerRVol}
    \overline{X}_{t_0}=x_0, \qquad\overline{X}_{t_{i+1}} = \overline{X}_{t_{i}} + \psi(V_{t_i})(B_{t_{i+1}}- B_{t_i})- \half  \psi(V_{t_i})^2 \Delta,
\end{equation}
to the original rough volatility model, where $X$ represents the log-price,
\begin{equation}\label{eq:RVol}
    X_t = x_0 + \int_0^t \psi(V_r)\,\D B_r -\half \int_0^t \psi(V_r)^2 \dr, \quad V_t= \int_0^t (t-r)^{H-\half}\,\D W_r.
\end{equation}
Here, $x_0\in\RR$, $B$ and $W$ are correlated Brownian motions with the natural filtration~$(\Ff_t)_{t\in\TT}$, and $V$ is a Gaussian (Volterra) process with known covariance function, which can therefore be exactly sampled at discrete-time points by Cholesky decomposition. 
The pivotal parameter is~$H$, which controls the H\"older regularity of~$V$. When~$H=\half$, one recovers the well-known Markov and semimartingale theories. The singular case~$H\in(0,\half)$, supported by empirical data under both historical and pricing measures~\cite{BFG16, bolko2022gmm, dupret2022impact, gatheral2018volatility, han2021hurst, romer2022empirical, wu2022rough}, is precisely where both theoretical and numerical analyses go haywire. While~$X$ remains a semimartingale, $V$ fails to be so, and an application of Itô's formula with~$\psi$ Lipschitz continuous shows that
\[
\EE\left[\abs{X_T-\overline{X}_T}^2\right] \lesssim\sum_{i=0}^{N-1} \int_{t_i}^{t_{i+1}} \EE[\abs{V_r-V_{t_i}}^2]\dr \lesssim \Delta^{2H},
\]
namely that the \emph{strong rate of convergence} is of order~$H$. When~$H$ is close to zero, as is generally agreed upon in the community, a practical implementation would not converge in a reasonable amount of time (dividing the error by~$2$ requires multiplying~$N$ by~$2^{1/H}$). This shortcoming is confirmed by the lower bound of~\cite[Theorem 2]{neuenkirch2016order}, and in the more general framework of Stochastic Volterra Equations (SVEs) by~\cite[Theorems~2.3 and~2.4]{li2022numerical}, \cite[Theorems~2.2 and~2.4]{richard2021discrete} (rate~$H$ and~$2H$ for Euler and Milstein schemes respectively) and \cite[Corollary~3.1]{alfonsi2021approximation} (rate~$4H/3$ with a multifactor approximation). A more precise analysis was undertaken 
in~\cite[Theorem~2.1]{fukasawa2023limit} and~\cite[Theorems~2.1 and~2.2]{nualart2023error}, which provide a Central Limit Theorem  for the asymptotic error between an SVE and its Euler discretisation. The rescaling is the strong error rate~$N^H$.

\subsection{Weak rates}

Fortunately, the financial applications we have in mind, such as option pricing, only require to approximate quantities of the form~$\EE[\phi(X_T)]$, for a payoff function~$\phi$.
This falls into the realm of the \emph{weak rate of convergence} which is usually much higher than its strong counterpart. Most friendly stochastic integrals, such as Itô's SDEs where~$H=\half$, exhibit~$1$ for the former and~$\half$ for the latter. One then naturally wonders how the weak error rate evolves in the inhospitable interval~$H\in(0,\half)$. In particular, a positive lower bound is needed to justify the use of simulation schemes.

Unfortunately, the analysis of the weak rate turns out to be an intricate problem, even in relatively simple settings. For $H>\half$, \cite[Theorem 18]{bezborodov2019option} obtained a weak rate of order $H$ where the volatility follows a fractional Ornstein-Uhlenbeck process. Let us also mention~\cite{leon2023euler} for the analysis of fractional SDEs with~$H>\frac{1}{3}$, where a weak rate of~$4H-1$ is obtained via rough paths and Malliavin calculus techniques.

Regarding rough volatility models, the rough Donsker theorem of~\cite[Theorem 2.11]{horvath2017functional}, with weak rate~$H$, was designed in an age when bare convergence was already an achievement,  while the quantization method developed in~\cite{bonesini2021functional} only applies to VIX derivatives. 
Regarding the weak error rate between~\eqref{eq:EulerRVol} and~\eqref{eq:RVol}, the following table is, as far as we are aware, an exhaustive literature review of the results obtained so far.
{\footnotesize
\begin{table}[ht!]
\begin{center}
 \begin{tabularx}{142.9mm}
 {|p{34mm} p{17mm} p{79mm}|}
 \hline
 \textbf{Authors} & \textbf{Weak rate} & \textbf{Assumptions}\\
 \hline
 Bayer, Hall, Tempone
 
 \cite[Theorem~2.1]{bayer2020weak}
 & $H+\half$
 & Linear variance ($\psi(x)=x$), bounded payoff ($\phi \in \mathcal{C}_b^{\lceil \frac{1}{H}\rceil}$)\\
 \hline 
 Bayer, Fukasawa, Nakahara
 \cite[Theorem 1.2]{bayer2022weak}
 & $H+\half$
 & Linear variance ($\psi(x)=x$),
 regular payoff ($\phi \in \mathcal{C}_p^{2+\lceil\frac{1}{2H}\rceil}$) \\ 
 \hline
Gassiat
\cite[Theorem 2.1]{gassiat2022weak}
& $(3H+\half)\wedge1$
& Linear variance ($\psi(x)=x$),
bounded payoff ($\phi \in \mathcal{C}_b^{3+2\lceil\frac{1}{4H}\rceil}$)

\emph{or} 

regular variance ($\psi \in \mathcal{C}_b^2$), cubic payoff ($\phi(x)=x^3$)\\
 \hline
 Friz, Salkeld, Wagenhofer
 \cite[Theorem 1.1]{friz2022weak} 
 & $(3H+\half)\wedge1$
 & Regular variance ($\psi \in \mathcal{C}^m$, $m\in\NN$), polynomial payoff ($\phi(x)=x^n, \, n\le m$)\\
 \hline 
 \end{tabularx}
\end{center}
\end{table}}
\\
They give a positive answer to the question of the lower bound: a minimum of~$\half$ is achieved in all these papers, even supplemented by an unexpected~$3H$ (and a logarithmic correction for~$H=\halfs$) in the last two. On the other hand, all of them rely one way or another on the given structure of the model (the last column) to derive explicit computations. 
Friz, Salkeld and Wagenhofer~\cite{friz2022weak} built on previous ideas by Gassiat~\cite{gassiat2022weak}, which is itself in a similar spirit as the duality approach in~\cite{clement2006duality}, to write down an explicit formula for the moments of the stochastic integral in~\eqref{eq:RVol}. In both cases, a precise fractional analysis is necessary to wind up with a sweet~$3H+\half$ weak rate. Furthermore, let us mention that the constant obtained in~\cite{friz2022weak} blows up as~$n$ goes to infinity, ruling out a polynomial approximation of more general payoffs. The authors also prove the optimality of this rate by providing a counterexample when~$\phi$ is cubic function~\cite[Proposition 6.1]{friz2022weak}.

This problem is difficult because, here more than elsewhere, standard techniques for diffusions rely on It\^o's formula or on the Markov property, and in particular on PDE methods~\cite{talay1986discretisation,talay1990expansion,pages2018numerical}.
Inspired by the functional It\^o formula developed in~\cite[Theorem~3.10 and~3.17]{viens2019martingale} and the resulting path-dependent PDEs (PPDEs) studied in~\cite[Theorem~3.1,~3.2,~3.4]{wang2022path}, we decide to explore a PDE approach analogous to the Markov case. The aim of this paper is then twofold: \medskip
\begin{enumerate}
    \item We present the general PPDE theory for SVEs (Section~\ref{sec:GenPPDE}) and prove that it applies to rough volatility models (Section~\ref{sec:RvolApplication}).
    \item For the Euler approximation~\eqref{eq:EulerRVol} with no drift and assuming  $\phi\in \Cc^5,\psi\in \Cc^{3}$ with polynomial and exponential growth respectively, we derive the optimal weak error rate of order~$(3H+\half)\wedge1$ in Theorem~\ref{th:main}. 
    If~$\phi$ is quadratic we obtain a weak error of rate one.
\end{enumerate}
\medskip
Relaxing the requirements on $\phi$ and $\psi$ is the main achievement of this work compared to former results as it breaks away from the polynomial setting and provides a condition independent of~$H$ (the constant may blow up as~$H\downarrow0$ though).
Although the error analysis is restricted to the case without drift and exactly sampled variance for conciseness, the PDE approach should stretch beyond this setting without difficulty, and even has the potential to extend to larger classes of SVEs and different types of approximation.

\subsection{Path-dependent PDEs}
Viens and Zhang~\cite{viens2019martingale} were interested in understanding the path-dependent structure of conditional expectations of functionals of Volterra processes. Let us start by explaining the idea at the core of their paper in a simple setting. 
For~$0\le t\le s\le T$, the authors discarded the decomposition~$V_s= V_t + [V_s-V_t]$ as~$V$ does not have a flow or the Markov property. Instead, they promoted the decomposition
$$
    V_s = \int_0^s K(s,r) \D W_r = \underbrace{\int_0^t K(s,r) \D W_r}_{\text{\normalsize{$=:\Theta^t_s$}}} + \underbrace{\int_t^s K(s,r) \D W_r}_{\text{\normalsize{$=:I^t_s$}}},
$$
which is an orthogonal decomposition in the sense that, for $t\le s$, $\Theta^t_s$ is~$\Ff_t$-measurable  and~$I^t_s$ is independent of~$\Ff_t$. In financial terms, $\Theta^t_s=\EE[V_s|\Ff_t]$ corresponds to the forward variance. Moreover, it turns out that~$\Theta^t$ encodes precisely the path-dependence one needs to express the conditional expectation. Consider~\eqref{eq:RVol} and ignore the drift for clarity, then it holds
\begin{align*}
    \EE[\phi(X_T)|\Ff_t]
    & = \EE\left[\phi\left(X_t + \int_t^T \psi(V_s) \D B_s\right) \Big|\Ff_t\right]\\
    & = \EE\left[\phi \left(X_t + \int_t^T \psi(\Theta^t_s + I^t_s ) \D B_s  \right) \Big| X_t, \Theta^t_{[t,T]} \right]\\
    &= u\left(t,X_t,\Theta^t_{[t,T]}\right),
\end{align*}
where $u:\TT\times\RR\times C([t,T])\to \RR$ is defined as
\begin{equation}\label{eq:utxomega}
    u(t,x,\omega) 
    := \EE\left[ \phi(X_T) | X_t=x, \, \Theta^t=\omega\right].
\end{equation}
This means that the conditional expectation is only a function of~$X_t$ and~$\Theta^t_{[t,T]}$, in other words we do not need to take into account the past of~$X$ and~$V$, respectively~$X_{[0,t)}$ and~$V_{[0,t)}$. 
We pay the recovery of the Markov property of~$(X,\Theta^t)$ (see~\cite[Proposition 1 and Corollary]{fukasawa2021hedging} for a proof) by lifting the state space to a path space. Let us mention that such a representation was achieved the same year by~\cite[Theorem~2.1]{euch2018perfect} for the rough Heston model using different techniques and that, in the affine framework, this also induces a stochastic PDE representation~\cite[Equation~(3.5)]{abi2019markovian},\cite[Theorem~5.12]{cuchiero2020generalized}.
\\

In a more general setting where~$\bm{X}$ solves an SVE like~\eqref{eq:MainSVE} and~$\bm{\Theta}^t$ is the appropriate ~$\Ff_t$-measurable projection, the main result of~\cite{viens2019martingale} is a functional Itô formula~\cite[Theorem~3.17]{viens2019martingale} for functions of the concatenated process~$\bm{X}\otimes_t \bm{\Theta}^t:= \bm{X} \one_{[0,t)} + \bm{\Theta}^t\one_{[t,T]}$ taking values in~$C(\TT,\RR^d)$, see Section~\ref{sec:ItoPPDE}. 
This involves Fréchet derivatives where the path is only perturbed on~$[t,T]$ in the direction of the kernel.
When the kernel~$K(\cdot,t)$ is singular at~$t$, hence not continuous on~$[t,T]$, the derivative is defined as the limit of those derivatives perturbed by a truncated kernel.

Using BSDE techniques and an intermediary process, in \cite[Theorem 3.2]{wang2022path} the authors demonstrated that conditional expectations of the latter are classical solutions to (semilinear parabolic) path-dependent PDEs, as alluded to in~\cite{viens2019martingale}.
As we detail in Remark~\ref{rem:path_der_special_cases}, the Fréchet derivative---morally, because in a different space---generalises Dupire's vertical derivative~\cite{dupire2019functional} as the path is not frozen and the direction not constant. The Volterra case leads to a different class of PPDEs which do not lie in the scope of previous existence results~\cite{peng2016bsde,ren2014overview}.

We adapt the proof of~\cite[Theorem 3.2]{wang2022path} to include singular kernels, derive uniqueness of the solution, and show the connection with the original SVE. This confirms that the representation~\eqref{eq:utxomega} holds for a broad class of SVEs. The caveat lies in the stringent assumptions made on the road to Itô's formula: the function~$u$ needs to be twice Fréchet differentiable with particular regularity conditions. Exploiting the smoothness of the payoff~$\phi$ and of the volatility function~$\psi$, we are however able to verify these for rough volatility models of interest~\eqref{eq:RVol}, and thus prove well-posedness of the pricing PPDE~\eqref{eq:PPDE_rvol}. 
Outside weak error rates, we believe this opens the gates to a number of PDE applications for rough volatility, such as but not limited to, stochastic control and optimisation, numerical methods \cite{jacquier2019deep,pannier2024path} and regularity of the value function.

\subsection{Method of proof}
The idea of our proof consists in expressing the error~$\EE[\phi(X_T)]-\EE[\phi(\overline{X}_T)]$ as a telescopic sum of conditional expectations between successive discretisation points as in~\cite{talay1990expansion}, ~\cite[Section~7.6]{pages2018numerical} and~\cite[Section 3.3]{bayer2020weak} for instance. We apply the functional Itô formula to the Euler approximation between~$t_i$ and~$t_{i+1}$ and then cancel the time derivative of~$u$ thanks to the PPDE; this is presented in Proposition~\ref{prop:Telescopic}. The problem boils down to studying local errors of the type
\begin{equation}\label{eq:PsiDifference}
\EE\left[\left(\psi^2(V_t)-\psi^2(V_{t_i})\right) \phi''(X_T)\right].
\end{equation}
Only a sharp analysis can exploit this difference without exhibiting the strong rate. Indeed, the variance process is not a semimartingale thus we cannot apply Itô's formula again and standard approximations and estimates necessitate to take absolute values, hence invoking the strong rate.
On the other hand, one can show that~$\EE\big[\psi^2(V_t)-\psi^2(V_{t_i})\big]\lesssim t^{2H}-t_i^{2H}$. This only yields a local weak rate of~$2H$, but the passage to the global rate (after integrating over~$t$ and summing over~$i$) achieves a rate of order one (see Lemma~\ref{lemma:trick_rate_one}). This trick allows us to recover this optimal rate for quadratic payoffs, and inspires us to look for ways of disentangling the stochastic factors in~\eqref{eq:PsiDifference}.

The crucial tool for expressing the differences in~\eqref{eq:PsiDifference} without taking absolute values---which would only yield the strong rate of convergence, $H$---is found in a combination of the Clark-Ocone formula and integration by parts. 
This approach effectively decouples the stochastic terms from the kernels, eliminating the need for crude estimates like those provided by the Cauchy-Schwarz inequality. However, for some of the terms, applying such a strategy once is not sufficient; multiple applications are required but ultimately succeed in achieving the desired rate.
\\

In a nutshell, this paper is a tale that starts from a numerical analysis problem, then dives into stochastic analysis before bouncing back on PDE theory, adds a salty touch of Malliavin calculus and ends its course with fractional calculus.

\subsection{Organisation of the paper}
    The paper is structured as follows:
    In Section~\ref{sect:def_and_ito_form}, we introduce the framework and definitions needed to extend the It\^{o} formula to functionals with super-polynomial growth.
    In Section~\ref{sec:GenPPDE}, the connection with the path-dependent PDE is established, as in the Markovian setting. We apply these results to rough volatility models in Section~\ref{sec:RvolApplication} and postpone the proofs to Section~\ref{sec:proof_RvolApplication}. 
     In Section~\ref{sec:weak_rates_of_conv} we state our main result for the weak rates of convergence, whose proof is carried out in details in Sections~\ref{sec:Bone} and~\ref{sec:Btwo}.

\section{The functional It\^o formula and the path-dependent PDE}\label{sec:ItoPPDE}

\subsection{Notations}
    We write, for $m,n \in \NN$ with $m <n$,  $[\![m,n]\!]:=\{m,\dots,n\}$ and $\nn := [\![1,n]\!]=\{1,\ldots, n\}$. 
    We fix a finite time horizon~$T>0$ and denote the corresponding time interval as $\TT:=[0,T]$.
    For a couple of topological spaces $\mathcal{Y}$ and $\mathcal{Z}$, the set $\Cc^0(\mathcal{Y,Z})$ (resp.  $D^0(\mathcal{Y,Z})$) represents the set of continuous (resp. c\`adl\`ag) functions from $\mathcal{Y}$ to $\mathcal{Z}$.
    We write $f(x)\lesssim g(x)$ for two positive functions $f,g$ if there exists $c>0$ independent of~$x$ such that~$f(x)\le c g(x)$. We only use this notation when it is clear that the constant~$c$ is inconsequential. We define~$\bdg_p$ to be the BDG constant for any~$p>1$.

We consider multi-dimensional stochastic Volterra equations (SVE) given by
    \begin{equation}\label{eq:MainSVE}
\bm{X}_t = x + \int_0^t b(t,r,\bm{X}_r)\dr + \int_0^t \sigma(t,r,\bm{X}_r)\,\D \bm{W}_r,
    \end{equation}
taking values in~$\RR^d$, with $d\geq 1$.
The coefficients~$b:\TT^2\times\RR^d\to\RR^d$ and $\sigma:\TT^2\times\RR^d\to\RR^{d\times m}$ are Borel-measurable functions, 
and~$\bm{W}$ is an~$m$-dimensional Brownian motion on the filtered probability space~$(\Omega,\Ff,\{\Ff_t\}_{t\in\TT},\PP)$ satisfying the usual conditions. 
For each~$t\in \TT$, we introduce the key~$\Ff_t$-measurable process~$(\bm{\Theta}^t_s)_{s\ge t}$:
\begin{equation}\label{eq:ThetaDef}
\bm{\Theta}^t_s = x + \int_0^t b(s,r,\bm{X}_r)\dr + \int_0^t \sigma(s,r,\bm{X}_r)\,\D \bm{W}_r.
\end{equation}
We highlight one important property: fixing~$s\in\TT$ and viewing the index~$t \in [0,s]$ as time, $(\bm{\Theta}^t_s - \int_0^t b(s,r,\bm{X}_r)\dr)_{t\in[0,s]}$ is a martingale (provided the right integrability conditions).

\subsection{Definitions and the functional It\^o formula}\label{sect:def_and_ito_form}
We need to introduce a number of notations that lead to the functional It\^o formula.
The following is a summary of \cite[Section~3.1]{viens2019martingale} with the small twist that we allow for faster than polynomial growth.
We define the following spaces and distances: 
\begin{alignat}{2}
    &\Ww := \Cc^0 (\TT, \mathbb{R}^d), \qquad &&\overline{\Ww} := \Df^0 (\TT, \mathbb{R}^d), \qquad \Ww_t = \Cc^0 ([t, T], \mathbb{R}^d); \nonumber\\
&\Lambda := \TT \times \Ww, \qquad &&\overline{\Lambda} := \left\{(t, \bm{\omega})\in\TT \times \overline{\Ww}:\bm{\omega}\!\mid_{[t, T]}\in\Ww_t\right\} \label{eq:defs_spaces}\\
&\norm{\bm{\omega}}_{\TT} = \sup_{t \in \TT} \abs{\bm{\omega}_t}, \qquad &&\bm{d}((t,\bm{\omega}),(t',\bm{\omega}')):= \abs{t-t'} + \norm{\bm{\omega}-\bm{\omega}'}_{\TT}.\nonumber
\end{alignat}
For two paths~$\bm{\omega},\bm{\theta}$ on $\TT$ we define their concatenation at time~$t\in\TT$ as
$$
\bm{\omega} \otimes_t \bm{\theta} := \bm{\omega} \one_{[0,t)} + \bm{\theta} \one_{[t,T]}.
$$
Let $\Cc^0(\overline{\Lambda}):=\Cc^0(\overline{\Lambda}, \RR)$ denote the set of functions~$U:\overline{\Lambda}\to\RR$ continuous under $\bm{d}$ and define the right time derivative
$$
\partial_t U(t,\bm{\omega}):= \lim_{\ep\downarrow0}\frac{U(t+\ep,\bm{\omega})-U(t,\bm{\omega})}{\ep},
$$
for all~$(t,\bm{\omega})\in\overline{\Lambda}$, provided the limit exists.  {For all~$(t,\bm{\omega})\in\overline{\Lambda}$, we introduce the spatial Fréchet derivative~$\Ww\ni\eta\mapsto \langle\partial_{\bm{\omega}} U(t,\bm{\omega}),\eta\rangle\in\RR$ with respect to~$\bm{\omega}\lvert_{[t,T]}$:
\begin{equation}
U(t,\bm{\omega}+\eta\one_{[t,T]})-U(t,\bm{\omega}) = \langle \partial_{\bm{\omega}} U(t,\bm{\omega}),\eta\rangle
    +o(\norm{\eta\one_{[t,T]}}_{\TT}).
    \label{eq:PathwiseDerDef}
\end{equation}
}
This is a linear operator and, if it exists, it is equal to the Gateaux derivative
\begin{equation*}
 \langle \partial_{\bm{\omega}} U(t,\bm{\omega}),\eta\rangle
    = \lim_{\ep\downarrow0}\frac{U(t,\bm{\omega}+\ep\eta\one_{[t,T]})-U(t,\bm{\omega})}{\ep}.
\end{equation*}
 {For all~$(t,\bm{\omega})\in\overline{\Lambda}$, we similarly define the second derivative~$\Ww^2\ni(\eta^{(1)},\eta^{(2)})\mapsto\langle\partial_{\bm{\omega}\bm{\omega}}U(t,\bm{\omega}),(\eta^{(1)},\eta^{(2)})\rangle \in\RR$ with respect to $\bm{\omega}\lvert_{[t,T]}$:
\begin{equation}
\langle \partial_{\bm{\omega}} U(t,\bm{\omega}+\eta^{(2)}\one_{[t,T]}),\eta^{(1)}\rangle - \langle \partial_{\bm{\omega}} U(t,\bm{\omega}),\eta^{(1)}\rangle =\langle\partial_{\bm{\omega}\bm{\omega}}U(t,\bm{\omega}),(\eta^{(1)},\eta^{(2)})\rangle
+o\left(\norm{\eta^{(2)}\one_{[t,T]}}_{\TT}\right).
    \label{eq:PathwiseDerDef2}
\end{equation}
}

The following assumption ensures well-posedness of~\eqref{eq:MainSVE}.
\begin{assumption}\label{assu:wellposedness}The coefficients $b$ and $\sigma$ are such that:
\begin{enumerate}
    \item[(i)] The SVE~\eqref{eq:MainSVE} has a weak solution $(\bm{X},\bm{W})$;
    \item[(ii)] For all~$p\ge1$, $\EE[\|\bm{X}\|_\TT^p]=\EE[\sup_{t\in \TT} \abs{\bm{X}_t}^p]<\infty$.
\end{enumerate}
\end{assumption}
This assumption, from~\cite{viens2019martingale}, only requires $\bm{X}$ to have finite moments forcing us to take into account only functions of~$\bm{X}$  with polynomial growth. 
We would like to relax this assumption.
\begin{definition}
    Given~\eqref{eq:MainSVE}, 
    let $\Xx:=\{G:\RR\to\RR: \EE[\|G({\bm{X}})\|_\TT^p]<\infty \text{ for all }p>1\}$.
\end{definition}
\begin{example}\
\begin{enumerate}
    \item If $b$ and $\sigma$ have linear growth in space then standard arguments show that~$\EE[\|{\bm{X}}\|_\TT^p]$ is finite for all~$p>1$, so that~$\Xx$ includes all functions with at most polynomial growth.
    \item If $b, \sigma$ are independent of~$x$ then~$\bm{X}$ is a Gaussian Volterra process, thus $\EE[\E^{p\norm{\bm{X}}_{\TT}}]<\infty$ for all~$p>1$ \cite[Lemma 6.13]{jacquier2021rough} and~$\Xx$ includes all functions of at most exponential growth. 
    \item The model we have in mind for relaxing the polynomial growth assumption of~\cite{viens2019martingale} is a rough volatility model where the second component (the stochastic process driving the volatility) is a Gaussian Volterra process and the first (the log-price) is a non-Gaussian stochastic integral. The latter has finite moments of all orders while the former has finite exponential moments (by item (2) of this list). Hence in this framework the set $\Xx$ includes all functions with at most polynomial growth in the first component and exponential growth in the second, that is for all~$x\in\RR^2$,
    $$
G(x) \le C\left(1+ |x^{(1)}| + \exp\{x^{(2)}\}\right), 
\quad \text{for some } C>0,
    $$
\end{enumerate}
\end{example}

A thorough examination of the proofs in~\cite[Theorems 3.10 and 3.17]{viens2019martingale} confirms that the functional It\^o formula still stands with the new growth definitions that follow.
\begin{definition}
Let $U\in \Cc^0(\overline{\Lambda})$ such that~$\partial_{\bm{\omega}} U$ exists for all~$(t,\bm{\omega})\in\overline{\Lambda}$.
For~$G\in\Xx$, we say that $\partial_{\bm{\omega}} U$ has $G$-growth if 
$$
\abs{\langle \partial_{\bm{\omega}} U(t,\bm{\omega}),\eta\rangle}\lesssim \norm{G(\bm{\omega})}_{\TT} \norm{\eta\one_{[t,T]}}_{\TT}, \qquad \text{ for all  }\eta\in\Ww, \text{ for all  }t\in \TT,
$$
and similarly for all 
$\eta^{(1)}, \eta^{(2)}\in\Ww$
and all $t\in \TT$,
$$
    \abs{\langle \partial_{\bm{\omega}\bm{\omega}} U(t,\bm{\omega}),(\eta^{(1)},\eta^{(2)})\rangle}\lesssim \norm{G(\bm{\omega})}_{\TT}\norm{|\eta^{(1)}||\eta^{(2)}|\one_{[t,T]}}_{\TT}.
    $$
\end{definition}

\begin{definition}
    We say~$U\in \Cc^{1,2}(\overline{\Lambda})\subset \Cc^0(\overline{\Lambda})$ if~$\partial_t U(t,\bm{\omega}),\,\langle \partial_{\bm{\omega}} U(t,\bm{\omega}),\eta^{(1)}\rangle$ and $\langle \partial_{\bm{\omega}\bm{\omega}} U(t,\bm{\omega}),(\eta^{(1)},\eta^{(2)})\rangle $ exist and are continuous with respect to~$(t,\bm{\omega})\in\overline{\Lambda}$ and the distance~$\bm{d}$ defined in~\eqref{eq:defs_spaces}, and for all~$\eta^{(1)},\eta^{(2)}\in\Ww$.
    Moreover, $U\in \Cc^{1,2}_+(\overline{\Lambda})\subset \Cc^{1,2}(\overline{\Lambda})$ if there exists~$G\in\Xx$ such that all derivatives of~$U$ have~$G$-growth.
\end{definition}

\begin{assumption}\label{assu:RegTimeDerivative}
    For $\varphi=b,\sigma$, for every~$s\in(t,T]$, with $t \in \TT$, $\partial_t\varphi(s,t,\cdot)$ exists  and there exist~$G\in\Xx$ and~$H\in(0,\half)$ such that, for all $x\in\RR^d$,
    $$
\abs{\varphi(s,t, x)} \lesssim G(x) (s-t)^{H-\half}
\quad\text{and}\quad
\abs{\partial_t\varphi(s, t, x)} \lesssim G(x) (s-t)^{H-\frac32}.
    $$
\end{assumption}
The following definition is a variation of \cite[Definition 3.4]{viens2019martingale} where we modify the estimates required on the second derivative, thereby enabling the use of the functional Itô formula in our model.
\begin{definition}\label{def:Cplusalpha}
    We say that $U\in \Cc^{1,2}_{+,\alpha}(\Lambda)$ with $\alpha\in(0,1)$ if there exist a continuous extension of~$U$ in~$\Cc^{1,2}_{+}(\overline\Lambda)$,
    a growth function~$G\in\Xx$,
    a modulus of continuity function~$\varrho\le G$ such that for any~$0\le t<T, \,\delta,\delta_1,\delta_2\in(0, T-t]$, $\eta,\eta_1,\eta_2\in\Ww$ with supports contained in~$[t,t+\delta],[t,t+\delta_1],[t,t+\delta_2]$ respectively,
    \begin{enumerate}
        \item[(i)] for any~$\bm{\omega}\in\overline\Ww$ such that~$\bm{\omega}\one_{[t,T]}\in\Ww_t$,
        \begin{align*}
            \abs{\langle \partial_{\bm{\omega}} U(t,\bm{\omega}),\eta\rangle} &\le \norm{G(\bm{\omega})}_{\TT} \norm{\eta}_{\TT} \delta^\alpha,\\
            \abs{\langle \partial_{\bm{\omega}\bm{\omega}} U(t,\bm{\omega}),(\eta_1,\eta_2)\rangle} &\le \norm{G(\bm{\omega})}_{\TT} \norm{\eta_1}_{\TT}\norm{\eta_2}_{\TT} \delta_1^\alpha\delta_2^{\alpha};
        \end{align*}
        \item[(ii)] for any other~$\bm{\omega}'\in\overline\Ww$ such that~$\bm{\omega}'\one_{[t,T]}\in\Ww_t$,
        \begin{align*}
            \abs{\langle \partial_{\bm{\omega}} (U(t,\bm{\omega})-U(t,\bm{\omega}')),\eta\rangle} &\le \big[\norm{G(\bm{\omega})}_{\TT}+\norm{G(\bm{\omega}')}_{\TT}\big] \norm{\eta}_{\TT} \varrho(\norm{\bm{\omega}-\bm{\omega}'}_{\TT}) \delta^\alpha,\\
            \abs{\langle \partial_{\bm{\omega}\bm{\omega}} (U(t,\bm{\omega})-U(t,\bm{\omega}')),(\eta_1,\eta_2)\rangle} &\le \big[\norm{G(\bm{\omega})}_{\TT}+\norm{G(\bm{\omega}')}_{\TT}\big]   \norm{\eta_1}_{\TT}\norm{\eta_2}_{\TT} \varrho(\norm{\bm{\omega}-\bm{\omega}'}_{\TT}) \delta_1^\alpha\delta_2^{\alpha} ;
        \end{align*}
        \item[(iii)] for any~$\bm{\omega}\in\overline\Ww$, $\langle \partial_{\bm{\omega}} U(t,\bm{\omega}),\eta\rangle$ and $\langle \partial_{\bm{\omega}\bm{\omega}} U(t,\bm{\omega}),(\eta_1,\eta_2)\rangle$ are continuous in~$t$.
    \end{enumerate}
\end{definition}
\begin{remark}
    In~\cite[Definition 3.4]{viens2019martingale}, the modulus of continuity is assumed bounded, which we relax by bounding the growth of~$\varphi(s,t,\cdot)$ by~$G$, which controls the growth of~$\varrho$.
\end{remark}
\begin{remark}
    The time continuity of item (iii) in the list is not necessary to apply It\^o's formula but we need it in the proof of Proposition~\ref{prop:ClassicalSol}.
\end{remark}
We intend to deal with singular coefficients that satisfy Assumption~\ref{assu:RegTimeDerivative} but may not be continuous on the diagonal. Hence, for~$\varphi=b,\sigma$ and~$\delta>0$, introduce the truncated functions
$$
\varphi^{\delta}(s,t,x):= \varphi(s\lor(t+\delta);t,x),
$$
and for $u\in \Cc^{1,2}_{+,\alpha}$ the spatial derivatives
\begin{align}\label{eq:SpatialDerivatives}
    \langle \partial_{\bm{\omega}} U(t,\bm{\omega}),\varphi^{t,\bm{\omega}}\rangle &:= \lim_{\delta\downarrow0} \langle \partial_{\bm{\omega}} U(t,\bm{\omega}),\varphi^{\delta,t,\bm{\omega}}\rangle, \qquad \text{for }
    \varphi\in\{b,\sigma\},\\ \nonumber
    \langle \partial_{\bm{\omega}\bm{\omega}} U(t,\bm{\omega}),(\sigma^{t,\bm{\omega}}\sigma^{t,\bm{\omega}})\rangle &:= \lim_{\delta\downarrow0} \langle \partial_{\bm{\omega}\bm{\omega}} U(t,\bm{\omega}),(\sigma^{\delta,t,\bm{\omega}}\sigma^{\delta,t,\bm{\omega}})\rangle,
\end{align}
where~$\varphi^{t,\bm{\omega}}(s):=\varphi(s;t,\bm{\omega}_t)$.
The proof of the functional It\^o formula from~\cite[Theorem~3.17]
{viens2019martingale} remains valid in the setting of Assumption~\ref{assu:RegTimeDerivative}.
\begin{theorem}\label{thm:Viens_Zhang}
Under Assumptions~\ref{assu:wellposedness}-\ref{assu:RegTimeDerivative}, let~$U\in \Cc^{1,2}_{+,\alpha}$, with~$\alpha>\half-H$, $H\in(0,\half)$. 
Then the spatial derivatives~\eqref{eq:SpatialDerivatives} exist and the following functional It\^o formula holds:
    \begin{align}
        \D U(t,\bm{X}\otimes_t \bm{\Theta}^t) =&\,  \partial_t U(t,\bm{X}\otimes_t \bm{\Theta}^t)\dt + \half \langle\partial_{\bm{\omega}\bm{\omega}} U(t,\bm{X}\otimes_t \bm{\Theta}^t),(\sigma^{t,\bm{X}},\sigma^{t,\bm{X}})\rangle \dt \nonumber\\
        &+ \langle \partial_{\bm{\omega}} U(t,\bm{X}\otimes_t \bm{\Theta}^t),b^{t,\bm{X}}\rangle \dt + \langle \partial_{\bm{\omega}} U(t,\bm{X}\otimes_t \bm{\Theta}^t), \sigma^{t,\bm{X}}\rangle \,\D \bm{W}_t. 
        \label{eq:Ito}
    \end{align}
\end{theorem}

\begin{proof}
The condition set in \cite[Definition 3.4]{viens2019martingale} for the second derivative features the estimate $\norm{|\eta_1||\eta_2}_{\TT}\delta^{2\alpha}$ for $\eta_1,\eta_2$ supported on the same interval~$[t,t+\delta]$. This assumption only seems to hold true for a Gaussian process, hence we replace this estimate with~$\norm{\eta_1}_{\TT}\norm{\eta_2}_{\TT} \delta_1^\alpha\delta_2^{\alpha}$ which is more aligned with the spirit of Fréchet derivatives. Below we provide a justification that the proof of the Itô formula in the singular case still goes through.

We refer to Step 1 of the proof (on pages 29-30, in \cite{viens2019martingale}) where the authors prove that \\
$(\langle \partial_{\bm{\omega}\bm{\omega}} U(t,\bm{\omega}),(\sigma^{\delta,t,\bm{\omega}},\sigma^{\delta,t,\bm{\omega}})\rangle)_{\delta\in(0,1)}$ is a Cauchy sequence. This is achieved by decomposing $\sigma^{\delta,t,\bm{\omega}}$ in order to apply the estimates of Definition~\ref{def:Cplusalpha} on small intervals. We recall the notations used in the proof. Let $\delta_n:=2^{-n}$ and $t_n:=t+\delta_n$ for all $n\in\NN$. Introduce a sequence of continuous functions $\psi_n$ with the property that~$\psi_n$ is supported on $(t_{n+1},t_{n-1})$ and $\psi_n+\psi_{n+1}=1$ on $[t_{n+1},t_n]$ (the exact construction can be found in Eq. (3.42) \cite{viens2019martingale}).
Let $\delta'<\delta$ and $\delta'\in[\delta_{n+1},\delta_n)$, while $\delta\in[\delta_{m+1},\delta_m)$ for some $m\le n$. 
This allows to decompose the directions via the following decomposition:
\begin{equation}\label{eq:decomposition1}
    \one_{[t,T]} = (1-\psi_m)\one_{[t_m,T]}+\sum_{k=m}^n\psi_k + (1-\psi_n)\one_{[t,t_n]}.
\end{equation}
We will also use this different decomposition where we replace~$m$ with $m_0=\min\{ m\in\NN : t_m\le T\}$:
\begin{equation}\label{eq:decomposition2}
    \one_{[t,T]} = (1-\psi_{m_0})\one_{[t_{m_0},T]}+\sum_{i=m_0}^n\psi_i + (1-\psi_n)\one_{[t,t_n]}.
\end{equation}
We notice that~$\sigma^{\delta,t,\bm{\omega}}-\sigma^{\delta',t,\bm{\omega}}=0$ over $[t_m,T]$ and we use the decompositions \eqref{eq:decomposition1} and~\eqref{eq:decomposition2} for the first and second directions, respectively:
\begin{align*}
    &\langle \partial_{\bm{\omega}\bm{\omega}} U(t,\bm{\omega}),(\sigma^{\delta,t,\bm{\omega}},\sigma^{\delta,t,\bm{\omega}})\rangle - \langle \partial_{\bm{\omega}\bm{\omega}} U(t,\bm{\omega}),(\sigma^{\delta',t,\bm{\omega}},\sigma^{\delta',t,\bm{\omega}})\rangle \\
    &=\Big\langle \partial_{\bm{\omega}\bm{\omega}} U(t,\bm{\omega}),\Big(\big(\sigma^{\delta,t,\bm{\omega}}-\sigma^{\delta',t,\bm{\omega}}\big) \one_{[t,T]},\big(\sigma^{\delta,t,\bm{\omega}}+\sigma^{\delta',t,\bm{\omega}}\big) \one_{[t,T]}\Big)\Big\rangle \\
    &=
    \sum_{k=m}^n \Big\langle \partial_{\bm{\omega}\bm{\omega}} U(t,\bm{\omega}), \Big(\psi_k(\sigma^{\delta,t,\bm{\omega}}-\sigma^{\delta',t,\bm{\omega}}),(1-\psi_{m_0})\one_{[t_{m_0},T]}(\sigma^{\delta,t,\bm{\omega}}+\sigma^{\delta',t,\bm{\omega})}) \Big)\Big\rangle\\
    &\quad+  
    \sum_{k=m}^n \sum_{i=m_0}^n \Big\langle \partial_{\bm{\omega}\bm{\omega}} U(t,\bm{\omega}), \Big((\psi_k(\sigma^{\delta,t,\bm{\omega}}-\sigma^{\delta',t,\bm{\omega}}),\psi_i(\sigma^{\delta,t,\bm{\omega}}+\sigma^{\delta',t,\bm{\omega})}) \Big)\Big\rangle\\
    &\quad + \sum_{k=m}^n \Big\langle \partial_{\bm{\omega}\bm{\omega}} U(t,\bm{\omega}), \Big(\psi_k(\sigma^{\delta,t,\bm{\omega}}-\sigma^{\delta',t,\bm{\omega}}),(1-\psi_n)\one_{[t,t_n]}(\sigma^{\delta,t,\bm{\omega}}+\sigma^{\delta',t,\bm{\omega})}) \Big)\Big\rangle\\
    &\quad+ \Big\langle \partial_{\bm{\omega}\bm{\omega}} U(t,\bm{\omega}), \Big((1-\psi_n)\one_{[t,t_n]} (\sigma^{\delta,t,\bm{\omega}}-\sigma^{\delta',t,\bm{\omega}}),(1-\psi_{m_0})\one_{[t_{m_0},T]}(\sigma^{\delta,t,\bm{\omega}}+\sigma^{\delta',t,\bm{\omega})}) \Big)\Big\rangle\\
    &\quad+ \sum_{i=m_0}^n \Big\langle \partial_{\bm{\omega}\bm{\omega}} U(t,\bm{\omega}), \Big((1-\psi_n)\one_{[t,t_n]} (\sigma^{\delta,t,\bm{\omega}}-\sigma^{\delta',t,\bm{\omega}}),\psi_i (\sigma^{\delta,t,\bm{\omega}}+\sigma^{\delta',t,\bm{\omega})}) \Big)\Big\rangle\\
    &\quad+ \Big\langle \partial_{\bm{\omega}\bm{\omega}} U(t,\bm{\omega}), \Big((1-\psi_n)\one_{[t,t_n]} (\sigma^{\delta,t,\bm{\omega}}-\sigma^{\delta',t,\bm{\omega}}),(1-\psi_n)\one_{[t,t_n]}(\sigma^{\delta,t,\bm{\omega}}+\sigma^{\delta',t,\bm{\omega})}) \Big)\Big\rangle\\
    &=: J_1 + J_2 + J_3 + J_4 + J_5 + J_6.
\end{align*}
Notice that, since $\delta\in[\delta_{m+1},\delta_m)$ and $\delta'\in[\delta_{n+1},\delta_n)$, Assumption \ref{assu:RegTimeDerivative} entails
\begin{equation}
    \sup_{s\in[t_{k+1},t_{k-1}]} \abs{\sigma^{\delta,t,\bm{\omega}}} \lesssim 
    \left\{ 
    \begin{aligned}
        &\norm{G(\bm{\omega})}_{\TT} \delta_{m+1}^{H-1/2}, \quad k\ge m\\
        &\norm{G(\bm{\omega})}_{\TT} \delta_{k+1}^{H-1/2}, \quad k< m
    \end{aligned}
    \right.\; ;
    \quad \sup_{s\in[t_{k+1},t_{k-1}]} \abs{\sigma^{\delta',t,\bm{\omega}}} \lesssim 
    \left\{ 
    \begin{aligned}
        &\norm{G(\bm{\omega})}_{\TT} \delta_{n+1}^{H-1/2}, \quad k\ge n\\
        &\norm{G(\bm{\omega})}_{\TT} \delta_{k+1}^{H-1/2}, \quad k< n
    \end{aligned}
    \right.
\end{equation}
By the same principle we have
\begin{align*}
    &\sup_{s\in[t,t_n]} \Big(\abs{\sigma^{\delta,t,\bm{\omega}}} + \abs{\sigma^{\delta',t,\bm{\omega}}}\Big) 
    \lesssim \norm{G(\bm{\omega})}_{\TT} \big( \delta_{m+1}^{H-1/2} + \delta_{n+1}^{H-1/2} \big)\le 2\norm{G(\bm{\omega})}_{\TT} \delta_{n+1}^{H-1/2}, \\
    &\sup_{s\in[t_{m_0},T]} \Big(\abs{\sigma^{\delta,t,\bm{\omega}}} + \abs{\sigma^{\delta',t,\bm{\omega}}}\Big) 
    \lesssim 2\norm{G(\bm{\omega})}_{\TT} \delta_{m_0}^{H-1/2}.
\end{align*}
Since $\psi_k$ is supported on $(t_{k+1},t_{k-1})$ we have by Definition~\ref{def:Cplusalpha}(i) and Assumption~\ref{assu:RegTimeDerivative} that, for any~$k,i\in\{m,\cdots,n\}$
\begin{align}
    \big\lvert \big\langle&\partial_{\bm{\omega}\bm{\omega}} U(t,\bm{\omega}),
    \big(\psi_k \big(\sigma^{\delta,t,\bm{\omega}}-\sigma^{\delta',t,\bm{\omega}}\big),\psi_i\big(\sigma^{\delta,t,\bm{\omega}}+\sigma^{\delta',t,\bm{\omega}}\big)\big)\big\rangle\big\lvert\\
    &\le  \norm{G(\bm{\omega})}_{\TT} \left(\sup_{s\in[t,T]} \psi_k(s)\abs{\sigma^{\delta,t,\bm{\omega}}(s)-\sigma^{\delta',t,\bm{\omega}}(s)}\right) \left(\sup_{s\in[t,T]} \psi_i(s)\abs{\sigma^{\delta,t,\bm{\omega}}(s)+\sigma^{\delta',t,\bm{\omega}}(s)}\right) \delta^{\alpha}_{k-1}\delta^{\alpha}_{i-1} \\
    &\le  \norm{G(\bm{\omega})}_{\TT} \left(\sup_{s\in[t_{k+1},t_{k-1}]} \abs{\sigma^{\delta,t,\bm{\omega}}(s)}+\abs{\sigma^{\delta',t,\bm{\omega}}(s)}\right) \left(\sup_{s\in[t_{i+1},t_{i-1}]} \abs{\sigma^{\delta,t,\bm{\omega}}(s)}+\abs{\sigma^{\delta',t,\bm{\omega}}(s)}\right) \delta^{\alpha}_{k-1}\delta^{\alpha}_{i-1}  \\
    &\le 4\norm{G(\bm{\omega})}_{\TT}   \delta^{H-\half}_{k+1} \delta^{H-\half}_{i+1} \delta^{\alpha}_{k-1}\delta^{\alpha}_{i-1}\\
    &=  \norm{G(\bm{\omega})}_{\TT}   2^{3-2H+2\alpha} 2^{-k(\alpha+H-\half)} 2^{-i(\alpha+H-\half)}.
\end{align} 
Recall that $\beta=\alpha+H-1/2>0$, thus the double sum now reads:
\begin{align*}
    \abs{J_2}\lesssim 
    \sum_{k=m}^\infty 2^{-k(\alpha+H-1/2)} \sum_{i=m_0}^\infty 2^{-i(\alpha+H-1/2)} =\frac{2^{-\beta(m+1)}}{1-2^{-\beta}}\frac{2^{-\beta(m_0+1)}}{1-2^{-\beta}}\lesssim \delta^{\beta}_m.
\end{align*}
This tends to zero as $\delta$ goes to zero (equivalently as $m\to\infty$). Similarly,
\begin{align*}
    \abs{J_1} 
    &\le  \norm{G(\bm{\omega})}_{\TT} \sum_{k=m}^n \left(\sup_{s\in[t_{k+1},t_{k-1}]} \abs{\sigma^{\delta,t,\bm{\omega}}(s)}+\abs{\sigma^{\delta',t,\bm{\omega}}(s)}\right) \left(\sup_{s\in[t_{m_0},T]} \abs{\sigma^{\delta,t,\bm{\omega}}(s)}+\abs{\sigma^{\delta',t,\bm{\omega}}(s)}\right) \delta^{\alpha}_{k-1} \\
    &\le \norm{G(\bm{\omega})}_{\TT} \sum_{k=m}^n  \delta^{H-\half}_{k+1} \delta^{H-\half}_{m_0} \delta^{\alpha}_{k-1} \lesssim \delta_m^\beta\\
    \abs{J_3} &\le \norm{G(\bm{\omega})}_{\TT}  \sum_{k=m}^n \delta^{H-\half}_{k+1}  \delta^{\alpha}_{k-1} \delta^{H-\half}_{n+1}\delta^{\alpha}_{n-1} \lesssim \delta_m^\beta  \delta_n^\beta \\
    \abs{J_4} &\le \norm{G(\bm{\omega})}_{\TT} \delta^{H-\half}_{n+1}\delta^{\alpha}_{n-1}  \delta^{H-\half}_{m_0}  \lesssim \delta_n^\beta \\
    \abs{J_5} &\le \norm{G(\bm{\omega})}_{\TT} \delta^{H-\half}_{n+1}\delta^{\alpha}_{n-1}  \sum_{i=m_0}^n \delta^{H-\half}_{i+1}  \delta^{\alpha}_{i-1}  \lesssim \delta_n^\beta \\
    \abs{J_6} &\le \norm{G(\bm{\omega})}_{\TT}( \delta^{H-\half}_{n+1}\delta^{\alpha}_{n-1}  )^2 \lesssim \delta_n^{2\beta}.
\end{align*}
This concludes the proof of Step 1 with the modified assumption because we confirmed that \\
$(\langle \partial_{\bm{\omega}\bm{\omega}} U(t,\bm{\omega}),(\sigma^{\delta,t,\bm{\omega}},\sigma^{\delta,t,\bm{\omega}})\rangle)_{\delta>0}$ is a Cauchy sequence. Moreover, by taking $\delta'\to0$ or equivalently $n\to\infty$, we have the estimate
\begin{align*}
    \abs{\langle \partial_{\bm{\omega}\bm{\omega}} U(t,\bm{\omega}),(\sigma^{\delta,t,\bm{\omega}},\sigma^{\delta,t,\bm{\omega}})\rangle - \langle \partial_{\bm{\omega}\bm{\omega}} U(t,\bm{\omega}),(\sigma^{t,\bm{\omega}},\sigma^{t,\bm{\omega}})\rangle}
    \lesssim \norm{G(\bm{\omega})}_{\TT} \delta^{\beta}.
\end{align*}
This is slightly weaker than Eq. (3.46) in \cite{viens2019martingale} but is sufficient to conclude Step 3 of the proof since it does vanish as~$\delta\to0$.
\end{proof}

\begin{remark}\label{rem:path_der_multidim}
    One can show that if $\bm{\omega}$ is $\RR^d$-valued, then for $\eta$ (living in $\RR^d$ as well),
    \begin{align*}
    \langle \partial_{\bm{\omega}} U(t,\bm{\omega}),\eta\rangle 
    = \sum_{i=1}^d \langle \partial_{\bm{\omega}_i} U(t,\bm{\omega}),\eta_i\rangle  
    \quad \text{and} \quad
    \langle \partial_{\bm{\omega} \bm{\omega}}U(t,\bm{\omega}),(\eta,\eta)\rangle 
    = \sum_{i,j=1}^d \langle \partial_{\bm{\omega}_i\bm{\omega}_j} U(t,\bm{\omega}),(\eta_i,\eta_j)\rangle.
    \end{align*}
    Furthermore, if $\boldsymbol{\eta}$ is matrix valued, say in~$\RR^{d\times m}$, { extending the $\RR^{d}$-valued directions space we have restricted ourselves so far}, then
    \begin{align*}
    \langle \partial_{\bm{\omega}} { U}(t,\bm{\omega}),\boldsymbol{\eta}\rangle  
    & = 
    \left(\sum_{i=1}^d\langle \partial_{\bm{\omega}_i} { U}(t,\bm{\omega}),\boldsymbol{\eta}_{i,j}\rangle\right)_{j=1}^m, \\
    \langle \partial_{\bm{\omega} \bm{\omega}}{ U}(t,\bm{\omega}),(\boldsymbol{\eta},\boldsymbol{\eta})\rangle 
    & = \sum_{k=1}^m  \sum_{i,j=1}^d \langle \partial_{\bm{\omega}_i\bm{\omega}_j} { U}(t,\bm{\omega}),(\boldsymbol{\eta}_{j,k},\boldsymbol{\eta}_{i,k})\rangle.
    \end{align*}
    In finite dimensions, these correspond to the gradient and trace of the Hessian, respectively.
\end{remark}


\subsection{A path-dependent PDE} \label{sec:GenPPDE}

As in the Markovian case, this It\^o formula leads to natural connections with PDEs. In this section, we rigorously prove that conditional expectations (equivalently, solutions to a BSDE) are solutions to path-dependent PDEs. This was predicted by~\cite[Section~4.2]{viens2019martingale} where they also showed the reverse implication, known as Feynman-Kac formula. 
The crucial aspect is to show time differentiability, for which we follow~\cite[Theorem 3.2]{wang2022path}, where a similar result is presented for regular kernels. This requires to introduce new variables and notations.
Define the forward-backward SVE for~$\ww\in\Ww$ and { all~$0 \leq t, s\le T$:  
    \begin{align}\label{eq:XsxEq}
        \bm{X}_s^{t,\ww} 
        &= 
        \begin{cases}
            \ww_s, & s\leq t, \\
            \ww_s + \int_t^s b(s,r,\bm{X}_r^{t,\ww}) \dr + \int_t^s \sigma(s,r,\bm{X}_r^{t,\ww}) \,\D \bm{W}_r, & t\leq s,
        \end{cases}
        \\
        \bm{Y}_s^{t,\ww} &= \Phi(\bm{X}^{t,\ww}_\cdot) +\int_s^T f(r,\bm{X}^{t,\ww}_r,\bm{Y}_r^{t,\ww},\bm{Z}_r^{t,\ww}) \dr- \int_s^T \bm{Z}^{t,\ww}_r \,\D \bm{W}_r.
         \label{eq:BSDE}
    \end{align}}
    The forward process~$\bm{X}$ lives in~$\RR^d$ and the backward one~$\bm{Y}$ in~$\RR^{d'}$ with~$d,d'\in\NN$.
    Moreover~$\Phi$ maps~$\Ww$ to~$\RR^{d'}$ and $f:\TT\times\RR^d\times\RR^{d'}\times\RR^{d'\times d}\to\RR^{d'}$ is measurable in all variables.
    Under the following assumptions, the backward SDE has a unique square integrable solution~\cite[Theorem 4.3.1]{zhang2017backward}.
\begin{assumption}\label{assu:BSDE}\
\begin{enumerate}
    \item[(i)] $\EE[\Phi(\bm{X}^{t,\ww})^2]<\infty$, for all~$\ww\in\Ww$ and all~$0\le t\le s\le T$;
    \item[(ii)] $f$ is uniformly Lipschitz continuous in $(y,z)$ and uniformly continuous in~$t$ and $x$;
    \item[(iii)] $\EE\left[\left(\int_0^T \abs{f(r,\bm{X}^{t,\ww}_r,0,0)} \dr \right)^2\right]$ is finite.
\end{enumerate}
\end{assumption}
Condition~(iii) is satisfied if, for instance, there exist~$G\in\Xx$ and~$g\in L^1(\TT)$ such that~$f(r,\ww,0,0)\le g(r) \sqrt{\norm{G(\ww)}_{\TT}}$.
We now show that
\begin{align}\label{eq:def_U}
    U(t,\ww):=\bm{Y}^{t,\ww}_t
\end{align}
satisfies a semi-linear path-dependent PDE.

\begin{assumption}\label{assumption:regcoef+U}\
\begin{enumerate}
    \item[(i)] There is a modulus of continuity~$\varrho'$ and~$G\in\Xx$ such that 
\begin{align*}
    \abs{b(s,r,x)-b(s,r,y)}+\abs{\sigma(s,r,x)-\sigma(s,r,y)} \le (s-r)^{H-1/2} \varrho'(\abs{x-y}) (G(x)+G(y)).
\end{align*} 
    \item[(ii)] Moreover, for any $t\in\TT$, $\ww,\ww'\in \overline{\Ww}$ such that~$\ww,\ww',\eta,\eta'\in\Ww_t$, it holds
\begin{align*}
    &\abs{\big\langle\partial_{\ww} U(t,\ww),\eta\big\rangle} \le \norm{G(\ww)}_{\TT} \norm{\eta}_{L^2(\TT)};\\
    &\abs{\big\langle\partial_{\ww\ww} U(t,\ww),(\eta,\eta')\big\rangle} \le \norm{G(\ww)}_{\TT} \norm{\eta}_{L^2(\TT)}\norm{\eta'}_{L^2(\TT)};\\
    &\abs{\big\langle\partial_{\ww} U(t,\ww),\eta\big\rangle - \big\langle\partial_{\ww} U(t,\ww'),\eta\big\rangle} \le \big[\norm{G(\omega)}_{\TT}+\norm{G(\omega')}_{\TT}\big] \varrho'(\norm{\ww-\ww'}_{\TT} )\norm{\eta}_{L^2(\TT)};\\
    &\abs{\big\langle\partial_{\ww\ww} U(t,\ww),(\eta,\eta)\big\rangle - \big\langle\partial_{\ww\ww} U(t,\ww'),(\eta,\eta)\big\rangle} \le \big[\norm{G(\omega)}_{\TT}+\norm{G(\omega')}_{\TT}\big] \varrho'(\norm{\ww-\ww'}_{\TT} )\norm{\eta}_{L^2(\TT)}^2,
\end{align*}
for some growth function $G\in\Xx$ and a concave modulus of continuity~$\varrho\le G$.
\end{enumerate}
\end{assumption}

\begin{proposition}\label{prop:ClassicalSol}
Let Assumptions~\ref{assu:wellposedness},~\ref{assu:RegTimeDerivative}, \ref{assu:BSDE} and~\ref{assumption:regcoef+U} hold. 
Moreover, assume~the SVE~\eqref{eq:XsxEq} has a unique strong solution and~$U\in \Cc^{0,2}_{+,\alpha}$, with~$\alpha>\half-H$ and~$H\in(0,\half)$. 
Then the time derivative of~$U${ , as defined in~\eqref{eq:def_U},} exists and~$U$ is a classical solution to the path-dependent PDE
\begin{equation}\label{Eq:PPDE_u}
   \partial_t U(t,\bm{\omega}) + \half \langle\partial_{\bm{\omega}\bm{\omega}} U(t,\bm{\omega}),(\sigma^{t,\bm{\omega}},\sigma^{t,\bm{\omega}})\rangle + \langle \partial_{\bm{\omega}} U(t,\bm{\omega}),b^{t,\bm{\omega}}\rangle + f\big(t, \bm{\omega}_r,U(t,\bm{\omega}),\langle \partial_{\bm{\omega}} U(t,\bm{\omega}),\sigma^{t,\bm{\omega}}\rangle\big) = 0, 
\end{equation}
with boundary condition $U(T,\bm{\omega}) = \Phi(\bm{\omega})$.
\end{proposition}
\begin{proof}
     We follow the proof of~\cite[Theorem 3.2]{wang2022path} but drop one argument (the first one in~$U$) which makes it slightly simpler. 
     
     \textbf{Step 1.} For any~$0 \le t\le r\le s\le T$ and any $\ww\in \Ww$, define 
     \begin{equation}\label{eq:Xtilde}
\widetilde{\bm{X}}^{t,\ww}_{r,s} := \ww_s + \int_t^r b(s,r',\bm{X}^{t,\ww}_{r'}) \D r' + \int_t^r \sigma(s,r',\bm{X}^{t,\ww}_{r'}) \D \bm{W}_{r'}.
     \end{equation}
     Hence, by strong uniqueness of solutions for~\eqref{eq:XsxEq}, we have
$$
    \bm{X}^{t,\ww}_s 
    = \bm{X}_s^{r,\widehat{\bm{X}}^{t,\ww,r}}, \qquad 
    \text{where   } 
    \widehat{\bm{X}}^{t,\ww,r}_{s} 
    := \big(\bm{X}^{t,\ww}_\cdot\otimes_r \widetilde{\bm{X}}^{t,\ww}_{r{ ,\cdot}}\big)_{s}
    = \bm{X}^{t,\ww}_{s} \one_{s\le r} + \widetilde{\bm{X}}^{t,\ww}_{r,s} \one_{s>   r}.
    $$
By uniqueness of the solution to the BSDE, then $\bm{Y}^{t,\ww}_r = \bm{Y}^{r,\widehat{\bm{X}}^{t,\ww,r}}_r= U(r,\widehat{\bm{X}}^{t,\ww,r})$. 

\textbf{Step 2.} Fix~$n\in\NN$, $\delta=\frac{T-t}{n}$, $t_i=t+i\delta$ for all~$i\le n$. For~
 {$s\in[t_i,t_{i+1}]$ and $s\ge t$}, let
$$
\bm{Y}^n_s := U(t_{i+1}, \widehat{\bm{X}}^{t,\ww,s}), \qquad \bm{Z}^n_s:=\langle \partial_{\bm{\omega}} U(t_{i+1}, \widehat{\bm{X}}^{t,\ww,s}), \sigma^{s,\widehat{\bm{X}}^{t,\ww,s}} \rangle.
$$
Thanks to our assumptions we apply the functional It\^o formula to~$U(t_{i+1},\cdot)$ for $s\in(t_i,t_{i+1}]$:
\begin{align}\label{eq:Ynr}
&\bm{Y}^n_{t_{i+1}} - \bm{Y}^n_s = \int_s^{t_{i+1}}\bm{Z}^n_r \,\D \bm{W}_r \\
&+\int_s^{t_{i+1}} \bigg[ \half \langle \partial_{\bm{\omega}\bm{\omega}} U(t_{i+1}, \widehat{\bm{X}}^{t,\ww,r}), (\sigma^{r,\widehat{\bm{X}}^{t,\ww,r}}, \sigma^{r,\widehat{\bm{X}}^{t,\ww,r}})\rangle + \langle \partial_{\bm{\omega}} U(t_{i+1}, \widehat{\bm{X}}^{t,\ww,r}), b^{r,\widehat{\bm{X}}^{t,\ww,r}} \rangle \bigg] \dr.
\end{align}
We recall that~$(\bm{Y}^{t,\ww},\bm{Z}^{t,\ww})$ solves, for $s\in(t_i,t_{i+1}]$,
\begin{align*}
    \bm{Y}^{t,\ww}_{t_{i+1}} -\bm{Y}^{t,\ww}_s = - \int_s^{t_{i+1}}  f(r,\widehat{\bm{X}}^{t,\ww,r}_{r},\bm{Y}_r^{t,\ww},\bm{Z}_r^{t,\ww}) \dr + \int_s^{t_{i+1}}\bm{Z}^{t,\ww}_r \,\D \bm{W}_r,
\end{align*}
where we used $\widehat{\bm{X}}^{t,\ww,r}_{r}=\bm{X}^{t,\ww}_{r}$. Further define~$\Delta \bm{Y}^n_r := \bm{Y}^n_r - \bm{Y}^{t,\ww}_r$ and~$\Delta \bm{Z}^n_r := \bm{Z}^n_r - \bm{Z}^{t,\ww}_r$. 
Upon recognising that~$\Delta\bm{Y}^{n}_{t_{i+1}}=\bm{Y}^{n}_{t_{i+1}}-\bm{Y}^{t,\ww}_{t_{i+1}}=0$,  we obtain that $(\Delta \bm{Y}^n,\Delta \bm{Z}^n)$ satisfies a new BSDE, for $s\in(t_i,t_{i+1}]$,
\begin{equation}\label{eq:DeltaYnr}
\Delta \bm{Y}^n_s = - \int_s^{t_{i+1}}\widehat{f}\left(r,t_{i+1},\widehat{\bm{X}}^{t,\ww,r},\bm{Y}^n_r-\Delta \bm{Y}^n_r,\bm{Z}^n_r-\Delta \bm{Z}^n_r\right) \dr - \int_s^{t_{i+1}}\Delta \bm{Z}^n_r \,\D \bm{W}_r,
\end{equation}
where, for $r\in(t_i,t_{i+1}]$,
$$
\widehat{f}(r,t_{i+1},\bm{x},y,z) := \half \langle \partial_{\bm{\omega}\bm{\omega}}^2 U(t_{i+1},\bm{x}), (\sigma^{r,\bm{x}}, \sigma^{r,\bm{x}})\rangle + \langle \partial_{\bm{\omega}} U(t_{i+1}, \bm{x}), b^{r,\bm{x}} \rangle + f(r,\bm{x}_r, y,z).
$$
We use standard BSDE estimates~\cite[Theorem 4.2.1]{zhang2017backward}, $\Delta\bm{Y}^n_{t_{i+1}}=0$, the uniform Lipschitz continuity of~$(y,z)\mapsto f(r,x,y,z)$ and Jensen's inequality to get 
\begin{align}\label{eq:BSDE_estimate}
&\EE\left[\sup_{t_i< s\le t_{i+1}} \abs{\Delta \bm{Y}_s^n}^2 + \int_{t_i}^{t_{i+1}} \abs{\Delta \bm{Z}_r^n}^2\dr\right]
\lesssim \EE\left[\left(\int_{t_i}^{t_{i+1}} \abs{\widehat{f}(r,t_{i+1},\widehat{\bm{X}}^{t,\ww,r},\bm{Y}^{t,\ww}_r,\bm{Z}^{t,\ww}_r)}\dr\right)^2\right] \nonumber\\
&\qquad\qquad \qquad \lesssim \delta\EE\left[ \int_{t_i}^{t_{i+1}} \abs{\widehat{f}(r,t_{i+1},\widehat{\bm{X}}^{t,\ww,r},0,0)}^2+\abs{\bm{Y}^{t,\ww}_r}^2+|\bm{Z}^{t,\ww}_r|^2\dr\right]. 
\end{align}
BSDE estimates combined with Assumption~\ref{assu:BSDE}(i) and $G$-growth of the pathwise derivatives entails that
$$
\EE\left[ \int_{t}^T \abs{\widehat{f}(r,\frac{\lceil rN\rceil}{N},\widehat{\bm{X}}^{t,\ww,r},0,0)}^2+\abs{\bm{Y}^{t,\ww}_r}^2+|\bm{Z}^{t,\ww}_r|^2\dr\right]<\infty.
$$

We combine the previous estimates to conclude this step:
\begin{align*}
    &\EE\left[\sum_{i=1}^{n-1}\int_{t_i}^{t_{i+1}} \abs{\bm{Z}^{t,\ww}_r -\bm{Z}^n_r}^2 \dr \right] 
    = \EE\left[\sum_{i=1}^{n-1}\int_{t_i}^{t_{i+1}} \abs{\Delta \bm{Z}_r^n}^2\dr \right]\\
    & \le \delta\,\EE\left[\int_{t}^{T} \abs{\widehat{f}\left(r,\frac{\lceil rN \rceil}{N},\widehat{\bm{X}}^{t,\ww,r}_r,0,0\right)}^2+\abs{\bm{Y}^n_r}^2+|\bm{Z}^n_r|^2\dr\right],
\end{align*}
which goes to zero as~$\delta\downarrow0$ and therefore implies~$\bm{Z}^{t,\ww}_r =\bm{Z}^n_r=\big\langle \partial_{\bm{\omega}} U(t+\delta, \widehat{\bm{X}}^{t,\ww,r}), \sigma^{r,\widehat{\bm{X}}^{t,\ww,r}} \big\rangle$.

\textbf{Step 3.}
In particular, recalling that~$t_0=t,t_1=t+\delta$ and~$\Delta \bm{Y}^n_{t_{i+1}}=0$, we have
\begin{align}\label{eq_incr_U}
    U(t+\delta,\ww)-U(t,\ww)
      &= -(\Delta \bm{Y}^n_{t+\delta}- \Delta \bm{Y}^n_{t})\nonumber\\
     & = -\EE\left[\int_{t}^{t+\delta} \widehat{f}(r,t+\delta,\widehat{\bm{X}}^{t,\ww,r},\bm{Y}^{t,\ww}_r,\bm{Z}^{t,\ww}_r) \D r\right]\\
     &= - \EE\left[\delta \widehat{f}(t,t+\delta,\widehat{\bm{X}}^{t,\ww,t},\bm{Y}^{t,\ww}_t,\bm{Z}^{t,\ww}_t)\right] - R(\delta),
\end{align}
where we took expectations while noting that the left-hand-side is deterministic, it holds
\begin{align*}
    \widehat{\bm{X}}^{t,\ww,t}=\ww,\quad\bm{Y}^{t,\ww}_t=U(t,\ww),\quad \bm{Z}^{t,\ww}_t=\big\langle \partial_{\ww} U(t,\ww), \sigma^{t,\ww}\big\rangle,
\end{align*}
and where
\begin{align*}
     R(\delta):&=\EE\left[\int_{t}^{t+\delta} \Big(\widehat{f}(r,t+\delta,\widehat{\bm{X}}^{t,\ww,r},\bm{Y}^{t,\ww}_r,\bm{Z}^{t,\ww}_r) -\widehat{f}(t,t+\delta,\widehat{\bm{X}}^{t,\ww,t},\bm{Y}^{t,\ww}_t,\bm{Z}^{t,\ww}_t)\Big)\D r\right] \\
     &= \half \EE\bigg[\int_{t}^{t+\delta} \Big( \big\langle \partial_{\omega\omega}U(t+\delta,\widehat{\bm{X}}^{t,\ww,r}),(\sigma^{r,\widehat{\bm{X}}^{t,\ww,r}},\sigma^{r,\widehat{\bm{X}}^{t,\ww,r}}) \big\rangle \\
     &\qquad \qquad
     - \big\langle \partial_{\omega\omega}U(t+\delta,\widehat{\bm{X}}^{t,\ww,t}),(\sigma^{t,\widehat{\bm{X}}^{t,\ww,t}},\sigma^{t,\widehat{\bm{X}}^{t,\ww,t}}) \big\rangle
     \Big)\D r\bigg] \\
     &\quad + \EE\left[\int_{t}^{t+\delta} \Big( \big\langle \partial_{\omega}U(t+\delta,\widehat{\bm{X}}^{t,\ww,r}),b^{r,\widehat{\bm{X}}^{t,\ww,r}} \big\rangle
     - \big\langle \partial_{\omega}U(t+\delta,\widehat{\bm{X}}^{t,\ww,t}),b^{t,\widehat{\bm{X}}^{t,\ww,t}} \big\rangle
     \Big)\D r\right] \\
     &\quad + \EE\left[\int_{t}^{t+\delta} \Big(f(r,\widehat{\bm{X}}^{t,\ww,r}_r,\bm{Y}^{t,\ww}_r,\bm{Z}^{t,\ww}_r) - f(t,\widehat{\bm{X}}^{t,\ww,t}_t,\bm{Y}^{t,\ww}_t,\bm{Z}^{t,\ww}_t)\Big)\D r\right] \\
     &=: R_1(\delta)+R_2(\delta)+R_3(\delta).
\end{align*}

First notice that, since~$\widehat{\bm{X}}^{t,\ww,r}_r=\bm{X}^{t,\ww}_r$ and $\sigma^{r,\ww}(s)=\sigma(s,r,\ww_r)$, it entails the simplifications~$\sigma^{r,\widehat{\bm{X}}^{t,\ww,r}}=\sigma^{r,\bm{X}^{t,\ww}}$ and $\sigma^{t,\widehat{\bm{X}}^{t,\ww,t}}=\sigma^{t,\bm{X}^{t,\ww}}=\sigma^{t,\ww}$, and similarly for~$b$. We make use of Assumption~\ref{assumption:regcoef+U}. With these in hand, let us look at $R_1(\delta)$,

\begin{align*}
    &\abs{ \big\langle \partial_{\omega\omega}U(t+\delta,\widehat{\bm{X}}^{t,\ww,r}),(\sigma^{r,\bm{X}^{t,\ww}},\sigma^{r,\bm{X}^{t,\ww}}) \big\rangle 
     - \big\langle \partial_{\omega\omega}U(t+\delta,\ww),(\sigma^{t,\ww},\sigma^{t,\ww}) \big\rangle} \nonumber\\
     &= \lim_{\ep\downarrow0} \abs{ \big\langle \partial_{\omega\omega}U(t+\delta,\widehat{\bm{X}}^{t,\ww,r}),(\sigma^{\ep,r,\bm{X}^{t,\ww}},\sigma^{\ep,r,\bm{X}^{t,\ww}}) \big\rangle 
     - \big\langle \partial_{\omega\omega}U(t+\delta,\ww),(\sigma^{\ep,t,\ww},\sigma^{\ep,t,\ww}) \big\rangle} \nonumber\\
     &\le \lim_{\ep\downarrow0} \abs{ \big\langle \partial_{\omega\omega}U(t+\delta,\widehat{\bm{X}}^{t,\ww,r}),(\sigma^{\ep,r,\bm{X}^{t,\ww}},\sigma^{\ep,r,\bm{X}^{t,\ww}}) \big\rangle 
     - \big\langle \partial_{\omega\omega}U(t+\delta,\widehat{\bm{X}}^{t,\ww,r}),(\sigma^{\ep,r,\ww},\sigma^{\ep,r,\ww}) \big\rangle} \\ 
     &\; + \lim_{\ep\downarrow0} \abs{ \big\langle \partial_{\omega\omega}U(t+\delta,\widehat{\bm{X}}^{t,\ww,r}),(\sigma^{\ep,r,\ww},\sigma^{\ep,r,\ww}) \big\rangle
     - \big\langle \partial_{\omega\omega}U(t+\delta,\widehat{\bm{X}}^{t,\ww,r}),(\sigma^{\ep,t,\ww},\sigma^{\ep,t,\ww}) \big\rangle} \nonumber\\
     &\; + \lim_{\ep\downarrow0} \abs{ \big\langle \partial_{\omega\omega}U(t+\delta,\widehat{\bm{X}}^{t,\ww,r}),(\sigma^{\ep,t,\ww},\sigma^{\ep,t,\ww}) \big\rangle
     - \big\langle \partial_{\omega\omega}U(t+\delta,\ww),(\sigma^{\ep,t,\ww},\sigma^{\ep,t,\ww}) \big\rangle} \nonumber\\
     &\le \lim_{\ep\downarrow0}\bigg(
     \big\lVert G(\widehat{\bm{X}}^{t,\ww,r})\big\lVert_{\TT}
    \big\lVert \sigma^{\ep,r,\bm{X}^{t,\ww,r}} - \sigma^{\ep,r,\ww} \big\lVert_{L^2(\TT)} \big\lVert \sigma^{\ep,r,\bm{X}^{t,\ww,r}} + \sigma^{\ep,r,\ww} \big\lVert_{L^2(\TT)} \nonumber\\
    &\quad+ G(\widehat{\bm{X}}^{t,\ww,r})\big\lVert_{\TT}
    \big\lVert \sigma^{\ep,r,\ww} - \sigma^{\ep,t,\ww} \big\lVert_{L^2(\TT)} \big\lVert \sigma^{\ep,r,\ww} + \sigma^{\ep,t,\ww} \big\lVert_{L^2(\TT)} \nonumber\\
    &\quad+ \varrho\Big(\big\lVert\widehat{\bm{X}}^{t,\ww,r}-\ww\big\lVert_{\TT}\Big)  \big\lVert\sigma^{\ep,t,\ww}\big\lVert_{L^2(\TT)}^2 \bigg). \nonumber
\end{align*}
Routine arguments combined with Assumption~\ref{assu:RegTimeDerivative} and Jensen's inequality show that
\begin{align*}
    \EE\Big[\varrho\big(\big\lVert \widehat{\bm{X}}^{t,\ww,r}-\ww\big\lVert_{\TT}\big)\Big]\le \varrho\Big(\EE\Big[\big\lVert \widehat{\bm{X}}^{t,\ww,r}-\ww\big\lVert_{\TT}^2\Big]^\half \Big)\lesssim \varrho\big( (r-t)^{H}\big).
\end{align*}
Meanwhile, Assumption~\ref{assu:RegTimeDerivative} also yields~$\abs{\sigma^{\ep,t,\ww}(s)} \le G(\ww_t) (s-t)^{H-\half}$ for a possibly different $G\in\Xx$, and hence~$\norm{\sigma^{\ep,t,\ww}}_{L^2(\TT)}^2\le G(\ww_t) T^{2H}/(2H)$. Furthermore, dominated convergence and Assumption~\ref{assu:RegTimeDerivative} lead to
\begin{align*}
    &\lim_{\ep\downarrow0} \big\lVert  \sigma^{\ep,r,\ww} - \sigma^{\ep,t,\ww} \big\lVert_{L^2(\TT)} 
    = \big\lVert \sigma^{r,\ww} - \sigma^{t,\ww} \big\lVert_{L^2(\TT)} \\
    &= \int_0^T \abs{\sigma(s,r,\ww_r)-\sigma(s,t,\ww_t)}^2 \ds \\
    &\le 2 \int_0^T \abs{\sigma(s,r,\ww_r)-\sigma(s,t,\ww_r)}^2 \ds + 2 \int_0^T \abs{\sigma(s,t,\ww_r)-\sigma(s,t,\ww_t)}^2 \ds \\
    &= 2 \int_r^T \abs{\sigma(s,r,\ww_r)-\sigma(s,t,\ww_r)}^2 \ds
    + 2 \int_t^r \abs{\sigma(s,t,\ww_r)}^2 \ds 
    + 2 \int_t^T \abs{\sigma(s,t,\ww_r)-\sigma(s,t,\ww_t)}^2 \ds\\
    &\lesssim G(\ww_r)^2 \int_r^T \left( \int_r^t (s-u)^{H-3/2} \du\right)^2 \ds 
    + G(\ww_r)^2 \int_t^r (s-t)^{2H-1}\ds \\
    &\qquad+ \varrho'(\abs{\ww_r-\ww_t})^2 \int_t^T (s-t)^{2H-1}\ds \,G(\ww_r\vee\ww_t)^2.
\end{align*}
We finally note that
\begin{align*}
    \int_r^T \left( \int_r^t (s-u)^{H-3/2} \du\right)^2 \ds 
    \le \int_r^T \left( \int_r^t (s-t)^{H/2-1/2} (t-u)^{H/2-1} \du\right)^2 \ds 
    = \frac{T^H}{H} \frac{4(t-r)^{H/2}}{H^2}.
\end{align*}
In a similar fashion we have
\begin{align*}
    \lim_{\ep\downarrow0} \big\lVert  \sigma^{\ep,r,\bm{X}^{t,\ww}} - \sigma^{\ep,r,\ww} \big\lVert_{L^2(\TT)}^2
    \lesssim \norm{G(\widehat{\bm{X}}^{t,\ww,r})}_{\TT} G(\bm{X}^{t,\ww}_r\vee \ww_r) \varrho'\left(\abs{\bm{X}^{t,\ww}_r-\ww_r}\right)^2 ,
\end{align*}
which vanishes for $r=t$ and is continuous in~$r$ since one can show~$\big\lVert\bm{X}^{t,\ww}_r-\ww_r\big\lVert_{L^p(\Omega)}\lesssim (r-t)^{H}$. 
Eventually this shows that $R_1(\delta)/\delta$ tends to zero as $\delta$ goes to zero. The same strategy applies to $R_2(\delta)$ as well.  
Finally, as $r\to t$,
\begin{align*}
    f(r,\bm{X}^{t,\ww}_r,\bm{Y}^{t,\ww}_r,\bm{Z}^{t,\ww}_r) - f(t,\ww_t,\bm{Y}^{t,\ww}_t,\bm{Z}^{t,\ww}_t)
\end{align*}
tends to zero from continuity assumptions of $f$ in $t$, $x$ and in $(y,z)$. Indeed, $\bm{Y}^{t,\ww},\bm{Z}^{t,\ww}$ are almost surely continuous and routine computations show that~$\bm{X}^{t,\ww}_r-\ww_t$ tends to zero as $r\to t$ since $\ww$ is also continuous.
\end{proof}
\begin{proposition}\label{prop:UniqueSol}
    Under the same assumptions as in Proposition~\ref{prop:ClassicalSol}, the PPDE~\eqref{Eq:PPDE_u} with its boundary condition has a unique~$\Cc^{1,2}_{+,\alpha}$ solution.
\end{proposition}
\begin{proof}
    This is inspired by the proof of~\cite[Theorem~4.1]{peng2016bsde}. 
    Given the assumptions, Proposition~\ref{prop:ClassicalSol} guarantees the existence of a solution ~$U\in \Cc^{1,2}_{+,\alpha}$ to the PPDE.
    Apply It\^o's formula~\eqref{eq:Ito} to $U(s,\widehat{\bm{X}}^{t,\ww,s})$
    and, after cancelling the terms from the PPDE, we obtain
    \begin{align*}
    \D U(s,\widehat{\bm{X}}^{t,\ww,s}) 
    = & -f\left(s, \widehat{\bm{X}}^{t,\ww,s}_s, U(s,\widehat{\bm{X}}^{t,\ww,s}) , \left\langle\partial_{\bm{\omega}} U(s,\widehat{\bm{X}}^{t,\ww,s}), \sigma^{s,\widehat{\bm{X}}^{t,\ww,s}} \right\rangle \right) \ds\\
    & + \left\langle\partial_{\bm{\omega}} U(s,\widehat{\bm{X}}^{t,\ww,s}), \sigma^{s,\widehat{\bm{X}}^{t,\ww,s}} \right\rangle \D \bm{W}_s,
    \end{align*}
    where we recall that~$\widehat{\bm{X}}^{t,\ww,s}_s=\bm{X}^{t,\ww}_s$.
    Combined with the boundary condition and the fact that the BSDE in~\eqref{eq:BSDE} has a unique solution by Assumption~\ref{assu:BSDE}, this implies
    $$ (\bm{Y}^{t,\ww}_s,\bm{\bm{Z}}^{t,\ww}_s)
    =\big(U(s,\widehat{\bm{X}}^{t,\ww,s}),\langle\partial_{\bm{\omega}} U(s,\widehat{\bm{X}}^{t,\ww,s}), \sigma^{s,\widehat{\bm{X}}^{t,\ww,s}} \rangle \big)
    $$
    is the unique solution to the BSDE~\eqref{eq:BSDE}.
    In particular, $\bm{Y}^{t,\ww}_t=U(t,\ww)$, concluding the proof.
\end{proof}

\begin{remark}\label{rem:OptionRPZ}
In financial applications, we are interested in evaluating conditional expectations of the type~$\EE[\Phi(S)|\Ff_t]$
where the asset price $S$ is defined explicitly. However in the general Volterra case (where the coefficients are state-dependent) it is not trivial to express it as a function of~$(t,\bm{X}\otimes_t\bm{\Theta}^t)$. With this in mind we observe that, for~$0\le t\le s\le T$, 
\begin{align}
\bm{X}_s &= \bm{X}_0 + \int_0^t \big[ b(s,r,\bm{X}_r) \dr + \sigma(s,r,\bm{X}_r) \,\D \bm{W}_r\big] + \int_t^s \big[ b(s,r,\bm{X}_r) \dr + \sigma(s,r,\bm{X}_r) \,\D \bm{W}_r\big] \\ \label{eq:SVE_split}
& = \bm{\Theta}^t_s + \int_t^s \big[ b(s,r,\bm{X}_r) \dr + \sigma(s,r,\bm{X}_r) \,\D \bm{W}_r\big]. 
\end{align}
Hence, by the strong uniqueness of the solutions to the SVEs~\eqref{eq:MainSVE} and~\eqref{eq:XsxEq}, 
$\bm{X}=  \bm{X}^{t,\bm{X}\otimes_t\bm{\Theta}^t}$ almost surely on~$\TT$.

With $\bm{Y}$ being the solution of the BSDE~\eqref{eq:BSDE} with~$f\equiv0$,
and for any~$\Ff_t$-measurable random variable~$\xi\in L^2(\Omega;\Ww)$, equations~\eqref{eq:BSDE} and~\eqref{eq:def_U} entail
    \begin{equation}\label{eq:U_is_cond_exp}
        U(t,\xi)=\bm{Y}^{t,\xi}_t=\EE[\phi(\bm{X}^{t,\xi}) \lvert \Ff_t].
    \end{equation}
This yields in particular
\begin{equation}
    \EE[\Phi(\bm{X})|\Ff_t] = \EE\left[\Phi\left( \bm{X}^{t,\bm{X}\otimes_t\bm{\Theta}^t}\right)|\Ff_t\right] = \bm{Y}^{t,\bm{X}\otimes_t\bm{\Theta}^t}_t = U(t,\bm{X}\otimes_t\bm{\Theta}^t).
\label{eq:Conditional_Exp}
\end{equation}
\end{remark}

\begin{remark}
The results of this section do not go through without the assumption of strong uniqueness of the solution to the SVE. This acts as a further impetus for resolving this open problem in the rough Heston model.    
\end{remark}

\begin{remark}\label{rem:path_der_special_cases}
The general setting simplifies when applied to semimartingales and/or payoffs that depend only on the terminal value. 
We recall the function $U:\overline{\Lambda}\to\RR$ given by~$U(t,\ww)=\EE[\phi(\bm{X}^{t,\ww})]$. 
    \begin{enumerate}
        \item If $K$ is a constant, i.e. $\bm{X}$ is a Markov process and a semimartingale, then $\ww\mapsto U(t,\ww)$ is in fact measurable with respect to~$\sigma(\{\ww\lvert_{[0,t]}: \ww \in \overline{\Ww}\})$, hence there exists 
    $$
u_1:\Big\{(t,\ww_1)\in\TT\times \big(\Df^0([0,t])\cap \Cc^0([0,t)\big) \Big\} \to \RR, \quad \text{such that}\; u_1(t,\ww\lvert_{[0,t]})= U(t,\ww).
    $$
        The perturbation then becomes $U(t,\ww+\eta\one_{[t,T]})=u_1(t,\ww\lvert_{[0,t]}+\eta_t \one_{\{t\}})$, which consists in perturbing the path only at the end point, akin to Dupire's vertical derivative~\cite{dupire2019functional}. The conditional expectation~\eqref{eq:Conditional_Exp} reads~$\EE[\Phi(\bm{X})\lvert \Ff_t]=u_1(t,\bm{X}_{[0,t]})$, which is consistent with the observation that~$\boldsymbol{\Theta}^t_s=X_t$ for all~$s\ge t$ and which entails that the concatenated path~$\bm{X}\otimes_t \bm{\Theta}^t$ is frozen after~$t$. This setting is reminiscent of the functional It\^o calculus developed by Dupire~\cite{dupire2019functional} with the aim of pricing path-dependent options on Markovian underlyings. Dupire considers paths stopped after~$t$; they are thus of different \emph{lengths} which induce a different state space and a different time derivative. The applications, however, are precisely those encompassed by the map~$u_1$.
        \item If $K$ is not constant but $\phi$ only acts on the terminal time, then $\ww\mapsto U(t,\ww)$ is now measurable with respect to~$\sigma(\{\ww\lvert_{[t,T]}: \ww \in \overline{\Ww}\})$, hence there exists 
    $$
u_2:\Big\{(t,\ww)\in\TT\times \big(\Df^0([t,T])\cap \Cc^0(t,T]\big) \Big\} \to \RR, \quad \text{such that}\; u_2(t,\ww\lvert_{[t,T]})= U(t,\ww).
    $$ 
    The perturbation remains $U(t,\ww+\eta\one_{[t,T]})=u_2(t,\ww\lvert_{[t,T]}+\eta \lvert_{[t,T]})$, and the conditional expectation is~$\EE[\phi(\bm{X}_T)\lvert \Ff_t]=u_2(t,\bm{\Theta}^t)$.
        \item If $K$ is constant and $\phi$ only acts on the terminal time, then $\ww\mapsto U(t,\ww)$ is measurable with respect to~$\sigma(\{\ww_t: \ww \in \overline{\Ww}\})$, hence there exists 
    $$
u_3:\Big\{(t,x)\in\TT\times\RR^d \Big\} \to \RR, \quad \text{such that}\; u_3(t,\ww_t)= U(t,\ww).
    $$
    The perturbation is~$U(t,\ww+\eta\one_{[t,T]})=u_3(t,\ww_t+\eta_t)$. This corresponds to the classical problem of pricing vanilla options in a semimartingale model, with the option price being equal to    the conditional expectation $\EE[\phi(\bm{X}_T)\lvert \Ff_t]=u_3(t,\bm{X}_t)$. 
    \end{enumerate}
\end {remark}

\subsection{Application to rough volatility} \label{sec:RvolApplication}
We consider the following rough volatility model:
\begin{equation}\label{eq:MainModel}
     X_t  :=x_0+ \int_{0}^t \psi(r,V_r) \D B_r + \zeta \int_0^t \psi(r,V_r)^2\dr, \qquad V_t  := \int_0^t K(t,r) \D W_r,
\end{equation}
with~$x_0,\zeta\in\RR$, $B=\overline{\rho} \overline{W}+ \rho W$, $W$ and $\overline W$ independent Brownian motions, $\rho\in[-1,1]$ and~$\overline{\rho}:=\sqrt{1-\rho^2}$. The kernel~$K:\TT^2\to\RR$ can be singular, $V$ is a Gaussian (Volterra) process and thus has finite exponential moments. We are chiefly interested in the Riemann-Liouville case~$K(t,s)=(t-s)^{H-\half}$.
 The two-dimensional~$\bm{X}=(X,V)^\top$ is a particular case of SVE~\eqref{eq:MainSVE} with $d=m=2$ and
\begin{align}\label{eq:coeffs_roughvol}
    b(t,r,x_1,x_2) = \begin{bmatrix} \zeta \psi(r,x_2)^2 \\ 0 \end{bmatrix}, 
    \qquad \quad 
    \sigma(t,r,x_1,x_2) = \begin{bmatrix} \bar{\rho} \psi(r,x_2) & \rho \psi(r,x_2) \\ 0 & K(t,r)\end{bmatrix}.
\end{align}
In financial models, $X$ is the log of an exponential martingale and thus~$\zeta=-\half$ while, for simplicity, previous papers studying weak error rates had set~$\zeta=0$. In this section we check all the conditions of Proposition~\ref{prop:ClassicalSol} and derive the pricing PPDE for the rough volatility model~\eqref{eq:MainModel}. In particular, this gives a rigorous justification to \cite[Remark 5.2]{viens2019martingale}.

Denote with $\{\mathcal{F}_t\}_{t \in \TT}$ the filtration generated by~$(\overline{W},W)$. We note that in this case
\begin{align}
    \Theta_s^t = \int_0^t K(s,r) \D W_r= \EE[V_s|\Ff_t], \quad s \ge t.
\end{align}
As a special case of~\eqref{eq:ThetaDef}, this process arises from the decomposition of $V$ as
$$
    V_s = \int_0^s K(s,r) \D W_r = \int_0^t K(s,r) \D W_r + \int_t^s K(s,r) \D W_r =: \Theta_s^t + I_s^t,
    \quad\text{for all }
    0\le t\le s \le T.
$$
This decomposition is orthogonal in the sense that~$I_s^t$ is independent from~$\mathcal{F}_t$ whereas~$\Theta_s^t = \EE[V_s | \mathcal{F}_t]$. In particular, $\Theta_t^t = V_t$ and for~$s \in\TT$ fixed, $\Theta_s^\cdot=(\Theta_s^t)_{t \in[0, s]}$ is a martingale.    

We are interested in representations of the value function $U:\overline{\Lambda}\to\RR$ such that $U(t,\ww)=\EE[\Phi(\bm{X}^{t,\ww})]$, which fits in the framework of the previous subsection with~$d=2$. As we saw in the introduction, see~\eqref{eq:utxomega}, the semimartingality of $X$ and the state-dependent payoff imply that the conditional expectation~$\EE[\phi(X_T)\lvert \Ff_t]$ is only a function of~$(t,X_t,\Theta^t)$. Applying items (2) and (3) of Remark~\ref{rem:path_der_special_cases} sequentially, one shows that there exists 
    \begin{equation}\label{eq:u_from_U}
u:\Big\{(t,x,\omega)\in\TT\times \RR\times \big(\Df^0([t,T])\cap \Cc^0(t,T]\big) \Big\} \to \RR, \text{ such that } u\big(t,\omega^{(1)}_t,\omega^{(2)}\lvert_{[t,T]}\big)= U(t,\ww).
    \end{equation}
    We thus recover the relation already derived in the introduction
$$
u(t,X_t,\Theta^t_{[t,T]})=\EE[\phi(X_T) |\Ff_t].
$$
    The perturbation in the first component is $U(t,\ww+\eta\one_{[t,T]})=u(t,\omega_t^{(1)} + \eta^{(1)}_t, \omega^{(2)}\lvert_{[t,T]})${ , with  $\eta=(\eta^{(1)},0)$}, and thus corresponds to a standard derivative (right derivative if $\eta^{(1)}_t>0$), while the derivative in the second component is the Fréchet derivative defined in~\eqref{eq:PathwiseDerDef}.
Analogously to~$\Lambda$ and~$\overline{\Lambda}$, we name the space of interest on which it acts
\begin{equation}\label{def:Gamma_spaces}
\Gamma:=\TT\times\RR\times\Ww,\qquad 
\overline{\Gamma}: = \big\{ (t,x,\omega)=\left(t,(x,\omega)\right)\in\TT\times\RR\times\overline{\Ww}: \omega|_{[t,T]}\in\Ww_t \big\},
\end{equation}
where $\Ww,\Ww_t,\overline{\Ww}$ are defined as in Section~\ref{sect:def_and_ito_form}, with $d=1$. The corresponding state space, analogue to Definition~\ref{def:Cplusalpha}, is denoted~$\Cc^{1,2,2}_{+,\alpha}(\Gamma)$.
\begin{definition}
\
\begin{itemize}
    \item[(i)] We say that~$f:\RR\to\RR$ belongs to~$\Cc^{m}_{{\rm poly}}(\RR)$ with $m\in\NN$ if there exists~$\kappa>0$ such that 
    $\max_{j\le m} \abs{f^{(j)}(x)} \lesssim 1+|x|^\kappa$
    for all~$x\in\RR$.
    \item[(ii)]    We say that~$g:\TT\times\RR\to\RR$ belongs to~$\Cc^{0,n}_{{\rm exp}}(\TT\times\RR)$ with $n\in\NN$ if there exists~$\kappa>0$ such that $\max_{j\le n} \abs{\partial_x^{(j)} g(t,x)} \lesssim 1+\E^{\kappa (x+t)}$
    for all~$(t,x)\in\TT\times\RR$.
\end{itemize}
\end{definition}
\begin{assumption}\label{assu:phipsi}\
\begin{itemize}
    \item[(i)] The payoff function $\phi$ belongs to~$\Cc^{3}_{{\rm poly}}(\RR)$ with constant~$\kphi>0$ and~$\phi^{(3)}$ is locally H\"older continuous;
    \item[(ii)] The volatility function $\psi$ belongs to~$\Cc^{0,3}_{{\rm exp}}(\TT\times\RR)$ with constant~$\kpsi>0$. Its first and second derivative in space are denoted by $\psi'$ and $\psi''$, respectively.
    \item[(iii)] The kernel $K$, similarly to Assumption~\ref{assu:RegTimeDerivative}, satisfies for all~$s<t$
    $$
    \abs{K(t,s)} \lesssim (t-s)^{H-\half}, \quad \abs{\partial_t K(t,s)} \lesssim(t-s)^{H-\frac32}.
    $$
\end{itemize}
\end{assumption}
\begin{remark}
    We are limited to polynomial growth for~$\phi$ as it is still an open problem whether \hbox{$\EE[\E^{p X_T}]<\infty$} for most values of~$(p,\zeta)\in \RR^2$. 
    { For recent developments on the topic we refer to \cite{gassiat2019martingale, abi2025martingale}}.
\end{remark}
This assumption paves the way for the following lemma, proved in Section~\ref{sec:Proof_Bounds_XV}.

\begin{lemma}\label{lemma:Bounds_XV}
    Let Assumption~\ref{assu:phipsi}(ii) hold.
    For all~$(t,x,\omega)\in\overline{\Gamma}$ and $s\in[ t,T]$, define
    \begin{equation}\label{eq:Xtxomega}
    V^{t,\omega}_s:= \omega_s+I^t_s, \quad \text{and}\quad 
     X_T^{t,x,\omega}:=x+\int_t^T \psi(s,V^{t,\omega}_s)\,\D B_s+\zeta\int_t^T \psi(s,V^{t,\omega}_s)^2\ds.    
    \end{equation}
    Then, for all~$p\ge1$,
    \begin{equation}\label{eq:Bounds_XV}     \EE[\E^{p\norm{V}_{\TT}}]<\infty
    , \quad   {\sup_{s\in[ t,T]}\EE\left[\E^{p \abs{V^{t,\omega}_s}}\right] \lesssim \E^{p \norm{\omega}_{\TT}},} \quad
\EE\left[\norm{X^{t,x,\omega}}_\TT^p\right] \lesssim \abs{x}^p + \E^{2\kpsi p \norm{\omega}_\TT} .
    \end{equation}
\end{lemma}
    This lemma characterises functions belonging to~$\Xx$ in the rough volatility framework. Indeed, any function~$G:\RR^2\to\RR$ $G(x,v)\lesssim 1+|x|^\ell + \E^{\ell v}$ for some $\ell >0$ satisfies
    \begin{equation}\label{eq:lemma_bound_G}
        \EE\left[\|G(X,V)\|_\TT\right]<\infty.
    \end{equation}
    Such a function thus belongs to the set~$\Xx$ corresponding to the system~\eqref{eq:MainModel}. These estimates further lead to growth estimates of the pathwise derivatives.
    In terms of notations, we write~$\EE_{t,x,\omega}$ for the expectation conditioned on~$X_t=x,\,\Theta^t=\omega$. In other words, it corresponds to the expectation with~$X^{t,x,\omega}$ and~$V^{t,\omega}$, as in the previous lemma. 
    The following proposition, proved in Section~\ref{sec:Proof_SpatialDerivatives}, gives an explicit representation of the derivatives of~$u$.
\begin{proposition}\label{prop:SpatialDerivatives}
Let Assumption~\ref{assu:phipsi} hold.
 {The function~$(x,\omega)\mapsto u(t,x,\omega)$ is twice continuously differentiable (in the sense of~\eqref{eq:PathwiseDerDef})} and, for all~$(t,x,\omega)\in\overline{\Gamma}$, $\eta$, $\eta^{(1)},\eta^{(2)}\in\Ww_t$,
\begin{align*}
    \partial_x u(t,x,\omega)&=\EE_{t,x,\omega}[\phi'(X_T)], \qquad
    \partial_{xx} u(t,x,\omega)=\EE_{t,x,\omega}[\phi''(X_T)], \\
    \langle \partial_{\omega} u(t,x,\omega), \eta\rangle &= \EE_{t,x,\omega}\left[\phi'(X_T) \left\{\int_t^T \psi'(s,V_s) \eta_s \,\D B_s +\zeta \int_t^T (\psi^2)'(s,V_s) \eta_s \,\D s \right\}\right], \\
    \langle \partial_\omega (\partial_x u)(t,x,\omega), \eta \rangle 
    &= \EE_{t,x,\omega}\left[\phi''(X_T)  \left\{\int_t^T \psi'(s,V_s) \eta_s \,\D B_s +\zeta \int_t^T (\psi^2)'(s,V_s) \eta_s \,\D s \right\}\right], \nonumber\\
    \langle \partial_{\omega\omega} u(t,x,\omega), (\eta^{(1)},\eta^{(2)})\rangle & = \EE_{t,x,\omega}\bigg[\phi''(X_T) \left\{\int_t^T \psi'(s,V_s) \eta^{(1)}_s \,\D B_s +\zeta \int_t^T (\psi^2)'(s,V_s) \eta^{(1)}_s \,\D s \right\} \nonumber\\
    & \qquad \qquad \cdot \left\{\int_t^T \psi'(s,V_s) \eta^{(2)}_s \,\D B_s +\zeta \int_t^T (\psi^2)'(s,V_s) \eta^{(2)}_s \,\D s \right\} \nonumber\\
    & + \phi'(X_T) \int_t^T \psi''(s,V_s) \eta^{(1)}_s \eta^{(2)}_s \,\D B_s  +\zeta \int_t^T (\psi^2)''(s,V_s) \eta^{(1)}_s \eta^{(2)}_s \,\D s  \bigg]. \nonumber
\end{align*}
\end{proposition}
The passage to singular kernels/directions requires some Malliavin calculus, which we briefly recall.
Adopting notations and definitions from~\cite[Section~1.2]{Nualart06}, we denote by~$\bm{\Df}$ the  Malliavin derivative operator with respect to~$\bm{W}$ and by~$\DD^{1,2}$ its domain of application in~$L^2(\Omega)$. For~$F\in\DD^{1,2}$, $\bm{\Df} F= (\Df^W F, \Df^{\overline{W}}F)^\top=:(\Df F, \overline{\Df}F)^\top$.
The (conditional) Malliavin integration by parts formula plays a crucial role
\begin{equation}
    \EE\left[ F \int_t^T \langle h_s,\D \bm{W}_s\rangle \Big| \Ff_t\right] = \EE\left[\int_t^T \langle \bm{\Df}_s F, h_s \rangle \ds \Big| \Ff_t\right],
\end{equation}
for any~$F\in\DD^{1,2}$ and { any $\Ff$-adapted}~$h\in L^2(\Omega\times\TT, \RR^2)$, and with~$\langle\cdot,\cdot\rangle$ the inner product in~$\RR^2$ (not to be confused with the pathwise derivative).
In particular, we use extensively the following application with~$h$ a real-valued process and~$\bm{\rho}=(\rho,\overline{\rho})$
\begin{align}\label{eq:IBP_B}
    \EE\left[ F\int_t^T h_s\,\D B_s \Big|\Ff_t\right] 
    = \EE \left[\int_t^T \langle \bm{\Df}_s F,\bm{\rho} \rangle h_s \ds \Big|\Ff_t\right].
\end{align}

We now state the pathwise derivative with respect to the singular kernel, 
with proof postponed to Section~\ref{sec:Proof_SingularDerivative}.
This representation resorts to the integration by parts~\eqref{eq:IBP_B}
since the stochastic integral~$\int_t^T \psi''(s,V_s)K(s,t)^2\D B_s$ is ill-defined in general.
\begin{proposition}\label{prop:SingularDerivative}
Under Assumption~\ref{assu:phipsi}, for all~$t\in\TT$, the conclusions of Proposition~\ref{prop:SpatialDerivatives} still hold when the directions~$\eta,\eta^{(1)},\eta^{(2)}\in\Ww_t$ are replaced by the singular kernel~$K^t:=K(\cdot,t)$, where~$K$ satisfies Assumption~\ref{assu:phipsi}(iii), with the modification
\begin{align} \label{eq:SingularDerivative}
  & \qquad\langle \partial_{\omega\omega} u(t,x,\omega), (K^t,K^t)\rangle\\
  & = \EE_{t,x,\omega}\left[\phi''(X_T) \left\{\int_t^T \psi'(s,V_s) K(s,t) \,\D B_s +\zeta \int_t^T (\psi^2)'(s,V_s) K(s,t) \,\D s \right\}^2 \right] \\ \nonumber
    & \quad + \EE_{t,x,\omega}\left[ \int_t^T \langle\bm{\Df}_s \phi'(X_T),\bm{\rho}\rangle \psi''(s,V_s) K(s,t)^2 \ds +\zeta \int_t^T (\psi^2)''(s,V_s) K(s,t)^2 \,\D s  \right]. 
\end{align}
\end{proposition}

The following proposition and lemma verify that the derivatives satisfy the proper regularity conditions of Definition~\ref{def:Cplusalpha} and Assumption~\ref{assumption:regcoef+U} and are  proved in Section~\ref{sec:Proof_Cplusalpha} and in Section~\ref{app:u_reg}, respectively.
\begin{proposition}\label{prop:Cplusalpha}
Under Assumption~\ref{assu:phipsi},
$u \in \Cc^{0,2,2}_{+,\alpha}$ with~$\alpha=\half$.
\end{proposition}
\begin{lemma}\label{lemma:u_regularity}
    Under Assumption~\ref{assu:phipsi}, $u$ satisfies Assumption~\ref{assumption:regcoef+U}.
\end{lemma}

As a consequence of Propositions~\ref{prop:ClassicalSol} and~\ref{prop:UniqueSol}, this yields the well-posedness of the PPDE, the main result of this section.
\begin{theorem}\label{th:pricingPDE}
Let Assumption~\ref{assu:phipsi} hold, then~$u\in \Cc^{1,2,2}_{+,\alpha}$, with~$\alpha=\half$, and is the unique classical solution of the PPDE
\begin{equation}\label{eq:PPDE_rvol}
    \partial_t u + \zeta\psi(t,\omega_t)^2\partial_x u + \half\psi(t,\omega_t)^2 \partial_{xx} u + \half \langle \partial_{\omega\omega} u, (K^t,K^t)\rangle + \rho\psi(t,\omega_t)\langle \partial_{\omega} (\partial_x u), K^t\rangle =0,
\end{equation}
for all~$(t,x,\omega)\in\overline{\Gamma}$ and with boundary condition~$u(T,x,\omega)=\phi(x)$. 
\end{theorem}
\begin{proof}[Proof]
For a two-dimensional path~$\boldsymbol{\omega}=(\omega^{(1)},\omega^{(2)})$, recall that~$U(t,\boldsymbol{\omega})=u(t,\omega^{(1)}_t,\omega^{(2)}\lvert_{[t,T]})$ as defined in~\eqref{eq:u_from_U}.
Let us start by checking that all the necessary assumptions are satisfied, namely~\ref{assu:wellposedness} and~\ref{assu:RegTimeDerivative}, the existence and uniqueness of strong solutions for our SDEs and~\ref{assu:BSDE}.
The two-dimensional SVE~\eqref{eq:MainModel} is explicit and Lemma~\ref{lemma:Bounds_XV} provides the moment bounds hence Assumption~\ref{assu:wellposedness} holds.
Regarding the coefficients, 
$\abs{\psi(y)}\le G(y)$, $\abs{\psi'(y)}\le G(y)$, $\abs{\psi^2(y)}\le G^2(y)$ and $\abs{(\psi^2)'(y)}=\abs{2\psi(y)\psi'(y)}\le 2G^2(y)$ for some~$G$
(and consequently~$G^2$) in~$\Xx$, from Assumption~\ref{assu:phipsi}(ii) and Lemma~\ref{lemma:Bounds_XV}, while~$K$ verifies Assumption~\ref{assu:phipsi}(iii) which ensures that Assumption~\ref{assu:RegTimeDerivative} is satisfied. 

Moreover, Assumption~\ref{assu:BSDE}(i) is also verified because  of Lemma~\ref{lemma:Bounds_XV} and Assumption~\ref{assu:phipsi}, whereas Assumption~\ref{assu:BSDE}(ii) and~(iii) are trivially satisfied as~$f\equiv0$. 

The last ingredient is Proposition~\ref{prop:Cplusalpha}, which allows us to apply Propositions~\ref{prop:ClassicalSol} and~\ref{prop:UniqueSol}. 
These propositions state that~$U$ is the unique solution to the PPDE~\eqref{Eq:PPDE_u}
    $$
\partial_t U +  \half \langle\partial_{\boldsymbol{\omega}\boldsymbol{\omega}} U,(\sigma^{t,\boldsymbol{\omega}},\sigma^{t,\boldsymbol{\omega}})\rangle + \langle \partial_{\boldsymbol{\omega}} U,b^{t,\boldsymbol{\omega}}\rangle=0,
    $$
    with terminal condition~$U(T,\boldsymbol{\omega})=\phi(\omega^{(1)}_T)$. Since~$U(t,\boldsymbol{\omega})=u(t,\omega^{(1)}_t,\omega^{(2)})$, the derivative in~$\omega^{(1)}$ is a standard derivative.
    Thanks to Remark~\ref{rem:path_der_multidim} and the definitions of the coefficients~\eqref{eq:coeffs_roughvol}, the derivatives correspond to
    \begin{align*}
        \langle \partial_{\boldsymbol{\omega}} U,b^{t,\boldsymbol{\omega}}\rangle 
        &= \langle \partial_{\omega^{(1)}} U,\zeta\psi(t,\omega_t^{(2)})^2\rangle = \zeta\psi(t,\omega_t^{(2)})^2\partial_x u, \\
        \langle\partial_{\boldsymbol{\omega}\boldsymbol{\omega}} U,(\sigma^{t,\boldsymbol{\omega}},\sigma^{t,\bm{\omega}})\rangle
        &= \langle\partial_{\omega^{(1)}\omega^{(1)}} U,(\psi(t,\omega^{(2)}_t),\psi(t,\omega^{(2)}_t))\rangle
        + 2\langle\partial_{\omega^{(1)}\omega^{(2)}} U, (\rho \psi(t,\omega^{(2)}_t), K^t)\rangle \\
        & \qquad + \langle\partial_{\omega^{(2)}\omega^{(2)}} U,(K^t,K^t)\rangle \\
        &=\psi(t,\omega^{(2)}_t)^2 \partial_{xx} u 
        + 2 \rho \psi(t,\omega^{(2)}_t) \langle\partial_{\omega^{(2)}}(\partial_x u), K^t\rangle 
        + \langle\partial_{\omega^{(2)}\omega^{(2)}} u,(K^t,K^t)\rangle.
    \end{align*}
    This boils down to~\eqref{eq:PPDE_rvol}.
\end{proof}
In passing, we showed that the functional Itô formula~\eqref{eq:Ito} holds for $u_t=u(t,X_t,\Theta^t)$:
\begin{corollary}
    Let Assumption~\ref{assu:phipsi} hold, then the functional Itô formula holds
    \begin{align}\label{eq:Ito_roughvol}
        \D u_t &= \Big\{\partial_t u_t+ \zeta\psi(t,V_t)^2\partial_x u_t +  \half\psi(t,V_t)^2 \partial_{xx} u_t
        +  \rho \psi(t,V_t) \langle\partial_{\omega}(\partial_x u_t), K^t\rangle 
        + \half\langle\partial_{\omega\omega} u_t,(K^t,K^t)\rangle \Big\}\D t \nonumber \\
        & \qquad +\psi(t,V_t) \partial_x u_t \D B_t + \langle \partial_{\omega} u_t , K^t \rangle \D W_t.
    \end{align}
\end{corollary}
\begin{proof}
    Notice that~$u\in\Cc^{1,2,2}_{+,\alpha}(\Gamma)$ entails that~$U\in \Cc^{1,2}_{+,\alpha}(\Lambda)$.
    Recall that~$u_t=u(t,X_t,\Theta^t)=U(t,\bm{X}\otimes_t \bm{\Theta}^t)=U_t$. It thus suffices to apply the functional Itô formula~\eqref{eq:Ito} to~$U_t$: 
    \begin{align*}
        \D u_t=\D U_t
        &= \Big( \partial_t U_t +  \half \langle\partial_{\boldsymbol{\omega}\boldsymbol{\omega}} U_t,(\sigma^{t,\boldsymbol{\omega}},\sigma^{t,\boldsymbol{\omega}})\rangle + \langle \partial_{\boldsymbol{\omega}} U_t,b^{t,\boldsymbol{\omega}}\rangle \Big) \dt  + \langle \partial_{\boldsymbol{\omega}} U_t,\sigma^{t,\boldsymbol{\omega}}\rangle \D \bm{W}_t,
    \end{align*}
    and identifying the derivatives as in the proof of Theorem~\ref{th:pricingPDE} concludes the proof.
\end{proof}

\section{Weak rates of convergence for rough volatility}\label{sec:weak_rates_of_conv}

\subsection{The main result}
Let $N \in \mathbb{N}$, set $\Delta := \frac{T}{N}$. and $t_i:= i \Delta$ for $i=0, \ldots, N$. 
For conciseness we consider the model~\eqref{eq:MainModel} with~$\zeta=0$ and~$\psi(s,v)=\psi(v)$.
Relaxing the latter is only cumbersome but should not alter the main results while relaxing the former generates additional terms which require a more involved analysis. In the literature, the drift has always been considered uninfluential on the rate. As a matter of fact, the rate obtained in Case~1 of our main theorem does not hold if a drift is present ($\zeta\neq0$).
We perform an Euler discretisation of~$X$ that we name~$\overline{X}$:
\begin{align}
    & \overline{X}_{t_0}=x_0,\\
    & \overline{X}_{t_{i+1}} = \overline{X}_{t_{i}} + \psi(V_{t_i})(B_{t_{i+1}}- B_{t_i}),
\end{align}
for $i=0,\ldots, N-1$, where~$V$ is simulated exactly (e.g. by Cholesky decomposition).
Then, we extend to the whole time interval the process $\overline{X}$ by interpolation
\begin{align}\label{eq:DefEuler}
    \overline{X}_t= x_0+\int_0^t \psi(V_{\kappa_s}) \D B_s,
\end{align}
{ with $\kappa_s=t_i$ for $s \in [t_i, t_{i+1})$.}
Our goal is to estimate the difference between the price and its approximation with respect to the number of grid points~$N$:
$$ 
\mathcal{E}^N := \EE[\phi(X_T)]-\EE\left[\phi\left(\overline X_T\right)\right] .
$$
From now on, $f(\Delta)\lesssim g(\Delta)$ means there exists~$C>0$ such that $f(\Delta)\le Cg(\Delta)$ where~$C$ does not depend on~$\Delta$ (or equivalently on~$N$), but may depend on other constants such as~$p,H$ and the growth constants of~$\phi$ and~$\psi$. 
Our main result characterises the weak rate of convergence of this numerical scheme as follows:
\begin{theorem} \label{th:main} Let~$H\in[0,\half)$ and~$K(t,s)=(t-s)^{H-\half}$ for $s<t$ and zero otherwise.
\begin{itemize}
\item[\emph{Case 1.}] If~$\phi$ is quadratic and $\psi\in \Cc^3_{{\rm exp}}(\RR)$ then 
$\displaystyle \abs{\mathcal{E}^N} \lesssim \Delta$.
\item[]
\item[\emph{Case 2.}] 
If $\phi\in \Cc^5_{{\rm poly}}(\RR)$ and $\psi\in \Cc^{3}_{{\rm exp}}(\RR)$ 
then 
$\displaystyle \abs{\mathcal{E}^N}\lesssim \bigstar(\Delta)$, where
\begin{align*}
    \bigstar(\Delta):
 = \Delta^{3H+\half} \one_{H<\halfs}
 + \Delta \one_{H>\halfs}
 + \Delta|\log(\Delta)|\one_{H=\halfs}.
\end{align*}
\end{itemize}
\end{theorem}
\begin{remark}
    The rate~$\bigstar(\Delta)$ was proven to be optimal in~\cite[Proposition 6.1]{friz2022weak}.
\end{remark}
\begin{remark}
    The regularity required for~$\phi$ and~$\psi$ is independent of~$H$ but depends on the number of Malliavin integration by parts required to disentangle the different terms.
\end{remark}
\begin{remark}
    The growth conditions and Lemma~\ref{lemma:Bounds_XV} imply  that there exist~$\Const>0$ (depending only on the growth constants~$\kappa_\phi,\kappa_\psi$) such that for all~$p\ge1$ and~$k\le 3,\,n\le 5$,
\begin{equation}\label{eq:Estimates_phinpsin}
\EE\left[\Big\lvert\phi^{(n)}(X_T)\Big\lvert^p\right] \le \Const^{p}, \quad \sup_{t\in\TT} \EE\left[\Big\lvert\psi_t^{(k)}(V_t)\Big\lvert^p\right] \le \Const^{p},
\quad \sup_{t\in\TT} \EE\left[\Big\lvert{ (\psi_t^{2})^{(k)}}(V_t)\Big\lvert^p\right] \le \Const^{p}.
\end{equation}
\end{remark}
\subsection{Decomposition of the error}\label{sec:ErrorSum}
We first decompose the error 
in the spirit of~\cite[Section 3.3]{bayer2020weak}, see also \cite{talay1986discretisation}. Let $\bu_t:= u(t,\overline{X}_t, \Theta^t_{[t,T]})$ for all~$t\in\TT$,  where $u$ is defined in Theorem~\ref{th:pricingPDE} as the unique solution to the PPDE.
Then~\eqref{eq:U_is_cond_exp} yields~$\bu_t=\EE[\phi(\XTb)\lvert\Ff_t]=\EE[\phi(X_T)\lvert X_t=\overline{X}_t,\Theta^t]$, where~$X_T^{t,\overline{X}_t}:=\overline{X}_t+ \int_t^T \psi(V_s)\,\D B_s$.
\begin{proposition}\label{prop:Telescopic}
Under Assumption~\ref{assu:phipsi}, we have
\begin{align}
\mathcal{E}^{N} = \sum_{i=0}^{N-1} \frA_i, \qquad \text{where }\quad 
&
 \frA_i :=   \int_{t_i}^{t_{i+1}} \Big(\half \frB_1(t) + \rho\frB_2(t) \Big)\dt \qquad \text{for all   }i\le N-1,
\end{align}
and
\begin{align*}                                     
    \frB_1(t):=\EE\big[\left(\psi(V_t)^2 - \psi(V_{t_i})^2\right)\partial_{xx} \bu_t\big],\qquad
    \frB_2(t):=\EE\big[\big(\psi(V_{t}) - \psi(V_{t_i}) \big)  \langle \partial_\omega (\partial_x\bu_t) ,K^t\rangle \big].
\end{align*}
\end{proposition}
\begin{remark}
By virtue of Proposition~\ref{prop:SpatialDerivatives}, the following representations hold:
\begin{equation}\label{eq:der_u}
\left\{
\begin{array}{rl}    \partial_{xx}\bu_t
     & = \displaystyle \EE\left[\phi''(\XTb) |\, \Ff_t \right],\\
    \left\langle \partial_\omega (\partial_x \bu_t), K^t \right\rangle 
    &  = \displaystyle \EE\left[\phi''(\XTb)  \int_t^T \psi'(V_s) K(s,t) \,\D B_s \Big|\Ff_t\right].
\end{array}
\right.
\end{equation}
\end{remark}

\begin{proof}
Recalling that~$t_0=0$ and~$t_N=T$, we write~$\mathcal{E}^N$ as a telescopic sum:
\begin{align*}
    \EE\left[\phi(X_T) - \phi\left(\overline X_T\right)\right]
    & = \EE\left[\EE\left[\phi(X_T)| X_0=\overline{X}_0\right]\right]-\EE\left[\EE\left[\phi\left(\overline X_T\right)| X_T=\overline{X}_T\right]\right]  \\
    & = \EE\left[u\left(t_0, \overline{X}_{t_0}, \Theta^{t_0}_{\left[t_0,t_N\right]}\right)\right]
    - \EE\left[u\left(t_n, \overline{X}_{t_N}, \Theta^{t_N}_{t_N}\right)\right]\\
    & = \sum_{i=0}^{N-1} \Big\{\EE\Big[u\left(t_i, \overline{X}_{t_i}, \Theta^{t_i}_{\left[t_i,t_N\right]}\right)\Big] - \EE\Big[u\left(t_{i+1}, \overline{X}_{t_{i+1}}, \Theta^{t_{i+1}}_{\left[t_{i+1},t_N\right]}\right)\Big] \Big\}\\
    & =: \sum_{i=0}^{N-1} \frA_i.
\end{align*}
We now show that the terms $(\frA_i)_{i=0,\dots, N-1}$ have the form given in the proposition.
By virtue of the representation~\eqref{eq:DefEuler} and the regularity of~$u$ from Theorem~\ref{th:pricingPDE}, we apply the functional It\^o formula~\eqref{eq:Ito} on $[t_i, t_{i+1})$ to~$u(t,\overline{X}_t,\Theta^t)$ 
\begin{align*}
    \D \bu_t = &\bigg( \partial_t \bu_t + \psi(V_{t_i})^2\half\partial_{xx}  \bu_t + \half \langle \partial_{\omega\omega} \bu_t, (K^t,K^t)\rangle + \rho\psi(V_{t_i})\langle \partial_{\omega} (\partial_x \bu_t), K^t\rangle\bigg)\dt \\
    &+ \psi(V_{t_i}) \partial_x \bu_t \,\D B_t + \langle \partial_\omega \bu_t, K^t \rangle \,\D W_t.
\end{align*}
Combining this with the path-dependent PDE~\eqref{eq:PPDE_rvol} with~$\zeta=0$, we obtain, as claimed,
\begin{align*}
    \frA_i &= \EE\left[u\left(t_i,\overline{X}_{t_i},\Theta^{t_i}_{[t_i,T]}\right) - u\left(t_{i+1},\overline{X}_{t_{i+1}},\Theta^{t_{i+1}}_{[t_{i+1},T]}\right) \right] \\
    & = - \EE\left[ \int_{t_i}^{t_{i+1}} 
    \left\{ \partial_t \bu_t + \psi(V_{t_i})^2\half\partial_{xx}  \bu_t + \half \langle \partial_{\omega\omega} \bu_t, (K^t,K^t)\rangle + \rho\psi(V_{t_i})\langle \partial_{\omega} (\partial_x \bu_t), K^t\rangle\right\}\dt\right]\\
    & =  \EE \left[ \int_{t_i}^{t_{i+1}}  \left\{\left(\psi(V_t)^2 - \psi(V_{t_i})^2\right) \half\partial_{xx} \bu_t  + \rho \Big(\psi(V_{t}) - \psi(V_{t_i}) \Big)  \langle \partial_\omega (\partial_x\bu_t) ,K^t\rangle    \right\}\dt\right].
\end{align*}
\end{proof}

\subsection{Proof of Theorem~\ref{th:main}-Case~1}\label{sec:ProofThmCase1}
We start with a technical lemma~\cite[Lemma~2.2]{friz2022weak}:
\begin{lemma}\label{lemma:Friz}
Consider the function $\varphi:[0,\infty)\to\RR$ defined by
$$
\varphi(t) := \EE\left[g\left(V_t\right)\right],
$$
where~$g:\RR\to\RR$ is a Schwartz function.
Then $\varphi$ is~$\Cc^1$ and, for any $0\leq u\leq t$, 
$$
\left|\varphi(t) - \varphi(u)\right| 
 \le \frac{1}{2} \sup_{s\in\TT} \EE\left[\abs{g''(V_s)}\right] 
 \left(t^{2H} - u^{2H}\right).
$$
\end{lemma}
In cases where the (local) weak error rate of the integrand in~$\frA_i$ is independent of~$t$, say~$f(\Delta)$, then the global rate is~$\abs{\mathcal{E}^N}\lesssim \sum_{i=0}^{N-1}\int_{t_i}^{t_{i+1}}\dt f(\Delta)=f(\Delta)$. Thanks to the following observation, we obtain rate one even when the weak local error~$t^{2H}-t_i^{2H}$ in Lemma~\ref{lemma:Friz} would yield a lower rate.
\begin{lemma}\label{lemma:trick_rate_one}
For any~$\gamma>0$,
$\displaystyle
\sum_{i=0}^{N-1} \int_{t_i}^{t_{i+1}} (t^\gamma -t_i^\gamma)\dt \leq T^{1+\gamma} \Delta$.
\end{lemma}
\begin{proof}
The following computations are straightforward:
\begin{align*}
\sum_{i=0}^{N-1}\int_{t_i}^{t_{i+1}} (t^{\gamma}-t_i^{\gamma}) \D t 
    &= \sum_{i=0}^{N-1} \left\{  \frac{t_{i+1}^{\gamma+1}-t_i^{\gamma+1}}{\gamma+1} - t_i^{\gamma}(t_{i+1}-t_i) \right\}\\
    & =   \sum_{i=0}^{N-1} \left\{  \frac{((i+1)\Delta)^{\gamma+1}-(i \Delta)^{\gamma+1}}{\gamma+1} - (i \Delta)^{\gamma}\Delta \right\}\\
    & =  \Delta^{\gamma+1} \left\{  \frac{1}{\gamma+1}\sum_{i=0}^{N-1} \left[(i+1)^{\gamma+1}-i ^{\gamma+1}\right] - \sum_{i=0}^{N-1} i^{\gamma} \right\}\\
    & \le  \Delta^{\gamma+1} \left\{  \frac{N^{\gamma+1}}{\gamma+1}  - \int_{0}^{N-1} x^{\gamma} \D x \right\}
     = \frac{\Delta^{\gamma+1}}{\gamma+1} \left( N^{\gamma+1} - (N-1)^{\gamma+1}  \right)\\
 &  \leq  \frac{\Delta^{\gamma+1}}{\gamma+1} \max_{x \in [N-1,N]} x^{\gamma}(\gamma+1) \left( N-(N-1)\right)
    =   \Delta^{\gamma+1} N^{\gamma}
    =   \Delta T^{\gamma+1},
\end{align*}
where the last line follows from the mean value theorem and the monotonicity of  $x\mapsto x^{\gamma}$.
\end{proof}
Case~1 of Theorem~\ref{th:main} (with quadratic payoff) follows immediately by showing that, for each~$i\le N$, $\abs{\frA_i}\lesssim \int_{t_i}^{t_{i+1}} \big(t^{2H}-t_i^{2H}\big)\dt$ and invoking Lemma~\ref{lemma:trick_rate_one}.

We start with applying the Clark-Ocone formula to~$\psi(V_t)-\psi(V_{t_i})$, which is possible since~$\psi$ is differentiable and~$V_t\in\DD^{1,2}$: 
\begin{align*}
    \psi(V_t)-\psi(V_{t_i})
    &=\EE[\psi(V_t)-\psi(V_{t_i})]+\int_0^t \Df^W_s \EE[\psi(V_t)-\psi(V_{t_i})\lvert \Ff_s] \D W_s.
\end{align*}
Integration by parts then yields
\begin{align*}
    \frB_2(t)
    &= \EE[\psi(V_t)-\psi(V_{t_i})] \EE[\langle\partial_\omega(\partial_x \bu_t),K^t\rangle]\\
    &\quad + \EE\left[\int_0^t \Df^W_s \EE[\psi(V_t)-\psi(V_{t_i})\lvert \Ff_s] \Df^W_s \langle\partial_\omega(\partial_x \bu_t),K^t\rangle\ds\right].
\end{align*}
Note that for $t\ge0$, if $\phi$ is quadratic then $\langle \partial_{\omega}(\partial_x \bu_t),K^t\rangle=0$ as can be seen from~\eqref{eq:der_u}.  
Since~$\EE[\langle\partial_\omega(\partial_x \bu_t),K^t\rangle]$ is bounded we get~$ \abs{\frB_2(t)}\lesssim t^{2H}-t_i^{2H}$ 
by Lemma~\ref{lemma:Friz}. The same strategy applies to~$\frB_1$, which yields the claim.
\begin{remark}
    If a drift is present ($\zeta\neq0$), then~$\Df_s\langle \partial_\omega (\partial_x \bu_t),K^t\rangle$ may not be null for a quadratic~$\phi$ and the conclusion may not hold.
\end{remark}

\subsection{Proof of Theorem~\ref{th:main}-Case~2}\label{sec:ProofThmCase2}
The main computational ingredients are decomposed in two parts, whose proofs are postponed to Sections~\ref{sec:Bone} and~\ref{sec:Btwo} respectively, relying on independent estimates developed in Section~\ref{sec:useful_results}.
\begin{proposition}\label{prop:Bone}
For any~$t\in[t_i,t_{i+1})$, 
    $\abs{\frB_1(t)} \lesssim \bigstar(\Delta) + (t^{2H}-t_i^{2H}) $.
\end{proposition}
\begin{proposition}\label{prop:Btwo}
For any~$t\in[t_i,t_{i+1})$, 
    $\abs{\frB_2(t)} \lesssim \bigstar(\Delta) + (t^{2H}-t_i^{2H})$.
\end{proposition}
This is sufficient to finish the proof of the main result invoking Lemma~\ref{lemma:trick_rate_one}:
$$
\Ee^N
= \sum_{i=0}^{N-1} \frA_i
= \sum_{i=0}^{N-1} \int_{t_i}^{t_{i+1}} \left(\half \frB_1(t) + \frB_2(t) \right)\dt 
\lesssim \bigstar(\Delta). 
$$

\section{Useful results}\label{sec:useful_results}
To prove Propositions~\ref{prop:Bone} and~\ref{prop:Btwo} in Sections~\ref{sec:Bone} and~\ref{sec:Btwo}, we require a few estimates of independent interests, 
which we gather in this section.
Let us introduce the following useful notations:
\begin{itemize}
    \item We write $A\lesssim B$ whenever there is a constant $c>0$ such that $A\lesssim c B$  and $c$  only depends on~$H$ and $L^p$ bounds of $\phi,\,\psi$ and their derivatives (finite thanks to Assumption~\ref{assu:phipsi}).
    \item { We consider a time-grid of $(N+1)$ equally spaced points~$(t_i:=\frac{iT}{N})_{i=0}^{N}$ and we study the convergence, as $\Delta=\frac{T}{N}$ goes to zero. Throughout this section we fix~$t \in [t_i, t_{i+1}]$.
    \item For all $r\in\TT$, we recall the notation $\kappa_{r}:=\max\{t_j,\, t_j\le r\}$ and introduce~$\kappa^t_q=\kappa_q \one_{q\le t} + q\one_{q>t}$.
    \item Furthermore, for all~$r\in\TT$, 
    define~$\tau(r)$ as the unique integer such that~$t_{\tau(r)-1}=\kappa_r$ (equivalently $t_{\tau(r)}=\kappa_r+\Delta$).
    In particular, 
    $t_{\tau(r)-1} =\kappa_r \leq r < \kappa_r +\Delta =t_{\tau(r)}$.  
    \item Finally, let us note $\Uu_r:=\Df^W_r (\partial_x^2\bar{u}_t)$ and $\varphi_r(t):=\EE_r[(\psi^2)'(V_t)]$.}
\end{itemize}
Instead of the crude Cauchy-Schwarz inequality, we rely extensively on the Malliavin integration by parts
$\EE[Y \int Z_{s}\D W_s] = \EE[\int \Df^W_s Y Z_{s} \ds]$: 
this is necessary to achieve the optimal rate but induces a higher number of terms to estimate. 
Combining the Cauchy-Schwarz inequality, the~$L^p$ bounds of~$\phi,\psi$ (and their derivatives) and Lemmas~\ref{lemma:BoundDiscreteIntegral} and~\ref{lemma:Delta_varphi}, the quantities of interest can all nevertheless be reduced  to deterministic integrals. 
Then, we leverage Lemmas~\ref{lemma:bound_deltaK} to~\ref{lemma:doublesum}  and analogous computational tricks for fractional integrals to obtain the desired bounds. 
In particular, when direct computation is not feasible, we  regularly decompose the integrals on the partition $(t_j)_{j=0}^i$ and then replace the running variables by the closest point on that grid, e.g. for $s\in[t_j,t_{j+1}]$ and $k\le j-1$, we have $(s-t_k)^{\alpha}\le (t_j-t_k)^\alpha=(j-k)^\alpha \Delta^\alpha$. 
Finally, the last step consists in estimating the sums.


The following results are used on several occasions. 
First of them is the following trivial inequality,  stated without proof, which will appear in many forms repeatedly: for any continuous function~$g$ decreasing on some interval $[a,b]$,
with $a,b\in\NN$, then
\begin{equation}\label{eq:BoundIneqSum}
\sum_{u=a}^{b}g(u) \lesssim 
\sum_{\ell=a}^{b}\int_{\ell-1}^{\ell}g(u)\du
= \int_{a-1}^{b}g(u)\du.
\end{equation}

\subsection{Stochastic estimates}
\begin{lemma}\label{lemma:BoundDiscreteIntegral}
    Let $Y$ be an $L^2(\Omega)$-random variable  and $Z$ be an adapted process such that~${ \EE[\sup_{s\in\TT} \abs{Z(s)}^2]}$ is finite.
    Then, for all~$0\le r\le t\le t_{i+1}$:
    \begin{equation}\label{eq:BoundDiscreteIntegral}
        \EE\left[\abs{Y} \int_r^t \abs{Z(s)} K(\kappa_s,r)\D B_s\right]
        \lesssim \Big(1+K(\kappa_r+\Delta,r)\sqrt{\Delta}  \Big) \one_{r<t_i}.
    \end{equation}
\end{lemma}
\begin{proof}
By Cauchy-Schwarz inequality and It\^o's isometry,
\begin{align*}
\EE\left[\abs{Y} \int_r^t \abs{Z(s)} K(\kappa_s,r)\D B_s\right]^2
\le \EE\left[Y^2\right] \EE\left[\sup_{s\in\TT} \abs{Z(s)}^2\right] \int_r^t K(\kappa_s,r)^2\ds.
\end{align*}
Notice that~$K(\kappa_r,r)=0$ so that, for all~$r\in[0,t)$, it holds~$\int_r^{\kappa_r+\Delta} K(\kappa_s,r)^2\ds=0$. In particular,  {for~$r\in[t_i,t)$}, we have~$\int_r^tK(\kappa_s,r)^2\ds=0$.
For~$r<t_i$, we write
    \begin{align*}
        \left| \int_r^t K(\kappa_s,r)^2\ds\right|^\half 
        & = \left|\sum_{j=\tau(r)}^{i-1} K(t_j,r)^2\Delta + K(t_i,r)^2(t-t_i)\right|^{\half}\\
        & \le K(t_{\tau(r)},r)\sqrt{\Delta} + \left|\sum_{j=\tau(r)+1}^{i} K(t_j,r)^2\Delta \right|^{\half} \\
        &\le K(\kappa_r+\Delta,r)\sqrt{\Delta} + \left|\sum_{j=\tau(r)+1}^{i} \int_{t_{j-1}}^{t_j} K(t_j,r)^2\du  \right|^{\half}\\
        & = K(\kappa_r+\Delta,r)\sqrt{\Delta} + \left|\int_{t_{\tau(r)}}^{t_i} K(u,r)^2\du  \right|^{\half}\\
    &  = K(\kappa_r+\Delta,r)\sqrt{\Delta} + \frac{(t_i-t_{\tau(r)})^H}{\sqrt{2H}},
\end{align*}
and the claim follows.
\end{proof}
The following lemma is a conditional version of~\cite[Lemma~2.2]{friz2022weak}. \begin{lemma}\label{lemma:Delta_varphi}
For all $r\le t$ and $g\in \Cc^2(\RR)$, 
define
$ G_r(t):=\EE_r[g(V_t)]$. For all $r\le t_i$,
    $$
     G_r(t)-G_r(t_i)
=\int_{t_i}^t \left\{\EE_r[g'(V_t)]
\left(H-\half\right)\int_0^r (s-u)^{H-\halft} \D W_u+ \half \EE_r[g''(V_t)](s-r)^{2H-1}\right\}\ds.
$$
This leads to the estimate
$$
\abs{G_r(t)-G_r(t_i)}
\lesssim \sup_{q\in[0,t],n=1,2} \abs{\EE_r[g^{(n)}(V_q)]}  \int_{t_i}^t \left\{\abs{ \int_0^r (s-u)^{H-\halft} \D W_u}+ (s-r)^{2H-1}\right\}\ds.
$$
\end{lemma}
\begin{proof}
Assume first that~$g$ is a Schwartz function, in which case the proof mimics that of~\cite[Lemma~2.1]{friz2022weak} but in a conditional fashion.
Indeed, in this case~$g$ admits a Fourier transform, which reads~$
\widehat{g}(\xi) := \frac{1}{\sqrt{2\pi}}
\int_{\RR}\E^{ {\rm i} \xi x}g(x)\D x$, for all $\xi\in\RR.$
Since the Fourier transform preserves the Schwartz property, then~$g$ can be uniquely recovered through the identity~$g(x) = \frac{1}{\sqrt{2\pi}}\int_{\RR}\E^{-{\rm i}\xi x}\widehat{g}(\xi)\D\xi$, for all $x\in\RR$.
Note in particular that Fubini implies
\begin{align*}
g''(x)
= \frac{1}{\sqrt{2\pi}}\partial_{xx}\left(\int_{\RR}\E^{-{\rm i}\xi x}\widehat{g}(\xi)\D\xi\right)
 & = \frac{1}{\sqrt{2\pi}}\int_{\RR}
\partial_{xx}\left(\E^{- {\rm i}\xi x}\widehat{g}(\xi)\right)\D\xi\\
 & = -\frac{1}{\sqrt{2\pi}}\int_{\RR}
\xi^2 \E^{- {\rm i}\xi x}\widehat{g}(\xi)\D\xi.
\end{align*}
We have, for all $r\le t$, 
\begin{align*}
    G_r(t)=\EE_r[g(V_t)]
     & = \frac{1}{\sqrt{2\pi}}\int_{\RR}\EE_r\Big[\E^{- {\rm i}\xi V_t}\Big]\widehat{g}(\xi)\D\xi\\
     & = \frac{1}{\sqrt{2\pi}} \int_{\RR} \widehat{g}(\xi) \exp\Big\{-{\rm i}\xi\Theta^r_t  - \frac{1}{4H}\xi^2(t-r)^{2H}\Big\}\D \xi.
\end{align*}
Noting that $g'(x)=-\frac{{\rm i}}{\sqrt{2\pi}} \int_{\RR} \widehat{g}(\xi)\xi \E^{{ -}{\rm i} \xi x} \D \xi$  and $g''(x)=-\frac{1}{\sqrt{2\pi}} \int_{\RR} \widehat{g}(\xi)\xi^2 \E^{{ -}{\rm i} \xi x} \D \xi$, we obtain
\begin{align}
    G_r'(t)
    &= \frac{1}{\sqrt{2\pi}} \int_{\RR} \widehat{g}(\xi)\left(- {\rm i}\xi \partial_t \Theta^r_t - \half \xi^2 (t-r)^{2H-1}\right) \E^{{ -}{\rm i} \xi \Theta^r_t -\frac{\xi^2}{4H}(t-r)^{2H}}\D \xi \nonumber\\
    &=\frac{1}{\sqrt{2\pi}}  \int_{\RR} \widehat{g}(\xi)\left( -{\rm i}\xi \partial_t \Theta^r_t - \half \xi^2 (t-r)^{2H-1}\right)\EE_r[\E^{{ -}{\rm i}\xi V_t}]\D \xi\nonumber \\
    &=\left( \partial_t \Theta^r_t \EE_r[g'(V_t)] + \frac{(t-r)^{2H-1}}{2} \EE_r[g''(V_t)]\right). \label{eq:varphi_prime}
\end{align}
Since we have~$\partial_s \Theta^r_s =(H-\half)\int_0^r (s-u)^{H-\halft}\D W_u$, this entails
\begin{align*}
G_r(t)-G_r(t_i)
& = \int_{t_i}^t G'_r(s) \ds\\
& =\int_{t_i}^t \left\{\EE_r[g'(V_t)]\left(H-\half\right)\int_0^r (s-u)^{H-\halft} \D W_u
+ \frac{(s-r)^{2H-1}}{2}\EE_r[g''(V_t)]\right\}\ds.
\label{eq:DeltaVarphi}
\end{align*}
The result is finally extended from Schwartz function to $\mathcal{C}^2$ as in the proof of~\cite[Lemma~2.2]{friz2022weak}, simply by applying the conditional version of the dominated convergence theorem in place of the standard result.
\end{proof}


\subsection{Fractional calculus}
Recall that $t_i\le t\le t_{i+1}$. 
For $\alpha\in(-\half,2H]$, define
\begin{equation}\label{eq:diamondsuit}
\diamondsuit_\alpha(\Delta):=\Delta^{H+\half+\alpha}\one_{\alpha+H<\half}+\Delta\one_{\alpha+H>\half}
+ \Delta|\log(\Delta)|\one_{\alpha+H=\half},
\end{equation}
and in particular
\begin{equation}\label{eq:bigstar}
\bigstar(\Delta):= 
\diamondsuit_{2H}(\Delta)
 = \Delta^{3H+\half} \one_{H<\halfs}
 + \Delta \one_{H>\halfs}
 + \Delta|\log(\Delta)|\one_{H=\halfs}.
\end{equation}
\begin{lemma}\label{lemma:bound_deltaK} 
For all~$\alpha\le1$, $r\in[t_{j-1},t_{j})$, $j<i$,
$\displaystyle \abs{(t-r)^\alpha-(t_i-r)^\alpha} \lesssim \Delta^\alpha (i-j)^{\alpha-1}$. In particular, this entails
$\displaystyle
\abs{K(t,r)-K(t_i,r)} \lesssim \Delta^{H-\half} (i-j)^{H-\halft}$.
\end{lemma}
\begin{lemma}\label{lemma:bound_KDeltaK}
For $\alpha\in(-\half,2H]$,
$\displaystyle    \left|\int_{0}^{t_{i}} K(t_{i},r)\big((t_i-r)^\alpha-(t-r)^\alpha\big)\dr\right|
    \lesssim \diamondsuit_\alpha(\Delta).
$
\end{lemma}
\begin{lemma}\label{lemma:ti_tj}
For $j<i$, $\alpha \in (-\half,0]$,
$\displaystyle
        \int_0^{t_i} (t_i-r)^\alpha \abs{K(t_i,r)-K(t_j,r)}\dr 
        \lesssim (t_i-t_j)^{H+\half+\alpha}$.
\end{lemma}
\begin{lemma}\label{lemma:Bound_kappar_DeltaK}
For $t\in[t_i,t_{i+1})$, 
$\displaystyle
\int_0^{t_i} (t_i-\kappa_r)^{2H} \abs{K(t,r)-K(t_i,r)} \dr 
\lesssim \bigstar(\Delta)$.
\end{lemma}
The following is a mild modification of~\cite[Lemma 2.1~(2.2)]{gassiat2022weak} which features the logarithmic correction for $H=\halfs$ that the author refers to in the published version of the paper.
\begin{lemma}\label{lemma:Gassiat_log}
For $t\in[t_i,t_{i+1})$, 
$\displaystyle
\int_0^{t} (t-r)^{2H} \abs{K(t,r)-K(t_i,r)} \dr 
\lesssim \bigstar(\Delta)$.
\end{lemma}
\begin{lemma}\label{lemma:doublesum}
$\displaystyle
S:=\Delta^{2H}
 \sum_{j=1}^{i-2}\left( |i-j-1|^{H-\half}\sum_{\ell=0}^{j-1} |i-\ell-1|^{H-\halft}
 \int_{t_\ell}^{t_{\ell+1}} K(t_{j},r) \dr \right)\lesssim~\bigstar(\Delta)$.
\end{lemma}
\subsection{Proofs of the lemmas}
\begin{proof}[Proof of Lemma~\ref{lemma:bound_deltaK}]
Let $r\in[t_{j-1},t_{j})$ and $j \le i-1$; the lemma follows from
$$
\abs{(t-r)^\alpha-(t_i-r)^\alpha}
= \left|\alpha\right|\int_{t_i}^t (s-r)^{\alpha-1}\ds
\lesssim (t_i-t_{j})^{\alpha-1} (t-t_i)
\lesssim \Delta^{\alpha} (i-j)^{\alpha-1}.
$$
\end{proof}

\begin{proof}[Proof of Lemma~\ref{lemma:bound_KDeltaK}]
    First we decompose the integral
\begin{equation}\label{eq:decomposition_bound_KDeltaK}
    \int_{0}^{t_{i}} K(t_{i},r)\big[(t_i-r)^\alpha-(t-r)^\alpha\big]\dr
    = \sum_{\ell=0}^{i-1} 
 \int_{t_\ell}^{t_{\ell+1}} K(t_i,r)\big[(t_i-r)^\alpha-(t-r)^\alpha\big]\dr,
\end{equation}
for which we check that
\begin{align*}
\int_{t_{i-1}}^{t_{i}} K(t_i,r)\big((t_i-r)^\alpha-(t-r)^\alpha\big)\dr
& \le \int_{t_{i-1}}^{t_{i}} (t_i-r)^{H-\half+\alpha} \dr
& \lesssim \Delta^{H+\half+\alpha},\\
\int_{t_{i-2}}^{t_{i-1}} K(t_i,r)\big((t_i-r)^\alpha-(t-r)^\alpha\big)\dr
& \le \int_{t_{i-2}}^{t_{i-1}} (t_i-r)^{H-\half+\alpha} \dr
&\lesssim \Delta^{H+\half+\alpha}.
\end{align*}
Then we recall that for $r\in[t_{\ell},t_{\ell+1})$ and $\ell \le i-2$, Lemma~\ref{lemma:bound_deltaK} yields
$$
\abs{(t_i-r)^\alpha-(t-r)^\alpha} 
\lesssim \Delta^{\alpha} (i-\ell-1)^{\alpha-1}.
$$
Combined with~$K(t_i,r)\le K(t_i,t_{\ell+1})= \Delta^{H-\half} (i-\ell-1)^{H-\half}$, it entails 
\begin{align*}
     \sum_{\ell=0}^{i-3} 
 \int_{t_\ell}^{t_{\ell+1}} K(t_i,r)\big((t_i-r)^\alpha-(t-r)^\alpha\big)\dr 
 &\lesssim \Delta^{H+\half+\alpha} \sum_{\ell=0}^{i-3} 
  (i-\ell-1)^{H-\halft+\alpha}\lesssim \diamondsuit_\alpha(\Delta).
\end{align*}
For the last inequality, we used the fact that~$i-1\le N =\Delta^{-1}$ and that
\begin{align}\label{eq:sum_iell}
    \sum_{\ell=0}^{i-3} 
  (i-\ell-1)^{H-\halft+\alpha}
  & = \sum_{\ell=2}^{i-1} \ell^{H-\halft+\alpha} \int_{\ell-1}^{\ell}\ds\nonumber\\
  &\le  \sum_{\ell=2}^{i-1} \int_{\ell-1}^{\ell} s^{H-\halft+\alpha}\ds \nonumber\\
  &\lesssim \log(i-1)\one_{\alpha+H=\half} + (i-1)^{H-\half+\alpha}\one_{\alpha+H>\half}+\one_{\alpha+H<\half}.
\end{align}
\end{proof}


\begin{proof}[Proof of Lemma~\ref{lemma:ti_tj}]
    We split the integral into three parts that we study separately:
\begin{align}\label{eq:int_over}
 & \int_0^{t_i} (t_i-r)^\alpha \abs{K(t_i,r)-K(t_j,r)}\dr\\
        & = \int_0^{t_{j-1}} (t_i-r)^\alpha (K(t_j,r)-K(t_i,r))\dr\\
        & \qquad + \int_{t_{j-1}}^{t_j} (t_i-r)^\alpha (K(t_j,r)-K(t_i,r))\dr + \int_{t_j}^{t_i} (t_i-r)^\alpha K(t_i,r)\dr\nonumber\\
        & =: I_{[0,j-1]} + I_{[j-1,j]} + I_{[j,i]}. \nonumber
    \end{align}
    Regarding~$I_{[j,i]}$, an explicit integration yields the upper bound
    \begin{align}\label{eq_bound_C}
        I_{[j,i]} 
        = \int_{t_j}^{t_i} (t_i-r)^{\alpha+H-\half} \dr 
        & = \frac{\Delta^{\alpha+H+\half}}{\alpha+H+\half}(i-j)^{\alpha+H+\half}\nonumber\\
        & \lesssim \Delta^{\alpha+H+\half}(i-j)^{\alpha+H+\half}.
    \end{align}
    For $I_{[j-1,j]}$, we proceed bounding separately the two components in it, obtaining
    \begin{align}\label{eq_bound_B}
        I_{[j-1,j]}
        & = \int_{t_{j-1}}^{t_j} (t_i-r)^\alpha K(t_j,r)\dr - \int_{t_{j-1}}^{t_j} (t_i-r)^{\alpha+H-\half} \dr \\ 
        & 
        \lesssim (t_i-t_j)^\alpha\int_{t_{j-1}}^{t_j}  K(t_j,r)\dr\lesssim (i-j)^\alpha \Delta^{\alpha+H+\half}
        \lesssim (i-j)^{\alpha+H+\half}\Delta^{\alpha+H+\half}\nonumber.
    \end{align}
    Without loss of generality we now assume $j\ge 2$. 
    Note that $I_{[0,j-1]}$ can be bounded above as
    \begin{equation}\label{eq_bound_A_1}
        I_{[0,j-1]}
        \le (t_i-t_{j-1})^\alpha \int_0^{t_{j-1}} (K(t_j,r)-K(t_i,r))\dr
        =: \Delta^{\alpha}(i-(j-1))^\alpha \widetilde{I}_{[0,j-1]},
    \end{equation}
    with $\widetilde{I}_{[0,j-1]}:=\int_0^{t_{j-1}} (K(t_j,r)-K(t_i,r))\dr = \sum_{\ell= 0}^{j-2}\int_{t_{\ell}}^{t_{\ell+1}} (K(t_j,r)-K(t_i,r))\dr$.
    For each term in the sum, $r \in [t_{\ell},t_{\ell+1})$, we proceed as in the proof of Lemma~\ref{lemma:bound_KDeltaK} and obtain the following upper bound (independent of~$r$):
\begin{align*}
        K(t_j,r) - K(t_i,r)
         \leq K(t_j,t_{\ell+1}) - K(t_i,t_{\ell})
        = \Delta^{H-\half}\left(
        (j -(\ell+1))^{H-\half} - (i-\ell)^{H-\half}\right),
\end{align*}
as $K(t,\cdot)$ is increasing for $H-\half<0$, so that
\begin{align*}
    \widetilde{I}_{[0,j-1]}
    & = \sum_{\ell= 0}^{j-2}\int_{t_{\ell}}^{t_{\ell+1}} (K(t_j,r)-K(t_i,r))\dr\\
    & \leq \Delta^{H-\half} \sum_{\ell= 0}^{j-2}\left((j -(\ell+1))^{H-\half} - (i-\ell)^{H-\half}\right)\int_{t_{\ell}}^{t_{\ell+1}} \dr\\
    & \leq \Delta^{H+\half} \sum_{\ell= 0}^{j-2}\left((j -(\ell+1))^{H-\half} - (i-\ell)^{H-\half}\right)\\
    & = \Delta^{H+\half}
    \left(\sum_{k=1}^{j-1}k^{H-\half} - \sum_{k=i-(j-2)}^{i} k^{H-\half}\right).
\end{align*}
If $j-1 \leq i-(j-1)$, then, ignoring the positive second term, we write, using~\eqref{eq:BoundIneqSum}
\begin{align*}
\widetilde{I}_{[0,j-1]}
\leq \Delta^{H+\half}\sum_{k= 1}^{j-1}k^{H-\half}
\leq \Delta^{H+\half}\sum_{k= 1}^{i-(j-1)}k^{H-\half}
& \lesssim \Delta^{H+\half}\int_{0}^{i-(j-1)}x^{H-\half}\D x \\
& \lesssim \Delta^{H+\half}(i-j)^{H+\half}.
\end{align*}
Now, if $j-1 > i-(j-1)$, then the two sums overlap, so that, using~\eqref{eq:BoundIneqSum}, we obtain
\begin{align*}
    \widetilde{I}_{[0,j-1]}
    & \leq \Delta^{H+\half} \left(\sum_{k= 1}^{i-(j-2)}k^{H-\half} -  \sum_{k=j-1}^{i-1} k^{H-\half}\right)\\
    & \leq \Delta^{H+\half} \sum_{k= 1}^{i-(j-2)}k^{H-\half} \\
    & \lesssim \Delta^{H+\half}\int_{0}^{i-(j-2)}x^{H-\half}\D x\lesssim \Delta^{H+\half}(i-j)^{H+\half}.
\end{align*}
Combining these two cases with~\eqref{eq_bound_A_1}, we thus obtain
    \begin{align}\label{eq_bound_A_2}
        I_{[0,j-1]}
        \lesssim \Delta^{\alpha + H+\half}(i-(j-1))^\alpha (i-j)^{H+\half}
        \lesssim \Delta^{\alpha + H+\half} (i-j)^{\alpha+H+\half}.
    \end{align}
    The lemma follows from plugging the upper bounds in~\eqref{eq_bound_C}-\eqref{eq_bound_B}-\eqref{eq_bound_A_2} into~\eqref{eq:int_over}
$$
\int_0^{t_i} (t_i-r)^\alpha \abs{K(t_i,r)-K(t_j,r)}\dr
\lesssim \Delta^{\alpha + H+\half} (i-j)^{\alpha+H+\half}.
$$
\end{proof}


\begin{proof}[Proof of Lemma~\ref{lemma:Bound_kappar_DeltaK}]
For~$r\in[t_j,t_{j+1})$, $t_i-\kappa_r=t_i-t_j=\Delta(i-j)$. 
    By Lemma~\ref{lemma:bound_deltaK},
\begin{align*}
& \int_0^{t_i} (t_i-\kappa_r)^{2H} \abs{K(t,r)-K(t_i,r)} \dr \\
    &= \sum_{j=0}^{i-1} (i-j)^{2H}\Delta^{2H}\int_{t_j}^{t_{j+1}} \abs{K(t,r)-K(t_i,r)} \dr \\
    &\lesssim \sum_{j=0}^{i-2} (i-j)^{2H}\Delta^{3H+\half} (i-j-1)^{H-\halft} + \Delta^{2H} \int_{t_{i-1}}^{t_i} \abs{K(t,r)-K(t_i,r)} \dr.
\end{align*}
The last integral is dominated by~$\Delta^{H+\half}$. Furthermore, the simple observation
$$
(i-j)^{2H}(i-j-1)^{H-\halft}=\left(\frac{i-j}{i-j-1}\right)^{\halft-H} (i-j)^{3H-\halft}\le 2^{\halft-H} (i-j)^{3H-\halft}.
$$
allows us to compute, similarly to~\eqref{eq:sum_iell} with~$\alpha=2H$,
\begin{align}\label{eq:sum_ij} 
     \sum_{j=0}^{i-2} (i-j)^{2H} (i-j-1)^{H-\halft} 
     & \le 4\sum_{j=0}^{i-2} (i-j)^{3H-\halft}\\
    & \lesssim \left( \frac{i^{3H-\half} -1}{3H-\half}\one_{H\neq \halfs} + \log(i)\one_{H=\halfs}\right)\\
    &\lesssim  \Delta^{\half-3H} \one_{H>\halfs} +  \one_{H<\halfs} +|\log(\Delta)|\one_{H=\halfs}.\nonumber
\end{align}
\end{proof}

\begin{proof}[Proof of Lemma~\ref{lemma:Gassiat_log}]
Similarly to the previous proof, noting that for any~$r\in[t_j,t_{j+1}]$ with ~$j\le i-2$, we have~$(t-r)^{2H}\le (t_{i+1}-t_j)^{2H}\le (2(t_i-t_j))^{2H}$, it yields
\begin{align*}
    \int_0^t (t-r)^{2H}\abs{K(t,r)-K(t_i,r)}\dr 
    \lesssim \Delta^{3H+\half} \sum_{j=0}^{i-2}(i-j)^{3H- {3/2}} + \Delta^{2H}\int_{t_{i-1}}^{t} \abs{K(t,r)-K(t_i,r)}\dr.
\end{align*}
The claim follows from~\eqref{eq:sum_ij}.
\end{proof}

\begin{proof}[Proof of Lemma~\ref{lemma:doublesum}]

For $\ell=j-1$ then $ \int_{t_\ell}^{t_{\ell+1}} K(t_j,r) \dr=\frac{1}{H+\half}\Delta^{H+\half}$ and similarly to~\eqref{eq:sum_ij}, $\sum_{j=1}^{i-2} (i-j-1)^{H-\half} (i-j)^{H-\halft} \lesssim 1$.
If $\ell<j-1$, then $\int_{t_\ell}^{t_{\ell+1}} K(t_j,r)\dr \le  (j-\ell-1)^{H-\half} \Delta^{H+\half}$. 
We plug this in the rest of the sum in Lemma~\ref{lemma:doublesum}, we dispatch the $H-\halft$ exponent in the following way and exploit $(i-\ell-1)^{-\half}\le(i-j-1)^{-\half}$:
\begin{align*}
    S'
    &:=\sum_{j=1}^{i-2} (i-j-1)^{H-\half}\sum_{\ell=0}^{j-2} (i-\ell-1)^{H-\halft}  (j-\ell-1)^{H-\half}\\
    & \le  
    \sum_{j=1}^{i-2} (i-j-1)^{H-1}\sum_{\ell=0}^{j-2} (i-\ell-1)^{H-1}  (j-\ell-1)^{H-\half}\\
    &=  \sum_{\ell=0}^{i-4} (i-\ell-1)^{H-1}\sum_{j=\ell+2}^{i-2} (i-j-1)^{H-1}  (j-\ell-1)^{H-\half},
\end{align*}
where we swapped the sums in the last line.
Furthermore,
\begin{equation}\label{eq:sum_Beta_function}
\begin{aligned}
     \sum_{j=\ell+2}^{i-2}
     (i-j-1)^{H-1}  (j-\ell-1)^{H-\half}
     & \le \sum_{j=\ell+2}^{i-2} \int_j^{j+1} (i-s-1)^{H-1}  (s-\ell-2)^{H-\half}\ds\\
     & = \int_{\ell+2}^{i-1} (i-s-1)^{H-1}  (s-(\ell+2))^{H-\half}\ds \\
     & = (i-\ell-3)^{2H-\half} \B\left(H,H+\half\right),
\end{aligned}
\end{equation}
since $\int_0^{x} (x-y)^{a-1} y^{b-1}\D y = x^{a+b-1} \B(a,b)$ for $a,b>-1$, and with~$\B$ the Beta function.
Moreover, since $(i-\ell-1)^{H-1}\le (i-\ell-3)^{H-1}$, similarly to~\eqref{eq:sum_iell} we obtain 
\begin{align*}
    S'\lesssim\sum_{\ell=0}^{i-4} (i-\ell-3)^{3H-\halft}
    = 1+ \int_1^{i-3} s^{3H-\halft}\ds 
    \lesssim 
	\left\{
	\begin{array}{ll}
	\displaystyle 1, & \text{if } H<\halfs,\\
	\displaystyle \log(i-3), & \text{if } H=\halfs,\\
\displaystyle (i-3)^{3H-\half}, & \text{if } H>\halfs.
	\end{array}
	\right.
\end{align*}
Note that $i-3\le N =\frac{1}{\Delta}$ and $S\lesssim\Delta^{2H} \left(\Delta^{H+\half}S'+\Delta^{H+\half}\right)$, concluding the proof.
\end{proof}

 
\section{Proof of Proposition~\ref{prop:Bone}}\label{sec:Bone}
We fix~$i\le N$ and $t\in[t_i,t_{i+1})$ and recall that~$H\in(0,\half)$. 
To express the differences in $\frB_1(t)$ and~$\frB_2(t)$ without invoking the strong rate of convergence (that is without absolute values), we rely on a combination of the Clark-Ocone and the integration by parts formulae that allows us to decouple the stochastic terms from the kernels and saves us from the immediate use of crude estimates such as the Cauchy-Schwarz inequality. For some of the terms in the following computations, several applications of this strategy are necessary.
Let us begin with the Clark-Ocone formula:
\begin{align*}
    \psi(V_t)^2-\psi(V_{t_i})^2
    &= \EE\Big[\psi(V_t)^2-\psi(V_{t_i})^2\Big] 
    + \int_0^t \Df^W_r\EE_r\Big[\psi(V_t)^2-\psi(V_{t_i})^2\Big] \D W_r.
\end{align*}
Note $\Uu_r:=\Df^W_r (\partial_x^2\bar{u}_t)$ and $\varphi_r(t)=\EE_r[(\psi^2)'(V_t)]$.
Integration by parts then yields
\begin{align}\label{eq:def_Xis}
    \frB_1(t)
    = & \EE\left[\partial_x^2\bar{u}_t\right] \EE\Big[\psi(V_t)^2-\psi(V_{t_i})^2\Big]
    +\EE\left[\int_0^t \Uu_r \EE_r\Big[(\psi^2)'(V_t) K(t,r) - (\psi^2)'(V_{t_i}) K(t_i,r) \Big] \dr\right]
    \nonumber\\ 
    =& \EE\left[\partial_x^2\bar{u}_t\right] \EE\Big[\psi(V_t)^2-\psi(V_{t_i})^2\Big]
    +\EE \left[\int_0^t \Uu_r \varphi_r(t) \big( K(t,r)-K(t_i,r)\big) \dr \right]\nonumber\\
     & \quad + \EE\left[\int_0^t \Uu_r \big(\varphi_r(t)-\varphi_r(t_i)\big) K(t_i,r)\dr\right] \nonumber\\
    = & \EE\left[\partial_x^2\bar{u}_t\right] \EE\Big[\psi(V_t)^2-\psi(V_{t_i})^2\Big]\nonumber \\
    & \quad + \EE \left[\Uu_t \varphi_t(t)\right]\int_0^t \big( K(t,r)-K(t_i,r)\big) \dr\nonumber \\
    &\quad + \EE\left[\Uu_t\int_0^t   \Big( \varphi_r(t) -  \varphi_t(t)\Big) \big( K(t,r)-K(t_i,r)\big) \dr\right]\nonumber\\
     & \quad + \EE\left[  \int_0^t \varphi_r(t) \Big( \Uu_r-  \Uu_t \Big) \big( K(t,r)-K(t_i,r)\big) \dr\right]
    \nonumber\\
    &\quad 
    + \EE\left[\int_0^{t_i} \Uu_r \big(\varphi_r(t)-\varphi_r(t_i)\big) K(t_i,r)\dr\right] \nonumber\\
    & =: \sum_{k=1}^5 \XXi_k.
\end{align}
For future reference, we first  compute, for all~$r\in[0,t]$,
\begin{equation}\label{eq:exp_Uu}
    \Uu_r= \Df^W_r (\partial_x^2\bar{u}_t)
    =\EE_t\left[\phi^{(3)}(\XTb)\left\{\int_r^t \psi'(V_{\kappa_s})K({\kappa_s},r)\D B_s + \int_t^T \psi'(V_{s}) K(s,r)\D B_s + \rho\psi(V_{\kappa_r})\right\}\right].
\end{equation}
The first two terms $\XXi_1,\XXi_2$ are immediate and yield rate one by Lemma~\ref{lemma:trick_rate_one} and the next comment and lemma.
Indeed, the growth of $\phi,\psi$ and their derivatives imply that
$$
\EE\left[\abs{\phi^{(3)}\left(\XTb\right)\left(\int_t^T \psi'(V_{s})K(s,r)\D B_s + \rho\psi(V_{\kappa_r})\right)}^p\right]
$$
is uniformly bounded across~$r\le t\le T$ for all $p>1$.
\begin{lemma}\label{lemma:Xi1}
    $\abs{\XXi_1}+\abs{\XXi_2}  \lesssim t^{2H}-t_i^{2H}$.
\end{lemma}
\begin{proof}
    Thanks to the growth of~$\phi,\psi$ and their derivatives, the estimates of~$\XTb$ and~$V$ in~\eqref{eq:Bounds_XV}, the square integrability of~$K$ over~$[t,T]$, and Lemmas \ref{lemma:ti_tj} and \ref{lemma:Friz} we can conclude
\begin{align}
    \label{eq:Xione} 
    &\abs{\XXi_1} \lesssim \abs{\EE\left[\psi(V_t)^2-\psi(V_{t_i})^2\right]} \lesssim t^{2H}-t_i^{2H},\\
    \label{eq:Xitwo} 
    &\abs{\XXi_2} \lesssim \abs{\int_0^t \big(K(t,r)-K(t_i,r)\big)\dr} \lesssim t^{2H}-t_i^{2H}. \qquad \qedhere
\end{align}
\end{proof}
The other three terms require some tedious, but unambiguous, computations that are developed in the following sections.

\subsection{Bound for $\Xi_3$}
Recall that $\XXi_3$ was defined in~\eqref{eq:def_Xis}.
\begin{proposition}\label{prop:XiThree}
$\abs{\XXi_3} \lesssim  \bigstar(\Delta)$.
\end{proposition}

\begin{proof}
We apply the Clark-Ocone formula with respect to the Brownian motion~$W$:
\begin{equation}\label{eq:CO_varphi}
 \varphi_t(t) -  \varphi_r(t) 
 = \int_r^t \EE_q[\Df_q^W \varphi_q(t)]\D W_q
 = \int_r^t \EE_q[(\psi^2)''(V_t)]K(t,q)\D W_q.    
\end{equation}
Using integration by parts, this implies
\begin{equation}\label{eq:IBP_Uu}
\EE \Big[\Uu_t (\varphi_r(t) -  \varphi_t(t)) \Big]
= -\EE\left[\int_r^t \Df_q^W \Uu_t \EE_q[(\psi^2)''(V_t)]K(t,q)\right]\dq.
\end{equation}

Looking into the integrand, we compute
\begin{align}\label{eq:Xi3_UtDecompose}
    \Df_q^W \Uu_t 
    &= \Df_q^W \EE_t\left[\phi^{(3)}(\XTb)\left(\int_t^T \psi'(V_s)K(s,t)\D B_s + \rho\psi(V_{t_i})\right)\right]\nonumber\\
    &=\EE_t\left[\phi^{(4)}(\XTb)\left(\int_q^T \psi'(V_{\kappa_s^t})K(\kappa_s^t,q)\D B_s 
    + \rho \psi(V_{\kappa_q})\right)\left(\int_t^T \psi'(V_s)K(s,t)\D B_s 
    + \rho\psi(V_{t_i})\right)\right]\nonumber\\
    &\quad + \EE_t\left[\phi^{(3)}(\XTb)\left(\int_t^T \psi''(V_s) K(s,q) K(s,t)\D B_s + \rho\psi'(V_{t_i})K(t_i,q)\right)\right]\nonumber\\
    &=\EE_t\left[\phi^{(4)}(\XTb)\int_q^t \psi'(V_{\kappa_s})K(\kappa_s,q)\D B_s\left(\int_t^T \psi'(V_s)K(s,t)\D B_s + \rho\psi(V_{t_i})\right)\right]\nonumber\\
    &\quad + \EE_t\left[\phi^{(4)}(\XTb) \left(\int_t^T\psi'(V_{s})K(s,t)\D B_s+\rho\psi(V_{\kappa_q})\right)\left(\int_t^T \psi'(V_s)K(s,t)\D B_s + \rho\psi(V_{t_i})\right)\right]\nonumber\\
    &\quad + \EE_t\left[\phi^{(3)}(\XTb)\int_t^T \psi''(V_s) K(s,q) K(s,t)\D B_s\right] +\rho\EE_t\left[\phi^{(3)}(\XTb) \psi'(V_{t_i})K(t_i,q)\right]\nonumber\\
    & =: \sum_{m=1}^4 \Uu_{t,q}^{(m)}.
\end{align}
We thus have
\begin{align*}
\XXi_3
 & = -\EE\left[\Uu_t\int_0^t   \Big( \varphi_r(t) -  \varphi_t(t)\Big) \big( K(t,r)-K(t_i,r)\big) \dr\right]
\\
&= -\int_0^t \left(\int_r^t \EE\Big[\Df_q \Uu_t \EE_q[(\psi^2)''(V_t)]\Big]K(t,q)\dq\right) \Big(K(t,r)-K(t_i,r)\Big)\dr\\
    &=-\sum_{m=1}^4  \int_0^t\left( \int_r^t \EE\left[\Uu_{t,q}^{(m)}\EE_q[(\psi^2)''(V_t)]\right]K(t,q)\dq \right)\Big(K(t,r)-K(t_i,r)\Big)\dr
    =:- \sum_{m=1}^{4}\XXi_{3m}.
\end{align*}
We proceed by bounding each of these terms.
\begin{itemize}
\item[$\mathbf{(\XXi_{31})}$]
By Lemma~\ref{lemma:BoundDiscreteIntegral},
$\abs{\EE\Big[\Uu_{t,q}^{(1)}\EE_q[(\psi^2)''(V_t)]\Big]}
\lesssim (1 + K(\kappa_q+\Delta,q)\sqrt{\Delta})\one_{q<t_i}$.
For~$q<t_i$,
\begin{align*}
    \int_r^t K(\kappa_q+\Delta,q)\sqrt{\Delta} K(t,q)\dq 
    & = \sqrt{\Delta}\left[ \sum_{j=\tau(r)}^{i-1} \int_{t_{j-1}\vee r}^{t_{j}} K(t_{j},q)K(t,q)\dq + \int_{t_{i-1}}^{t_i} K(t,q)K(t_i,q)\dq
    \right]\\
    & \le \sqrt{\Delta}\left( 
    \sum_{j=\ell}^{i-1}K(t,t_{j}) \int_{t_{j-1}\vee r}^{t_{j}} K(t_{j},q)\dq +\Delta^{2H}\right)\\
    & \le \sqrt{\Delta}\left(\Delta^{H+\half} 
    \sum_{j=\ell}^{i-1}\Delta^{H-\half} (i-j)^{H-\half}  +2\Delta^{2H}\right),
\end{align*}
since $K(t,t_j)\le K(t_i,t_j)=\big(\Delta(i-j) \big)^{H-\half}$. 
Thus, using $t_{\ell-1}\le r < t_\ell$, we obtain
\begin{align}\label{eq:sum_ij_ell}
\sum_{j=\ell}^{i-1} (i-j)^{H-\half} = \sum_{j=1}^{i-\ell}j^{H-\half} \le \int_0^{i-\ell} s^{H-\half}\ds 
\lesssim (i-\ell)^{H+\half}
& \lesssim \left(\frac{t_i-t_\ell}
{\Delta}\right)^{H+\half}\nonumber\\
& \lesssim \left(\frac{t_i-r}{\Delta}\right)^{H+\half},
\end{align}
and hence
\begin{align}\label{eq:Xi31_1}
    \int_r^t K(\kappa_q+\Delta,q)\sqrt{\Delta} K(t,q)\dq  \lesssim \Delta^H (t_i-r)^{H+\half} + \Delta^{2H+\half}. 
\end{align}

\item[$\mathbf{(\XXi_{32})}$]
We note that $\Uu_t^{(2)}$ is bounded in~$L^2(\Omega)$, hence~$\EE\Big[\Uu_{t,q}^{(2)}\EE_q[(\psi^2)''(V_t)]\Big]\lesssim 1$. \\
\item[$\mathbf{(\XXi_{33})}$]
By Cauchy-Schwarz' and Jensen's inequality, we obtain
\begin{align*}
\EE\Big[\Uu_{t,q}^{(3)}\EE_q[(\psi^2)''(V_t)]\Big]^2
    &=\EE\left[\EE_t\left[\phi^{(3)}(\XTb)\int_t^T \psi''(V_s) K(s,q) K(s,t)\D B_s \right] \EE_q[(\psi^2)''(V_t)]\right]^2 \\
    &\le \EE\left[\phi^{(3)}(\XTb)^2(\psi^2)''(V_t)^2\right] \EE\left[\left(\int_t^T \psi''(V_s) K(s,q) K(s,t)\D B_s \right)^2\right] \\
    &\lesssim \int_t^T K(s,q)^2 K(s,t)^2\ds \le K(t,q)^2 \frac{(T-t)^{2H}}{2H}
    \lesssim K(t,q)^2.
\end{align*}
\item[$\mathbf{(\XXi_{34})}$]
 Similarly we write $\EE\Big[\Uu_{t,q}^{(4)}\EE_q[(\psi^2)''(V_t)]\Big]^2\lesssim \EE\big[\lvert\Uu_{t,q}^{(4)}\lvert^2\big]\lesssim K(t_i,q)^2$.
\end{itemize}
We then regroup all those estimates together.
By~\eqref{eq:IBP_Uu}-\eqref{eq:Xi3_UtDecompose}, 
the upper bound~\eqref{eq:Xi31_1} and the above estimates, we obtain
\begin{align*}
\abs{\EE \Big[\Uu_t (\varphi_r(t) -  \varphi_t(t)) \Big]}
&\le  \int_r^t \EE\Big[\Big(\sum_{m=1}^{4}\Uu^{(m)}_{t,q}\Big) \EE_q[(\psi^2)''(V_t)]\Big]K(t,q)\dq\\
&\lesssim \int_r^t \left(1+K(\kappa_q+\Delta,q)\sqrt{\Delta} + K(t,q) + K(t_i,q) \right)K(t,q)\dq\\
&\lesssim (t-r)^{2H} + (t_i-r)^{2H}\one_{r<t_i}.
\end{align*}
To conclude, we notice
\begin{align*}
\int_0^{t_i} (t_i-r)^{2H} \Big|K(t,r)-K(t_i,r)\Big| \dr
& \le \int_0^{t_i} (t-r)^{2H} \Big|K(t,r)-K(t_i,r)\Big| \dr\\
& \le \int_0^{t} (t-r)^{2H} \Big|K(t,r)-K(t_i,r)\Big| \dr,
\end{align*}
such that we obtain 
\begin{align*}
\abs{\XXi_3}
& \lesssim \int_0^t (t-r)^{2H} \Big|K(t,r)-K(t_i,r)\Big| \dr +  \int_0^{t_i} (t_i-r)^{2H} \Big|K(t,r)-K(t_i,r)\Big| \dr\\
& \le 2\int_0^t (t-r)^{2H} \Big|K(t,r)-K(t_i,r)\Big| \dr
\lesssim \bigstar(\Delta),
\end{align*}
using Lemma~\ref{lemma:Gassiat_log}, which concludes the proof.
\end{proof}

\subsection{Bound for $\Xi_4$}
Recall the definition of~$\XXi_4$ from~\eqref{eq:def_Xis}.
\begin{proposition}\label{prop:Xi4}
$|\XXi_4|\lesssim 
\bigstar(\Delta).$
\end{proposition}
\begin{proof}
We recall the expression of~$\Uu_r$ given in~\eqref{eq:exp_Uu}, which yields
\begin{align*}
&     \Uu_r - \Uu_t \\
&     =\EE_t\left[\phi^{(3)}(\XTb)\left\{\int_r^t \psi'(V_{\kappa_s})K({\kappa_s},r)\D B_s + \int_t^T \psi'(V_{s}) \big(K(s,r)-K(s,t)\big)\D B_s + \rho[\psi(V_{\kappa_r})-\psi(V_{t_i})]\right\}\right].
\end{align*}
This expression requires integration by parts because Cauchy-Schwarz yields a suboptimal rate.
We divide and conquer in the following way:
\begin{align*}
    \XXi_4 &= \EE\left[\int_0^t \varphi_r(t) \Big( \Uu_r-  \Uu_t \Big) \big( K(t,r)-K(t_i,r)\big) \dr\right] \\
    &=  \EE\left[\int_0^t \varphi_r(t)\EE_t\left[\phi^{(3)}(\XTb)\left( \int_r^t \psi'(V_{\kappa_s})K({\kappa_s},r)\D B_s \right)\right] \big( K(t,r)-K(t_i,r)\big) \dr\right]\\
    &\quad + \EE\left[\int_0^t \varphi_r(t)\EE_t\left[\phi^{(3)}(\XTb) \left(\int_t^T \psi'(V_{s}) \big(K(s,r)-K(s,t)\big)\D B_s\right)\right]\big( K(t,r)-K(t_i,r)\big) \dr\right]\\
    &\quad + \rho\EE\left[\int_0^t \varphi_r(t) \EE_t\left[\phi^{(3)}(\XTb)\Big(\psi(V_{\kappa_r})-\psi(V_{t_i})\Big)\right]\big( K(t,r)-K(t_i,r)\big) \dr\right]\\
    &=: \int_0^t \EE\Big[ \XXi_{41}(r) + \XXi_{42}(r) + \XXi_{43}(r) \Big] \big( K(t,r)-K(t_i,r)\big) \dr.
\end{align*}
The bound then follows from
Propositions~\ref{prop:Xi41_Int}-\ref{prop:Xi42_Int}-\ref{prop:Xi43_Int}, noting that
$\Delta^{2H+1}\lesssim\bigstar(\Delta)$.
\end{proof}
In spirit, the proofs show that 
\begin{equation*}
\abs{\EE[\varphi_r(t)(\Uu_r - \Uu_t)]}\lesssim (t_i-\kappa_r)^{2H} \one_{r<t_i} + (t-r)^{2H},
\end{equation*}
which is sufficient to conclude by Lemmas~\ref{lemma:Bound_kappar_DeltaK} and~\ref{lemma:Gassiat_log}. 

\subsubsection{Bound for $\Xi_{41}$}

\begin{proposition}\label{prop:Xi41_Int}
$\displaystyle
\abs{\int_0^t \EE\left[\XXi_{41}(r)\right] \big(K(t,r)-K(t_i,r)\big)\dr }
\lesssim \bigstar(\Delta)
$.
\end{proposition}

\begin{proof}
Applying integration by parts, we obtain the following decomposition:
\begin{align*}
     \abs{\EE[\XXi_{41}(r)] }
    &=\abs{\EE\left[\int_r^t \Big(\Df^{B}_s \varphi_r(t) \phi^{(3)}(\XTb) + \varphi_r(t) \Df^{B}_s\phi^{(3)}(\XTb) \Big) \psi'(V_{\kappa_s}) K(\kappa_s,r)\ds\right]} \\
    &\lesssim \abs{\int_r^t \EE\left[\EE_r[(\psi^2)''(V_t)] \phi^{(3)}(\XTb) \psi'(V_{\kappa_s})\right] K(t,s)K(\kappa_s,r)\ds}\\
    &\quad+ \abs{\int_r^t \EE\left[\varphi_r(t) \phi^{(4)}(\XTb) \left(\int_s^t\psi'(V_{\kappa_q})K(\kappa_q,s)\D B_q\right) \psi'(V_{\kappa_s}) \right] K(\kappa_s,r)\ds }\\
    &\quad +  \abs{\int_r^t \EE\left[\varphi_r(t) \phi^{(4)}(\XTb) \left(\int_t^T\psi'(V_{q})K(q,s)\D B_q+\psi(V_{\kappa_s})\right) \psi'(V_{\kappa_s}) \right] K(\kappa_s,r)\ds }\\
    &\lesssim \int_r^t  K(t,s)K(\kappa_s,r)\ds 
    + \int_r^t K(\kappa_s+\Delta,s)\sqrt{\Delta}K(\kappa_s,r)\ds
    + \int_r^t K(\kappa_s,r)\ds ,
\end{align*}
where we used Lemma~\ref{lemma:BoundDiscreteIntegral}.
Since $1\lesssim K(t,s)$ for all $s>0$,
we write
\begin{align*}
&\abs{\int_0^t \EE\left[\XXi_{41}(r)\right] \big(K(t,r)-K(t_i,r)\big)\dr }\\
 &\lesssim \int_0^t 
\left[\int_r^t  K(t,s)K(\kappa_s,r)\ds
 +\int_r^t K(\kappa_s+\Delta,s)\sqrt{\Delta}K(\kappa_s,r)\ds \right]\abs{K(t,r)-K(t_i,r)}\dr \\
& =: \XXi_{411} + \XXi_{412}.
\end{align*}
The proof then follows from
Lemma~\ref{lem:Xi411} and Lemma~\ref{lem:Xi412} below,
since $\Delta\lesssim\bigstar(\Delta)$.
\end{proof}
\begin{lemma}\label{lem:Xi411}
$\XXi_{411}\lesssim \bigstar(\Delta)$.
\end{lemma}
\begin{lemma}\label{lem:Xi412}
$\XXi_{412}\lesssim \Delta$.
\end{lemma}

\begin{proof}[Proof of Lemma~\ref{lem:Xi411}]
If $r\in [t_i,t)$ then $\int_r^t  K(t,s)K(\kappa_s,r)\ds=0$. 
For $r\in[0,t_i)$, $s\in [t_j,t_{j+1})$ with~$j\le i-2$, it holds $K(t,s)\le (t_i-t_{j+1})^{H-\half}=\Delta^{H-\half}(i-j-1)^{H-\half}$; thus
\begin{align}
& \int_r^t  K(t,s)K(\kappa_s,r)\ds\\
    &=\sum_{j=\tau(r)}^{i-2} \int_{t_j}^{t_{j+1}} K(t,s)K(t_j,r)\ds 
    + \int_{t_{i-1}}^{t_{i}} K(t,s) K(t_{i-1},r) \ds
    +\int_{t_{i}}^{t} K(t,s) K(t_{i},r) \ds
    \nonumber \\
    &\lesssim \Delta^{H+\half} \left(\sum_{j=\tau(r)}^{i-2} K(t_j,r) (i-j-1)^{H-\half}\right) + \Delta^{H+\half}\Big( K(t_{i-1},r)+K(t_i,r)\Big) \nonumber\\
    &=:\XXi_{4111}(r) + \XXi_{4112}(r).
    \label{eq:Xi411_decomp}
\end{align}
For the last term on the right-hand side of the equation above, note that, for $j=i$,
$\int_{0}^{t_{i}} K(t_{i},r)\big(K(t_i,r)-K(t,r)\big)\dr
\lesssim \Delta^{2H}$
by Lemma~\ref{lemma:bound_KDeltaK} (and
similarly $j=i-1$). 
Thus
$
\int_0^t \XXi_{4112}(r) \big(K(t,r)-K(t_i,r)\big)\dr \lesssim \Delta^{3H+\half}$.
Now, we assume~$i\ge3$. 
Plugging $\XXi_{4111}$ in the integral $\XXi_{411}$, we obtain
\begin{align*}
&     \int_0^t \XXi_{4111}(r) \abs{K(t,r)-K(t_i,r)}\dr \\
    &\lesssim  \Delta^{H+\half} \int_0^{t_i}  \sum_{j=\tau(r)}^{i-2} (i-j-1)^{H-\half} K(t_j,r)\big(K(t_i,r)-K(t,r)\big)\dr\\
    &\lesssim \Delta^{H+\half} \sum_{\ell=0}^{i-1}\int_{t_\ell}^{t_{\ell+1}}  \sum_{j=\ell+1}^{i-2} (i-j-1)^{H-\half}K(t_j,r)\big(K(t_i,r)-K(t,r)\big)\dr\\
    &=\Delta^{H+\half} \sum_{j=1}^{i-2} (i-j-1)^{H-\half}\sum_{\ell=0}^{j-1} 
     \int_{t_\ell}^{t_{\ell+1}} K(t_j,r)\big(K(t_i,r)-K(t,r)\big)\dr\\
     &\le \Delta^{2H}  \sum_{j=1}^{i-2} (i-j-1)^{H-\half}\sum_{\ell=0}^{j-1} (i-\ell-1)^{H-\halft}
     \int_{t_\ell}^{t_{\ell+1}} K(t_j,r) \dr
     \lesssim \bigstar(\Delta),
\end{align*}
where the second-to-last inequality follows from Lemma~\ref{lemma:bound_deltaK} and the last one from
Lemma~\ref{lemma:doublesum}.
\end{proof}

\begin{proof}[Proof of Lemma~\ref{lem:Xi412}]
For any $r\in[0,t_i)$, we have
\begin{align*}
    \int_r^t K(\kappa_s+\Delta,s)\sqrt{\Delta}K(\kappa_s,r)\ds
    & = \sum_{j=\tau(r)}^{i} \int_{t_j}^{t_{j+1}\wedge t} K(t_{j+1},s)\sqrt{\Delta} K(t_j,r)\ds \\
& \lesssim \Delta^{1+H} \sum_{j=\tau(r)}^{i}  K(t_j,r). 
\end{align*}
Integrating the last two terms of the sum yields
\begin{align*}
    \int_0^t K(t_{i-1},r) \big(K(t_i,r)-K(t,r)\big)\dr +\int_0^t K(t_i,r) \big(K(t_i,r)-K(t,r)\big)\dr \lesssim t_i^{2H}.
\end{align*}
Since $K(\kappa_s,r)=0$ for $r\in[t_i,t)$ and $s\in[r,t)$, it holds
\begin{align*}
    \XXi_{412}
    &\lesssim \int_0^{t_i} \Delta^{H+1}\sum_{j=\tau(r)}^{i-2} K(t_j,r) \abs{K(t,r)-K(t_i,r)}  \dr + \Delta^{H+1} \\
    &\lesssim \Delta^{H+1}\sum_{\ell=0}^{i-1} \int_{t_{\ell}}^{t_{\ell+1}} \sum_{j=\tau(r)}^{i-2} K(t_j,r) \big(K(t_i,r)-K(t,r)\big) \dr+ \Delta^{H+1}\\
    &\lesssim \Delta^{H+1}\sum_{\ell=0}^{i-1} \int_{t_{\ell}}^{t_{\ell+1}} \sum_{j=\ell+1}^{i-2} K(t_j,r) \big(K(t_i,r)-K(t,r)\big) \dr+ \Delta^{H+1}\\
    &\lesssim \Delta^{H+1}\sum_{\ell=0}^{i-3} \int_{t_{\ell}}^{t_{\ell+1}} \sum_{j=\ell+1}^{i-2} K(t_j,r) \big(K(t_i,r)-K(t,r)\big) \dr+ \Delta^{H+1},
\end{align*}
because for $r\in[t_\ell,t_{\ell+1})$, $\tau(r)=\ell+1$ and hence the inner sum is empty if~$\ell\ge i-2$. We continue by separating the term $j=\ell+1$, invoking Lemma~\ref{lemma:bound_deltaK} and~$K(t_j,r)\le K(t_j,t_{\ell+1})$:
\begin{align*}
    \XXi_{412}
    &\lesssim \Delta^{H+1}\sum_{\ell=0}^{i-3} \Big\{\int_{t_{\ell}}^{t_{\ell+1}} \sum_{j=\ell+2}^{i-2} K(t_j,r) \big(K(t_i,r)-K(t,r)\big) \dr\\
 &    \qquad\qquad\qquad +  \int_{t_{\ell}}^{t_{\ell+1}}  K(t_{\ell+1},r) \big(K(t_i,r)-K(t,r)\big) \dr\Big\}
    +\Delta^{H+1}\\
    &\lesssim \Delta^{H+1}\sum_{\ell=0}^{i-3} \int_{t_{\ell}}^{t_{\ell+1}} \sum_{j=\ell+2}^{i-2} \Delta^{H-\half} (j-\ell-1)^{H-\half} \Delta^{H-\half} (i-\ell-1)^{H-\halft}\dr\\
    &\qquad\qquad\qquad + \Delta^{H+1}\sum_{\ell=0}^{i-3}  \Delta^{H-\half} (i-\ell-1)^{H-\halft}\int_{t_{\ell}}^{t_{\ell+1}} K(t_{\ell+1},r)\dr 
    + \Delta^{H+1}\\
    &= \Delta^{3H+1}\sum_{\ell=0}^{i-3} \sum_{j=\ell+2}^{i-2} (j-\ell-1)^{H-\half}  (i-\ell-1)^{H-\halft}+\Delta^{3H+1} \sum_{\ell=0}^{i-3} (i-\ell-1)^{H-\halft} 
    +\Delta^{H+1}.
\end{align*}
On the one hand, we already saw in~\eqref{eq:sum_iell} that~$\sum_{\ell=0}^{i-3} (i-\ell-1)^{H-\halft}\lesssim 1$. On the other hand, we have, similarly to~\eqref{eq:sum_ij_ell}, that~$\sum_{j=\ell+2}^{i-2} (j-\ell-1)^{H-\half}\lesssim (i-\ell-3)^{H+\half}\le (i-\ell-1)^{H+\half}$ and~$  \sum_{\ell=0}^{i-3}  (i-\ell-1)^{2H-1}\lesssim (i-1)^{2H} \le \Delta^{-2H}$. 
Overall this results in~$\XXi_{412}\lesssim \Delta^{H+1}$.
\end{proof}

\subsubsection{Bound for $\Xi_{42}$}

\begin{proposition}\label{prop:Xi42_Int}
    $\displaystyle
    \int_0^t \abs{\EE[\XXi_{42}(r)]}\abs{K(t,r)-K(t_i,r)}\dr \lesssim \Delta^{2H+1}$.
\end{proposition}
\begin{proof}
This proof heavily replies on IBP formula as well. 
Note that, for $r<t<s$, $\varphi_r(t)$ is $\Ff_r$-measurable and hence~$\Df^{B}_s\varphi_r(t)=\rho \Df^{W}_s\varphi_r(t)=0$.
Thus we obtain
\begin{align*}
    \abs{\EE[\XXi_{42}(r)] }
    &=\abs{\EE\left[ \varphi_r(t) \phi^{(3)}(\XTb) \left(\int_t^T \psi'(V_{s}) \big(K(s,r)-K(s,t)\big)\D B_s\right)\right]}\\
    & = \abs{\EE\left[ \int_t^T \varphi_r(t) \Df^{B}_s \phi^{(3)}(\XTb) \psi'(V_{s}) \big(K(s,r)-K(s,t)\big) \ds \right]}\\
    & \le \sup_{s\in[t,T]} \abs{\EE\left[\varphi_r(t) \Df^{B}_s \phi^{(3)}(\XTb) \psi'(V_{s}) \right]}  \int_t^T  \abs{K(s,r)-K(s,t)} \ds.
\end{align*}
We recall that for all~$s\in[t,T]$, $ \Df^{B}_s \phi^{(3)}(\XTb)$ is bounded in~$L^2(\Omega)$. 
Thus, the prefactor multiplying the integral in the equation above is bounded. 
By Tonelli's theorem we conclude
\begin{align*}
    \abs{\EE[\XXi_{42}(r)]} 
    \lesssim \int_t^T \int_r^t (s-q)^{H-\halft}\dq\ds 
    \lesssim \int_r^t \int_t^T (s-q)^{H-\halft}\ds \dq
    & \lesssim \int_r^t (t-q)^{H-\half} \dq\\
    & \lesssim (t-r)^{H+\half}.
\end{align*}
The claim thus follows from \cite[Lemma 2.1]{gassiat2022weak}.
\end{proof} 

\subsubsection{Bound for $\Xi_{43}$}
\begin{proposition}\label{prop:Xi43_Int}
$\displaystyle 
\int_0^t \abs{\XXi_{43}(r)}\abs{K(t,r)-K(t_i,r)}\dr \lesssim \bigstar(\Delta)$.
\end{proposition}
\begin{proof}
Since $|\rho|\leq 1$, we omit it in the proof.
We show that $\abs{\XXi_{43}(r)}\lesssim (t_i-\kappa_r)^{2H}\one_{r\le t_i}$ for all $r\in[0,t)$. Together with Lemma~\ref{lemma:Bound_kappar_DeltaK}, this yields the claim.
Clark-Ocone formula yields
\begin{align*}
    \psi(V_{\kappa_r})-\psi(V_{t_i}) 
    =\EE[\psi(V_{\kappa_r})-\psi(V_{t_i}) ]+ \int_0^{t_i}\EE_q\Big(\psi'(V_{\kappa_r})K(\kappa_r,q)-\psi'(V_{t_i})K(t_i,q)\Big)\D W_q,
\end{align*}
and Malliavin IBP with respect to $W$ implies
\begin{align*}
    \XXi_{43}(r)
&=\rho\EE\left[\varphi_r(r)\EE_t\left[\phi^{(3)}(\XTb)\right] \right] \EE[\psi(V_{\kappa_r})-\psi(V_{t_i}) ] \\
    &\quad + \rho\EE\Big[\int_0^{t_i} \Big(\Df^W_q\varphi_r(t) \phi^{(3)}(\XTb) + \varphi_r(t) \Df^W_q \phi^{(3)}(\XTb)\Big) \times
    \\
    & \qquad\qquad\times\EE_q\Big(\psi'(V_{\kappa_r})K(\kappa_r,q)-\psi'(V_{t_i})K(t_i,q)\Big)\dq\Big]
    \\
    &= \rho\EE\left[\varphi_r(r)\EE_t\left[\phi^{(3)}(\XTb)\right] \right] \EE[\psi(V_{\kappa_r})-\psi(V_{t_i}) ] \\
    &\quad + \rho\EE\left[\int_0^{t_i} \Df^W_q\varphi_r(t) \phi^{(3)}(\XTb) K(\kappa_r,q)\EE_q\Big(\psi'(V_{\kappa_r})-\psi'(V_{t_i})\Big)\dq\right] \\
    &\quad +  \rho \EE\left[\int_0^{t_i} \Df^W_q\varphi_r(t) \phi^{(3)}(\XTb) \EE_q\psi'(V_{t_i})\Big(K(\kappa_r,q)-K(t_i,q)\Big)\dq\right] \\
    &\quad+  \rho\EE\left[\int_0^{t_i} \varphi_r(t) \Df^W_q \phi^{(3)}(\XTb) K(\kappa_r,q)\EE_q\Big(\psi'(V_{\kappa_r})-\psi'(V_{t_i})\Big)\dq\right] \\
    &\quad + \rho\EE\left[\int_0^{t_i} \varphi_r(t) \Df^W_q \phi^{(3)}(\XTb) \EE_q\psi'(V_{t_i})\Big(K(\kappa_r,q)-K(t_i,q)\Big)\dq\right] \\
    &=:\rho\Big( \XXi_{430}(r) +\XXi_{431}(r) + \XXi_{432}(r) + \XXi_{433}(r) + \XXi_{434}(r)\Big).
\end{align*}
The proof then follows from the lemmas below 
and Lemma~\ref{lemma:Bound_kappar_DeltaK}.
\end{proof}
\begin{lemma}\label{lem:Bound430}
$\abs{\XXi_{430}(r)}\lesssim (t_i-\kappa_r)^{2H}$.
\end{lemma}
\begin{lemma}\label{lem:Bound431}
$\abs{\XXi_{431}(r)}\lesssim (t_i-\kappa_r)^{2H}$.
\end{lemma}
\begin{lemma}\label{lem:Bound432}
$\displaystyle 
\abs{\XXi_{432}(r)}
\lesssim (t_i-\kappa_r)^{2H}
$.
\end{lemma}
\begin{lemma}\label{lem:433}
$
|\XXi_{433}(r) |
\lesssim
(t_i-\kappa_{r})^{2H}$.
\end{lemma}
\begin{lemma}\label{lem:434}
$|\XXi_{434}(r)|\lesssim
 (t_i-\kappa_r)^{H+\half}
$. 
\end{lemma}
\begin{proof}[Proof of Lemma~\ref{lem:Bound430}]
By virtue of Lemma~\ref{lemma:Friz}, 
$$
\abs{\EE[\psi(V_{\kappa_r})-\psi(V_{t_i}) ] }\lesssim t_i^{2H}-\kappa_r^{2H}\lesssim (t_i-\kappa_r)^{2H}.
$$
Since the prefactor is bounded, this yields the claim.
\end{proof}
\begin{proof}[Proof of Lemma~\ref{lem:Bound431}]
We recall that, by Lemma~\ref{lemma:Delta_varphi} we have
\begin{align*}
    \abs{\EE_q\left[\psi'(V_{\kappa_r})-\psi'(V_{t_i})\right]}
    \lesssim \int_{\kappa_r}^{t_i} (v-q)^{2H-1}\D v + \int_{\kappa_r}^{t_i}\abs{\int_0^q (v-x)^{H-\halft}\D W_x}\D v.
\end{align*}
Noticing that~$\Df_q^W \varphi_r(t)=\EE_r[(\psi^{2})^{(2)}(V_t)] K(t,q)$, the first term then yields
\begin{align*}
\abs{\XXi_{431}(r)}    
&\le \int_0^{\kappa_r} \EE\Big[ \EE_r[(\psi^{2})^{(2)}(V_t)] \phi^{(3)}(\XTb) \abs{\EE_q\big[\psi'(V_{\kappa_r})-\psi'(V_{t_i})\big]}\Big] K(t,q)K(\kappa_r,q) \dq\\
&\le \int_0^{\kappa_r} \EE\Big[ \EE_r[(\psi^{2})^{(2)}(V_t)] \phi^{(3)}(\XTb) \Big] \left(  \int_{\kappa_r}^{t_i} (v-q)^{2H-1}\D v \right) K(t,q)K(\kappa_r,q) \dq\\
&\quad + \int_0^{\kappa_r}\left\{
\int_{\kappa_r}^{t_i} \EE\left[ \EE_r[(\psi^{2})^{(2)}(V_t)] \phi^{(3)}(\XTb)  \abs{\int_0^q (v-x)^{H-\halft}\D W_x} \right]\D v \right\} K(t,q)K(\kappa_r,q) \dq\\
&=: \XXi_{4311}(r) + \XXi_{4312}(r).
\end{align*}
Since~$q\le \kappa_r$,
then 
$
\int_{\kappa_r}^{t_i} (v-q)^{2H-1}\D v
\le \int_{\kappa_r}^{t_i} (v-\kappa_r)^{2H-1}\D v
\lesssim (t_i-\kappa_r)^{2H}$,
and thus
\begin{align*}
    \XXi_{4311}(r) \lesssim \int_0^{\kappa_r} K(t,q)K(\kappa_r,q)\dq \,(t_i-\kappa_r)^{2H} \lesssim (t_i-\kappa_r)^{2H}.
\end{align*}
Now, by Cauchy-Schwarz and It\^o isometry, it holds
\begin{align*}
    & \XXi_{4312}(r) \\
    &\le \int_0^{\kappa_r}\left|\int_{\kappa_r}^{t_i} \EE\left[ \left|\EE_r[(\psi^{2})''(V_t)] \phi^{(3)}(\XTb) \right|^2\right]^{\half} \EE\left[\left|\int_0^q (v-x)^{H-\halft}\D W_x\right|^2 \right]^{\half}\D v \right|
    K(t,q)K(\kappa_r,q) \dq  \\
    &\lesssim \int_0^{\kappa_r}\int_{\kappa_r}^{t_i}  \left(\int_0^q (v-x)^{2H-3}\D x\right)^{\half}\D v  K(t,q)K(\kappa_r,q) \dq \\
    & \lesssim  \int_0^{\kappa_r}\left(\int_{\kappa_r}^{t_i}  (v-q)^{H-1}\D v \right) K(t,q)K(\kappa_r,q) \dq\\
    &= \int_0^{\kappa_r}\left(\int_{\kappa_r}^{t_i}  (v-q)^{H-\half} (v-q)^{-\half}\D v\right)  K(t,q)K(\kappa_r,q) \dq\\
    &\le \int_0^{\kappa_r}\left(\int_{\kappa_r}^{t_i}   (v-\kappa_r)^{H-\half}\D v \right) (\kappa_r-q)^{-\half} K(t,q)K(\kappa_r,q) \dq\\
    &\le \int_0^{\kappa_r} (t_i-\kappa_r)^{H+\half} K(t,q) (\kappa_r-q)^{H-1}\dq 
    \lesssim (t_i-\kappa_r)^{2H},
\end{align*}
where we used $K(t,q)\le K(t_i,\kappa_r)$ in the last inequality.
\end{proof}

\begin{proof}[Proof of Lemma~\ref{lem:Bound432}]
Exploiting Lemma~\ref{lemma:ti_tj} with $\alpha=H-\half$, we write
\begin{align*}
    \abs{\XXi_{432}(r)}
    &\le  \int_0^{t_i}\abs{\EE\Big[ \EE_r[(\psi^{2})^{(2)}(V_t)] \phi^{(3)}(\XTb) \EE_q\psi'(V_{t_i})\Big]} K(t,q)\abs{K(\kappa_r,q)-K(t_i,q)} \dq\\
    &\lesssim \int_0^{t_i} K(t_i,q)\abs{K(\kappa_r,q)-K(t_i,q)} \dq
    \le (t_i-\kappa_r)^{2H}.
\end{align*}
\end{proof}
\begin{proof}[Proof of  {Lemma}~\ref{lem:433}]
Let us now move on to discuss 
\begin{align}
    \XXi_{433}(r) 
    & = \EE\left[\int_0^{\kappa_r} \varphi_r(t) \Df_q^W \phi^{(3)}(\XTb) K(\kappa_r,q)\EE_q\Big[\psi'(V_{\kappa_r})-\psi'(V_{t_i})\Big]\dq\right]\\ \nonumber
    & 
    = \EE\left[\int_0^{\kappa_r} \varphi_r(t) \phi^{(4)}(\XTb) \Df_q^W \XTb K(\kappa_r,q)\EE_q\Big[\psi'(V_{\kappa_r})-\psi'(V_{t_i})\Big]\dq\right]\\ \nonumber
    & 
    = \EE\Bigg[\int_0^{\kappa_r} \varphi_r(t) \phi^{(4)}(\XTb) \Bigg\{\int_{\kappa_q +\Delta}^t \psi'(\kappa_s)K(\kappa_s,q)\D B_s + \int_{t}^T \psi'(V_{\kappa_s})K(\kappa_s,q)\D B_s \\ \nonumber
    & \qquad  + \rho \psi (V_{\kappa_q})\Bigg\} K(\kappa_r,q)\EE_q\Big[\psi'(V_{\kappa_r})-\psi'(V_{t_i})\Big]\dq\Bigg].
\end{align}
Let $\Gg_q:=\sup_{v\in\TT,n=2,3} \EE_q[\psi^{(n)}(V_v)]$, by Lemmas~\ref{lemma:BoundDiscreteIntegral} and~\ref{lemma:Delta_varphi} we can write,
with
$\boldsymbol{W}_{q,v} := \int_0^q (v-x)^{H-\halft}\D W_x
$ and $\phi^{(4)}_{t}:=\phi^{(4)}(\XTb)$,
\begin{align*}
    &\abs{\XXi_{433}(r)}\\
    &\lesssim \int_0^{\kappa_r} \EE\Bigg[\varphi_r(t)\phi^{(4)}_{t}\left[\int_{\kappa_q +\Delta}^t \psi'(V_{{\kappa_s}})K({\kappa_s+\Delta},q)\D B_s\right] \Gg_q\int_{\kappa_r}^{t_i}\abs{\boldsymbol{W}_{q,v}} \ \dv  \Bigg] K(\kappa_r,q)\dq \\ \nonumber
    & + \int_0^{\kappa_r} \EE\Bigg[\varphi_r(t)\phi^{(4)}_{t}\left[\int_{\kappa_q +\Delta}^t \psi'(V_{{\kappa_s}})K({\kappa_s+\Delta},q)\D B_s\right]\Gg_q\Bigg]  \left(\int_{\kappa_r}^{t_i} (v-q)^{2H-1}  \dv \right)  K(\kappa_r,q)\dq \\ \nonumber
    & +  \int_0^{\kappa_r}\EE\bigg[ \varphi_r(t)\phi^{(4)}_{t}\left[ \int_t^T \psi'(V_s)K(s,q)\D B_s  + \rho \psi (V_{\kappa_q})\right]\int_{\kappa_r}^{t_i} \EE_q\left[\psi''(V_{t_i})\right](H-\half)\boldsymbol{W}_{q,v}\dv\bigg] K(\kappa_r,q)\dq\\ \nonumber
    & + \int_0^{\kappa_r}\EE\left[ \varphi_r(t)\phi^{(4)}_{t}\left[ \int_t^T \psi'(V_s)K(s,q)\D B_s  + \rho \psi (V_{\kappa_q})\right]\Gg_q\right] \left(\int_{\kappa_r}^{t_i} (v-q)^{2H-1}  \dv \right)  K(\kappa_r,q)\dq \\
    &\lesssim \int_0^{\kappa_r} 
    \int_{\kappa_r}^{t_i}
    \left[\EE\left[\abs{\varphi_r(t)\phi^{(4)}_{t}\left[\int_{\kappa_q +\Delta}^t \psi'(V_{{\kappa_s}})K({\kappa_s+\Delta},q)\D B_s\right]\Gg_q}^2\right]\EE\left[\boldsymbol{W}_{q,v}^2\right] 
    \right]^\half \dv K(\kappa_r,q)\dq \\
    &\quad +\sqrt{\Delta} \int_0^{\kappa_r} K({\kappa_q+\Delta},q)  \left(\int_{\kappa_r}^{t_i} (v-q)^{2H-1} \dv\right)  K(\kappa_r,q)\dq \\
    &\quad + \int_0^{\kappa_r} \int_{\kappa_r}^{t_i}\abs{\EE\left[ \varphi_r(t)\phi^{(4)}_{t}\left[ \int_t^T \psi'(V_s)K(s,q)\D B_s  + \rho \psi (V_{\kappa_q})\right]
     \EE_q[\psi''(V_{t_i})]\boldsymbol{W}_{q,v}\right]}\dv K(\kappa_r,q)\dq \\
    &\quad + \int_0^{\kappa_r}  \left(\int_{\kappa_r}^{t_i} (v-q)^{2H-1}\dv \right)   K(\kappa_r,q)\dq \\
    &=: \XXi_{4331} + \XXi_{4332} + \XXi_{4333} + \XXi_{4334}.
\end{align*}
The rest of the proof follows from the lemmas below.
\end{proof}
We shall need the following technical lemma repeatedly, where we introduce
\begin{equation}\label{eq:Beta1}
\Pf_{\alpha}(\kappa_r, t_i) := \int_0^{\kappa_r} 
\left(\int_{\kappa_r}^{t_i} (v-q)^{\alpha}\D v \right)
K(\kappa_{q}+\Delta,q)K(\kappa_r,q)\dq.
\end{equation}
\begin{lemma}\label{lem:Beta1}
For $\alpha+H+1\geq 0$,
$\displaystyle 
\Pf_{\alpha}(\kappa_r, t_i) \lesssim (t_i-\kappa_r)^{\alpha+H+1}\Delta^{H-\half}$.
\end{lemma}
\begin{lemma}\label{lem:Bound4331}
$\displaystyle |\XXi_{4331}|\lesssim  (t_i-\kappa_r)^{2H}$.
\end{lemma}
\begin{lemma}\label{lem:Bound4332}
$\displaystyle |\XXi_{4332}| 
\lesssim (t_i-\kappa_r)^{3H}\Delta^{H}$.
\end{lemma}
\begin{lemma}\label{lem:Bound4333}
$\displaystyle |\XXi_{4333}|\lesssim (t_i-\kappa_r)^{2H}$.
\end{lemma}
\begin{lemma}\label{lem:Bound4334}
$\displaystyle |\XXi_{4334}|\lesssim (t_i-\kappa_r)^{2H}$.
\end{lemma}

\begin{proof}[Proof of Lemma~\ref{lem:Beta1}]

In the following, we use $(v-q)^\alpha=(v-q)^{-H} (v-q)^{\alpha+H}\le (\kappa_r-q)^{-H}(v-\kappa_r)^{\alpha+H}$, for all~$q\le \kappa_r \le v$ :
\begin{align*}
 &  \int_0^{\kappa_r} \int_{\kappa_r}^{t_i} (v-q)^{\alpha}\D v K(\kappa_{q}+\Delta,q)K(\kappa_r,q)\dq\\
 & \lesssim \int_0^{\kappa_r} (\kappa_r-q)^{-H}\left(\int_{\kappa_r}^{t_i} (v-\kappa_r)^{\alpha+H}\D v \right)K(\kappa_{q}+\Delta,q)K(\kappa_r,q)\dq\\
    & =  (t_i-\kappa_r)^{\alpha+H+1}  \int_0^{\kappa_r} (\kappa_r-q)^{-H}K(\kappa_{q}+\Delta,q)K(\kappa_r,q)\dq\\
 &  =  (t_i-\kappa_r)^{\alpha+H+1}  \int_0^{\kappa_r} (\kappa_r-q)^{-\half}((\kappa_{q}+\Delta)-q)^{H-\half}\dq\\
    & =  (t_i-\kappa_r)^{\alpha+H+1}  \left\{\sum_{j=0}^{\tau(r)-3}\int_{t_j}^{t_{j+1}}(\kappa_r-q)^{-\half}(t_{j+1}-q)^{H-\half}\dq+\int_{\kappa_r-\Delta}^{\kappa_r}(\kappa_r-q)^{H-1}\dq\right\}\\
    & \lesssim (t_i-\kappa_r)^{\alpha+H+1}  \left\{\sum_{j=0}^{\tau(r)-3}(\kappa_r-t_{j+1})^{-\half}\int_{t_j}^{t_{j+1}}(t_{j+1}-q)^{H-\half}\dq+\Delta^H\right\}\\
    & \lesssim \Delta^{H} (t_i-\kappa_r)^{\alpha+H+1}  \left\{\sum_{j=0}^{\tau(r)-3}(\tau(r)-1-({j+1}))^{-\half}+1\right\}
    \lesssim \Delta^{H-\half} (t_i-\kappa_r)^{\alpha+H+1},
\end{align*}
where we estimated the sum in the following way with $k=\tau(r)-1$:
\begin{align*}
\sum_{j=0}^{k-2}({k}-({j+1}))^{-\half}+1
=\sum_{j=1}^{k-1}j^{-\half}+1
& \le \int_{0}^{k-1}s^{-\half}\D s+1\\
 & = \frac{(k-1)^{\half}}{2}+1
  =\frac{\sqrt{t_{k-1}}+\Delta^{\half}}{\Delta^{\half} }.
\end{align*}
\end{proof}

\begin{proof}[Proof of Lemma~\ref{lem:Bound4331}]
A similar argument as for Lemma~\ref{lemma:BoundDiscreteIntegral} entails
\begin{align*}
\EE\left[\abs{\varphi_r(t)\phi^{(4)}(\XTb)\left(\int_{\kappa_q +\Delta}^t \psi'(V_{{\kappa_s}})K({\kappa_s+\Delta},q)\D B_s\right)\Gg_q}^2\right]
    \lesssim \left(1+K(\kappa_q+\Delta,q)\sqrt{\Delta}\right)^2,
\end{align*}
while It\^o isometry and routine integration yield
\begin{align*}
    \EE\left[\abs{\int_0^q (v-x)^{H-\halft}\D W_x }^2   \right]^\half
    =\Bigg(\int_0^q (v-x)^{2H-3}\D x\Bigg)^{\frac{1}{2}}
    & \lesssim\Bigg((v-q)^{2H-2}- v^{2H-2}\Bigg)^{\frac{1}{2}}\\
    & \lesssim(v-q)^{H-1}.
\end{align*}
Therefore, thanks to Lemma~\ref{lem:Beta1}, we conclude  
\begin{align*}
    \XXi_{4331} 
    & \lesssim \int_0^{\kappa_r} \left(\int_{\kappa_r}^{t_i} (v-q)^{H-1}\D v \right) \left(1+ K(\kappa_{q}+\Delta,q) \sqrt{\Delta} \right) K(\kappa_r,q)\dq \\
    & = \int_0^{\kappa_r} \left(\int_{\kappa_r}^{t_i} (v-q)^{2H-1}(v-q)^{-H}\D v \right) K(\kappa_r,q)\dq
     +\sqrt{\Delta}\Pf_{H-1}(\kappa_r, t_i)\\
    &\lesssim \int_0^{\kappa_r} \left(\int_{\kappa_r}^{t_i} (v-\kappa_r)^{2H-1}\D v \right) (\kappa_r-q)^{-\half}\dq
    +\sqrt{\Delta}\Pf_{H-1}(\kappa_r, t_i)\\
    & \lesssim  (t_i-\kappa_r)^{2H}+ \Delta^{H} (t_i-\kappa_r)^{2H}.
\end{align*}
\end{proof}

\begin{proof}[Proof of Lemma~\ref{lem:Bound4332}]
Since  $\XXi_{4332} = \sqrt{\Delta}
\Pf_{2H-1}(\kappa_r,t_i)$,
with~$\Pf_{2H-1}$ introduced in~\eqref{eq:Beta1},
the result follows from Lemma~\ref{lem:Beta1}.
\end{proof}
\begin{proof}[Proof of Lemma~\ref{lem:Bound4333}]
For the term~$\XXi_{4333}$ we need sharp bounds, hence we applying Malliavin integration by parts, which generates several new terms
(here $\phi_{t} := \phi(\XTb)$ and likewise for its derivatives):
\begin{align}
    &\EE\left[ \varphi_r(t)\phi^{(4)}_{t}\left\{ \int_t^T \psi'(V_s)K(s,q)\D B_s  + \rho \psi (V_{\kappa_q})\right\}
     \EE_q[(\psi^2)''(V_{t_i})]\int_0^q (v-x)^{H-\halft}\D W_x\right] \nonumber\\
     &= \int_0^q\EE\left[ \Df_x^W \varphi_r(t)\phi^{(4)}_{t}\left\{ \int_t^T \psi'(V_s)K(s,q)\D B_s  + \rho \psi (V_{\kappa_q})\right\}
     \EE_q[(\psi^2)''(V_{t_i})] \right] (v-x)^{H-\halft}\D x \nonumber\\
     &\quad+\int_0^q\EE\left[  \varphi_r(t)\Df_x^W\phi^{(4)}_{t}\left\{ \int_t^T \psi'(V_s)K(s,q)\D B_s  + \rho \psi (V_{\kappa_q})\right\}
     \EE_q[(\psi^2)''(V_{t_i})] \right] (v-x)^{H-\halft}\D x \nonumber\\
     &\quad+\int_0^q\EE\left[  \varphi_r(t)\phi^{(4)}_{t}\Df_x^W\left\{ \int_t^T \psi'(V_s)K(s,q)\D B_s  + \rho \psi (V_{\kappa_q})\right\}
     \EE_q[(\psi^2)''(V_{t_i})] \right] (v-x)^{H-\halft}\D x \nonumber\\
     &\quad+\int_0^q\EE\left[  \varphi_r(t)\phi^{(4)}_{t}\left\{ \int_t^T \psi'(V_s)K(s,q)\D B_s  + \rho \psi (V_{\kappa_q})\right\}
     \Df_x^W\EE_q[(\psi^2)''(V_{t_i})] \right] (v-x)^{H-\halft}\D x \nonumber\\
     &=\int_0^q\EE\left[ \EE_r[(\psi^2)''(V_t)]\phi^{(4)}_{t}\left\{ \int_t^T \psi'(V_s)K(s,q)\D B_s  + \rho \psi (V_{\kappa_q})\right\}
     \EE_q[(\psi^2)''(V_{t_i})] \right] K(t,x)(v-x)^{H-\halft}\D x \label{eq:Bound4333_line1} \\
     &\quad+\int_0^q\EE\bigg[  \varphi_r(t)\phi^{(5)}_{t}
     \left\{\int_x^t \psi'(V_{\kappa_s})K(\kappa_s,x)\D B_s+ \int_t^T \psi'(V_s)K(s,x)\D B_s  + \rho \psi (V_{\kappa_x})\right\}\cdot \label{eq:Bound4333_line2}\\
     &\qquad \cdot
     \left\{ \int_t^T \psi'(V_s)K(s,q)\D B_s  + \rho \psi (V_{\kappa_q})\right\}
     \EE_q[(\psi^2)''(V_{t_i})] \bigg] (v-x)^{H-\halft}\D x \nonumber\\
      &\quad+\int_0^q\EE\left[  \varphi_r(t)\phi^{(4)}_{t}\left\{ \int_t^T \psi''(V_s)K(s,x)K(s,q)\D B_s  + \rho \psi'(V_{\kappa_q})K(\kappa_q,x)\right\}
     \EE_q[(\psi^2)''(V_{t_i})] \right] (v-x)^{H-\halft}\D x \label{eq:Bound4333_line3}\\
     &\quad+\int_0^q\EE\left[  \varphi_r(t)\phi^{(4)}_{t}\left\{ \int_t^T \psi'(V_s)K(s,q)\D B_s  + \rho \psi (V_{\kappa_q})\right\}
     \EE_q[(\psi^2)''(V_{t_i})] \right] K(t_i,x) (v-x)^{H-\halft}\D x \label{eq:Bound4333_line4}\\
     &\lesssim \int_0^q K(t_i,x) (v-x)^{H-\halft}\dx
      + \int_0^q \left\{(1+K(\kappa_x+\Delta,x) \sqrt{\Delta}\right\} (v-x)^{H-\halft}\dx 
      + \int_0^q K(\kappa_q,x) (v-x)^{H-\halft}\dx \nonumber\\
     &=: \XXi_{43331}(v,q) + \XXi_{43332}(v,q) + \XXi_{43333}(v,q). \nonumber
\end{align}
In the last step we leveraged the fact that the expectations in~\eqref{eq:Bound4333_line1} and~\eqref{eq:Bound4333_line4} are bounded uniformly, hence both these terms contribute to~$\XXi_{43331}$. 
The estimate for line~\eqref{eq:Bound4333_line2} is a consequence of Lemma~\ref{lemma:BoundDiscreteIntegral} and it contributes to~$\XXi_{43332}$. 
Then we used Cauchy-Schwarz inequality to estimate~\eqref{eq:Bound4333_line3} as follows:
\begin{align*}
    &\EE\left[  \varphi_r(t)\phi^{(4)}_{t}\int_t^T \psi''(V_s)K(s,x)K(s,q)\D B_s  
     \EE_q\left[(\psi^2)''(V_{t_i})\right] \right]\\
     &\lesssim \left| \int_t^T \EE[\psi''(V_s)^2] K(s,x)^2 K(s,q)^2 \D s\right|^{\half}\\
     & \lesssim \left|\int_t^T  K(s,q)^2 \D s\,\sup_{s\in\TT}\EE[\psi''(V_s)^2]  K(t,x)^2\right|^{\half}\\
     & \lesssim K(t,x) \le K(t_i,x).
\end{align*}
Thus, this term contributes to~$\XXi_{43333}$.
Note that, for all $q<\kappa_r<v$, then
\begin{align*}
    \XXi_{43331}(v,q) 
    \le \int_0^q  (v-x)^{H-\halft}\dx \, K(t_i,q) 
    \lesssim (v-q)^{H-\half}  K(t_i,\kappa_r)
    \le  K(\kappa_r,q)K(t_i,\kappa_r).
\end{align*}
Plugging this back into the double integral  {for $\XXi_{4333}$} we obtain
\begin{align*}
    \int_0^{\kappa_r}\left(\int_{\kappa_r}^{t_i} \XXi_{43331}(v,q) \dv \right)K(\kappa_r,q)\dq
    \lesssim \left(\int_{\kappa_r}^{t_i} \dv \right)K(t_i,\kappa_r) \int_0^{\kappa_r}K(\kappa_r,q)^2 \dq
    \lesssim (t_i-\kappa_r)^{H+\half}.
\end{align*}
 {Recall that $\kappa_q=t_{\tau(q)-1}$.} 
For the term~$\XXi_{43332}$ we investigate
\begin{equation}\label{eq:int_K_kappa}
    \int_0^q K(\kappa_x+\Delta,x)\dx 
    =\sum_{j=1}^{\tau(q)-1} \int_{t_{j-1}}^{t_j} K(t_j,x)\dx + \int_{\kappa_q}^q K(\kappa_q+\Delta,x)\dx 
    \lesssim \sum_{j=1}^{\tau(q)} \Delta^{H+\half} \le \Delta^{H-\half}.
\end{equation}
Once we insert this back in the double integral, since $1\lesssim K(t_i,x)$, it entails
\begin{align*}
&\int_0^{\kappa_r}\left(\int_{\kappa_r}^{t_i} \XXi_{43332}(v,q) \dv\right) K(\kappa_r,q)\dq\\
    &\lesssim \int_0^{\kappa_r}\left(\int_{\kappa_r}^{t_i} \XXi_{43331}(v,q) \dv\right)K(\kappa_r,q)\dq\\ 
    & \qquad + \sqrt{\Delta} \int_0^{\kappa_r}\int_{\kappa_r}^{t_i}\left(\int_0^q K(\kappa_x+\Delta,x)\dx\right)  (v-q)^{H-\halft}\dv K(\kappa_r,q)\dq \\
    &\lesssim (t_i-\kappa_r)^{H+\half} + \Delta^H \int_0^{\kappa_r} \Big( K(\kappa_r,q)-K(t_i,q)\Big)K(\kappa_r,q)\dq \\
    & \lesssim (t_i-\kappa_r)^{H+\half} + \Delta^H(t_i-\kappa_r)^{2H},
\end{align*}
where we leveraged Lemma~\ref{lemma:ti_tj} with $\alpha=H-\half$ to obtain the last inequality. 
Finally, we exploit~$(v-x)^{H-\halft}\le (v-q)^{H-\halft}$ with~$x\le q$ in the last term to reach, thanks to Lemma~\ref{lemma:ti_tj},
\begin{align*}
    & \int_0^{\kappa_r}\left(\int_{\kappa_r}^{t_i} \XXi_{43333}(v,q) \dv\right) K(\kappa_r,q)\dq\\
    &\le \int_0^{\kappa_r}\left(\int_0^q K(\kappa_q,x)\dx\right)  \left(\int_{\kappa_r}^{t_i}(v-q)^{H-\halft}\dv\right) K(\kappa_r,q)\dq \\
    &\lesssim \int_0^{\kappa_r} \Big( K(\kappa_r,q)-K(t_i,q)\Big)K(\kappa_r,q)\dq
    \lesssim (t_i-\kappa_r)^{2H}.
\end{align*}
\end{proof}

\begin{proof}[Proof of Lemma~\ref{lem:Bound4334}]
Since
$\int_{\kappa_r}^{t_i} (v-q)^{2H-1}\D v 
    \leq \int_{\kappa_r}^{t_i} (v-\kappa_r)^{2H-1}\D v 
    \lesssim (t_i-\kappa_r)^{2H}$,
\begin{align*}
\XXi_{4334} = \int_0^{\kappa_r}  \left(\int_{\kappa_r}^{t_i} (v-q)^{2H-1}\dv \right)   K(\kappa_r,q)\dq
 & \lesssim 
(t_i-\kappa_r)^{2H}
\int_0^{\kappa_r}K(\kappa_r,q)\dq\\
 & \lesssim 
(t_i-\kappa_r)^{2H}.
\end{align*}
\end{proof}

\begin{proof}[Proof of Lemma~\ref{lem:434}]
By Lemma~\ref{lemma:BoundDiscreteIntegral} we get the following decomposition:
\begin{align*}
    & |\XXi_{434}(r) | \\
    & \leq \int_0^{t_i}\EE\Bigg[ \varphi_r(t) \phi^{(4)}(\XTb) 
    \Bigg\{ \int_q^t \psi'(V_{\kappa_{s}})K(\kappa_{s},q)\D B_q + \int_t^T \psi'(V_s)K(s,q)\D B_q + \rho\psi(V_{\kappa_r}) \Bigg\}\EE_q[\psi'(V_{t_i})]\Bigg]\\
    &\qquad \cdot \Big(K(\kappa_r,q)-K(t_i,q)\Big)\dq\\
    & \lesssim \int_0^{t_i} \Big(\sqrt{\Delta} K(\kappa_q+\Delta,q)+1\Big)\Big(K(\kappa_r,q)-K(t_i,q)\Big)\dq
    =: \XXi_{4341}(r) +\XXi_{4342}(r).
\end{align*}
For $j\le \tau(r)-2$, let us define
\begin{equation}\label{eq:def_vartheta}
    \vartheta_j:= \int_{t_j}^{t_{j+1}} K(t_{j+1},q)\big(K(\kappa_r,q)-K(t_i,q)\big)\dq .    
\end{equation}
Recall that~$\kappa_r=t_{\tau(r)-1}$. Thus, we have
\begin{align}
    |\XXi_{4341}(r)|
    & { \leq} \sqrt{\Delta}\Bigg\{ \sum_{j=0}^{\tau(r)-3}\int_{t_j}^{t_{j+1}} K(t_{j+1},q)\big(K(\kappa_r,q)-K(t_i,q)\big)\dq\\ & \qquad\qquad \qquad  +\int_{\kappa_r-\Delta}^{\kappa_r} K(\kappa_r,q)\big(K(\kappa_r,q)-K(t_i,q)\big)\dq \nonumber\\
    & \qquad\qquad \qquad  +\sum_{j=\tau(r)-1}^{i-2}\int_{t_j}^{t_{j+1}}K(t_{j+1},q)K(t_i,q)\dq + \int_{t_{i-1}}^{t_{i}}K(t_i,q)^2\dq\Bigg\}\nonumber\\
    &  \lesssim \sqrt{\Delta}\Bigg\{ \sum_{j=0}^{\tau(r)-3}\vartheta_j +\int_{\kappa_r-\Delta}^{\kappa_r} K(\kappa_r,q)^2\dq\nonumber\\
    & \qquad\qquad \qquad \quad +\sum_{j=\tau(r)-1}^{i-2}K(t_i,t_{j+1})\int_{t_j}^{t_{j+1}}K(t_{j+1},q)\dq + \Delta^{2H}\Bigg\}\nonumber\\
    & \lesssim \sqrt{\Delta}\Bigg\{ \sum_{j=0}^{\tau(r)-3}\vartheta_j + \Delta^{2H} +\sum_{j=\tau(r)-1}^{i-2}K(t_i,t_{j+1})\Delta^{H+\half} + \Delta^{2H}\Bigg\} \label{eq:comps_4341}\\
\end{align}
We note that, for $q\in[t_j,t_{j+1})$ and $t_{j+1}<\kappa_r$,
\begin{align*}
    K(\kappa_r,q)-K(t_i,q)
    & \lesssim
    \int_{\kappa_r}^{t_i} (s-q)^{H-\halft}\ds
    \le (\kappa_r-t_{j+1})^{-H-1} \int_{\kappa_r}^{t_i} (s-\kappa_r)^{2H-\half}\ds\\
    &=\Delta^{-H-1} (\tau(r)-j-2)^{-H-1}(t_i-\kappa_r)^{2H+\half}.
\end{align*}
Therefore, from~\eqref{eq:def_vartheta}, we obtain
\begin{align*}
\sqrt{\Delta} \sum_{j=0}^{\tau(r)-3}\vartheta_j
    \lesssim  (t_i-\kappa_r)^{2H+\half}\sum_{ j=0}^{\tau(r)-3}(\tau(r)-j-2)^{-H-1}
    \lesssim (t_i-\kappa_r)^{2H+\half},
\end{align*}
and
\begin{align*}
\sum_{j=\tau(r)-1}^{i-2}K(t_i,t_{j+1})
    \lesssim \sum_{j=\tau(r)-1}^{i-2} (i-(j+1))^{H-\half} \Delta^{H-\half} 
    & \lesssim \Delta^{H-\half} (i-\tau(r)-1)^{H+\half}\\
    & = \Delta^{-1} (t_i-\kappa_r)^{H+\half},
\end{align*}
where we exploited $i-\tau(r)-1=\Delta^{-1} (t_i-\kappa_r)$. 
Overall, from~\eqref{eq:comps_4341}, we get
\begin{align*}
    \XXi_{4341}(r)
    \lesssim (t_i-\kappa_r)^{2H+\half}+\Delta^H(t_i-\kappa_r)^{H+\half} + \Delta^{2H+\half}.
\end{align*}
Furthermore, 
$
    \abs{\XXi_{4342}(r)} 
    \lesssim \int_0^{t_i} \abs{ K(\kappa_r,q) - K(t_i,q)} \dq \lesssim (t_i-\kappa_r)^{H+\half}$, by Lemma~\ref{lemma:ti_tj}
    and the claim follows.
\end{proof}


\subsection{Bound for $\Xi_5$}
Recall that $\Gg_r=\sup_{s\in\TT,n=2,3} \EE_r[\psi^{(n)}(V_s)]$ and define 
\begin{align*}
Z(t,r)
:=\phi^{(3)}(\XTb) \left(\int_{\kappa_r+2\Delta}^t\psi'(V_{\kappa_s})K(\kappa_s,r)\D B_s + \int_t^T \psi'(V_{s})K(s,r)\D B_s +  {\rho}\psi(V_{\kappa_r})\right),
\end{align*}
where we interpret the first  integral to be null if~$\kappa_r+2\Delta\ge t$, i.e. $r\ge t_{i-1}$. 
Note that~$Z(t,r)$ is uniformly bounded in $L^p$ for all $p>1$, by virtue of the growth of $\phi,\psi$ and their derivatives, and the proof of Lemma~\ref{lemma:BoundDiscreteIntegral}. Equation~\eqref{eq:exp_Uu} shows that
\begin{align*}
\Uu_r
    &=\phi^{(3)}(\XTb)\int_{\kappa_r+\Delta}^{\kappa_r+2\Delta}\psi'(V_{\kappa_s})K(\kappa_s,r)\D B_s + Z(t,r)\\
    &= \phi^{(3)}(\XTb)\psi'(V_{{\kappa_r+\Delta}})(B_{{\kappa_r+2\Delta}}-B_{{\kappa_r+\Delta}})K({\kappa_r+\Delta},r) + Z(t,r).
\end{align*}
Exploiting this and Lemma~\ref{lemma:Delta_varphi} yields,
writing and $\boldsymbol{W}_{r,s} := \int_0^r (s-u)^{H-\frac{3}{2}}\D W_u
$,
\begin{align}
    & \XXi_5 
    = \EE\Bigg[\int_0^{t_i} \Uu_r \big(\varphi_r(t)-\varphi_r(t_i)\big) K(t_i,r)\dr\Bigg] \\ \nonumber
    & \leq \EE\Bigg[\int_0^{t_i} \Bigg( \EE_t\Bigg[\phi^{(3)}(\XTb)\psi'(V_{{\kappa_r+\Delta}})(B_{{\kappa_r+2\Delta}}-B_{{\kappa_r+\Delta}})\Bigg]K({\kappa_r+\Delta},r) + Z(t,r) \Bigg) \cdot \\ \nonumber
    & \qquad \cdot \Bigg( \int_{t_i}^t \left(\EE_r[(\psi^2)''(V_{s})]\boldsymbol{W}_{r,s} + \half\EE_r[(\psi^2)^{(3)}(V_{s})] (s-r)^{2H-1}\right) \ds \Bigg) K(t_i,r)\dr\Bigg] \\ \nonumber
    & \leq \int_0^{t_i} \EE\Bigg[\phi^{(3)}(\XTb)\psi'(V_{{\kappa_r+\Delta}})(B_{{\kappa_r+2\Delta}}-B_{{\kappa_r+\Delta}})\mathcal{G}_r\int_{t_i}^t
    \left|\boldsymbol{W}_{r,s} 
    \right|
    \ds  \Bigg] K({\kappa_r+\Delta},r)K(t_i,r)\dr \\ \nonumber
    &\qquad + \int_0^{t_i} \EE\Bigg[\phi^{(3)}(\XTb)\psi'(V_{{\kappa_r+\Delta}})\mathcal{G}_r(B_{{\kappa_r+2\Delta}}-B_{{\kappa_r+\Delta}})\Bigg]K({\kappa_r+\Delta},r)  \int_{t_i}^t (s-r)^{2H-1}  \ds  K(t_i,r)\dr \\ \nonumber
    &\qquad +  \int_0^{t_i}\EE\left[ Z(t,r) \int_{t_i}^t \EE_r[(\psi^2)''(V_{s})]\boldsymbol{W}_{r,s}\ds\right] K(t_i,r)\dr\\ \nonumber
    &\qquad + \int_0^{t_i} \EE\Big[ Z(t,r) \mathcal{G}_r \Big] \left(\int_{t_i}^t (s-r)^{2H-1} \ds\right)  K(t_i,r)\dr =: \sum_{i=1}^4 \XXi_{5i},
\end{align}
where we used~\eqref{eq:varphi_prime}.
We shall need the following lemma, mimicking Lemma~\ref{lem:Beta1}, but with slightly different arguments in with respect to the ones used in~\eqref{eq:Beta1}.
\begin{lemma}\label{lem:Bound_Two_Int2}
For $\alpha\in(-1,0)$ and $t \in [t_i,t_{i+1})$,
$$
\Pf_{\alpha}(t_i,t)
\lesssim 
\Delta^{2H+\alpha+1}
\left(\Delta^{-\alpha-H-\half} \one_{\alpha+H>-\half} +  \one_{\alpha+H<-\half} 
+|\log(\Delta)|\one_{\alpha+H=-\half} + 1\right).
$$
\end{lemma}
\begin{proof}
Since 
$
    \int_{t_i}^t (t_i-t_{j+1})^{\alpha} \ds
    \lesssim (i-j-1)^{\alpha}\Delta^{1+\alpha}
$, 
$
    \int_{t_j}^{t_{j+1}}K(t_{j+1},r)\dr 
    = \frac{(t_{j+1}-t_{j})^{H+\half}}{H+\half}
    \lesssim \Delta^{H+\half}
$ 
and  
$
    \int_{t_i}^t (s-t_i)^{\alpha}\ds
    = \frac{1}{1+\alpha}(t-t_i)^{1+\alpha}
    \lesssim \Delta^{1+\alpha}
$,
then we obtain
\begin{align*}
\Pf_{\alpha}(t_i,t)
 & =  \int_0^{t_i} \left(\int_{t_i}^t (s-r)^{\alpha}\ds \right) K(\kappa_r+\Delta,r)K(t_i,r)\dr\\
    &\le \sum_{j=0}^{i-3} \int_{t_j}^{t_{j+1}} \left(\int_{t_i}^t (t_i-t_{j+1})^{\alpha} \ds\right) K(t_{j+1},r)(t_i-t_{j+1})^{H-\half} \dr\\ 
    &\qquad+  \int_{t_{i-2}}^{t_{i-1}} \left(\int_{t_i}^t (s-t_i)^{\alpha}\ds\right) K(t_{i-1},r)K(t_i,r)\dr
    +  \int_{t_{i-1}}^{t_{i}} \left(\int_{t_i}^t (s-t_i)^{\alpha}\ds\right) K(t_{i},r)^2\dr\\
    &\lesssim \Delta^{2H+\alpha+1}  \sum_{j=0}^{i-3} (i-j-1)^{\alpha+H-\half}
     + \Delta^{1+\alpha} \left(\int_{t_{i-2}}^{t_{i-1}}K(t_{i-1},r)^2\dr + \int_{t_{i-1}}^{t_{i}}K(t_{i},r)^2\dr \right)\\
    &\lesssim \Delta^{2H+\alpha+1}
\left(\Delta^{-\alpha-H-\half} \one_{\alpha+H>-\half} +  \one_{\alpha+H<-\half} -\log(\Delta)\one_{\alpha+H=-\half}\right)
     + \Delta^{2H+\alpha+1}.
\end{align*}
Indeed, 
since $\alpha+H-\half<0$ by assumption, then we have
\begin{align*}
\sum_{j=0}^{i-3} 
(i-j-1)^{\alpha+H-\half}
& \leq 
 \int_{0}^{i-2}
(i-s-1)^{\alpha+H-\half}\ds\\
 & =
 \left\{
 \begin{array}{ll}
 \displaystyle \frac{(i-1)^{\alpha+H+\half} - 1}{\alpha+H+\half}
  \lesssim \Delta^{-\half-\alpha-H}, & \text{if }\alpha+H>-\half,\\
 \displaystyle \frac{(i-1)^{\alpha+H+\half} - 1}{\alpha+H+\half}
  \lesssim 1, & \text{if }\alpha+H<-\half,\\
   \displaystyle \log(i-1)\sim 
   |\log(\Delta)|, & \text{if }\alpha+H=-\half.
  \end{array}
\right.
\end{align*}
\end{proof}
\begin{proposition}
$\displaystyle
    |\XXi_{5}| 
    \lesssim \bigstar(\Delta)$.
\end{proposition}
\begin{proof}
Since 
$|\XXi_{5}| \lesssim \sum_{i=1}^4 \XXi_{5i}$,
the result follows from Lemmas~\ref{lem:Bound51}-\ref{lem:Bound52}-\ref{lem:Bound53}-\ref{lem:Bound54}.
\end{proof}
\begin{lemma}\label{lem:Bound51}
$\displaystyle \abs{\XXi_{51}} 
\lesssim \sqrt{\Delta} \Pf_{H-1}(t_i,t)\lesssim \bigstar(\Delta)$. \end{lemma}

\begin{lemma}\label{lem:Bound52}
$\displaystyle \abs{\XXi_{52}}
\lesssim 
    \Delta^{4H+\half}
\left(\Delta^{\half-3H} \one_{H>\halfs}
+  \one_{H<\halfs}
+|\log(\Delta)|\one_{H=\halfs} + 1\right)
\lesssim \bigstar(\Delta).
$
\end{lemma}

\begin{lemma}\label{lem:Bound53}
$\displaystyle
\abs{\XXi_{53}} 
\lesssim \bigstar(\Delta)$.
\end{lemma}

\begin{lemma}\label{lem:Bound54}
$\displaystyle
\abs{\XXi_{54}} 
\lesssim \bigstar(\Delta)$.
\end{lemma}

\begin{proof}[Proof of Lemma~\ref{lem:Bound51}]
By Cauchy-Schwarz and It\^o isometry, we write
\begin{align*}
    \abs{\XXi_{51}}
    &\lesssim \sqrt{\Delta} \int_0^{t_i} \left(\int_{t_i}^t \left\{\int_0^r (s-u)^{2H-3}\du\right\}^\half\ds \right) K(\kappa_r+\Delta,r)K(t_i,r)\dr\\
    &\lesssim \sqrt{\Delta} \int_0^{t_i} \left(\int_{t_i}^t (s-r)^{H-1}\ds \right) K(\kappa_r+\Delta,r)K(t_i,r)\dr
    = \sqrt{\Delta} \Pf_{H-1}(t_i,t),
\end{align*}
and the proof follows from  Lemma~\ref{lem:Bound_Two_Int2} since
\begin{align*}
\sqrt{\Delta} \Pf_{H-1}(t_i,t)
    & \lesssim \Delta^{3H+\half} \left[\Delta^{\half-2H} \one_{H>\frac{1}{4}} +  \one_{H<\frac{1}{4}} +|\log(\Delta)|\one_{H=\frac{1}{4}} + 1\right]\\
     & \lesssim \Delta^{(3H+\half)\land 1}
     \lesssim\bigstar(\Delta).
\end{align*}
\end{proof}

\begin{proof}[Proof of Lemma~\ref{lem:Bound52}]
As in the~$\XXi_{51}$ case, Cauchy-Schwarz and Lemma~\ref{lem:Bound_Two_Int2} yield
\begin{align*}
    \abs{\XXi_{52}} &\lesssim \sqrt{\Delta} \int_0^{t_i} \int_{t_i}^t (s-r)^{2H-1}\ds 
K(\kappa_r+\Delta,r)K(t_i,r)\dr
     = \sqrt{\Delta} \Pf_{2H-1}(t_i,t)
     \lesssim \bigstar(\Delta).
\end{align*}
\end{proof}

\begin{proof}[Proof of Lemma~\ref{lem:Bound54}]
It is immediate from
Lemma~\ref{lemma:bound_KDeltaK} since
$$
\abs{\XXi_{54}} 
\lesssim \int_0^{t_i} \left( (t_i-r)^{2H} -(t-r)^{2H}\right) K(t_i,r)\dr 
\lesssim \diamondsuit_{2H}(\Delta)
 = \bigstar(\Delta).$$
\end{proof}

\begin{proof}[Proof of Lemma~\ref{lem:Bound53}]
By IBP with respect to $W$, the inner integrand decomposes as
\begin{align*}
&\EE\left[Z(t,r)\EE_r[(\psi^2)''(V_{s})]\left(\int_0^r (s-u)^{H-\halft}\D W_u\right)\right]\\
    &= \EE\left[\int_0^r \Big(\Df^W_u Z(t,r) \EE_r[(\psi^2)''(V_{s})]+ Z(t,r)\Df^W_u \EE_r[(\psi^2)''(V_{s})]\Big) (s-u)^{H-\halft}\du\right] \\
    &= \int_0^r \EE\left[ \Df^W_u \phi^{(3)}_{t}\left(\int_{\kappa_r+2\Delta}^T \psi'(V_{\kappa_q^t})K(\kappa_q^t,r)\D B_q + \rho \psi(V_{\kappa_r})\right)(\psi^2)''(V_{s})\right] (s-u)^{H-\halft}\du\\
    &+  \int_0^r \EE\left[  \phi^{(3)}_{t} \Df^W_u \left(\int_{\kappa_r+2\Delta}^T \psi'(V_{\kappa_q^t})K(\kappa_q^t,r)\D B_q + \rho \psi(V_{\kappa_r})\right)(\psi^2)''(V_{s})\right] (s-u)^{H-\halft}\du\\
    & +  \int_0^r \EE\left[ \phi^{(3)}_{t}\left(\int_{\kappa_r+2\Delta}^T \psi'(V_{\kappa_q^t})K(\kappa_q^t,r)\D B_q +  \rho\psi(V_{\kappa_r})\right) \Df^W_u (\psi^2)''(V_{s})\right] (s-u)^{H-\halft}\du\\
    &=: \XXi_{531}(s,r) + \XXi_{532}(s,r) + \XXi_{533}(s,r),
\end{align*}
again writing $\phi^{(3)}_{t}:=\phi^{(3)}(\XTb)$, 
and the proof follows from the three lemmas below.
\end{proof}

\begin{lemma}\label{lem:Bound531}
    $\displaystyle \int_0^{t_i} \left(\int_{t_i}^t \abs{\XXi_{531}(s,r)} \ds \right)K(t_i,r)\dr\lesssim \Delta^{1+H}$.
\end{lemma}
\begin{lemma}\label{lem:Bound532}
    $\displaystyle \int_0^{t_i} \left(\int_{t_i}^t \abs{\XXi_{532}(s,r)} \ds \right)K(t_i,r)\dr\lesssim \bigstar(\Delta)$.
\end{lemma}
\begin{lemma}\label{lem:Bound533}
    $\displaystyle \int_0^{t_i} \left(\int_{t_i}^t \abs{\XXi_{533}(s,r)} \ds \right)K(t_i,r)\dr\lesssim \bigstar(\Delta)$.
\end{lemma}

\begin{remark}
    Some terms in the proofs below are reminiscent of the integrals in $\XXi_{4333}$. We however here aim for a lower error ($\bigstar(\Delta)$ vs $\Delta^{2H})$, which requires sharper estimates.
\end{remark}
The proofs are given in order of increasing length.
\begin{proof}[Proof of Lemma~\ref{lem:Bound533}]
Since $Z(t,r)\in L^2$, then
\begin{align*}
     \abs{\XXi_{533}(s,r)} 
     & = \abs{\int_0^r \EE\left[Z(t,r) (\psi^2)^{(3)}(V_s) \right] K(s,u) (s-u)^{H-\halft}\du}\\
     & \lesssim \int_0^r (s-u)^{2H-2}\du
     \lesssim (s-r)^{2H-1},
\end{align*}
and with that we conclude, 
exploiting Lemma~\ref{lemma:bound_KDeltaK}:
\begin{align*}
    \int_0^{t_i} \left\{\int_{t_i}^t \abs{\XXi_{533}(s,r)} \ds \right\}K(t_i,r)\dr
    &\lesssim \int_0^{t_i} \left(\int_{t_i}^t (s-r)^{2H-1} \ds \right)K(t_i,r) \dr\\
    &\lesssim \int_0^{t_i} \Big[(t-r)^{2H}-(t_i-r)^{2H} \Big] K(t_i,r) \dr
    \lesssim \bigstar(\Delta).
\end{align*}
\end{proof}


\begin{proof}[Proof of Lemma~\ref{lem:Bound531}]
By virtue of Lemma~\ref{lemma:BoundDiscreteIntegral}, we estimate
\begin{align}
    \abs{\XXi_{531}}(s,r) 
    &=\int_0^r \EE\bigg[\phi^{(4)}(\XTb) \left(\int_u^T \psi'(V_{\kappa^t_q})K(\kappa^t_q,u)\D B_q + \rho\psi(V_{\kappa_r})\right) \cdot \nonumber\\
    &\qquad \left(\int_{\kappa_r+2\Delta}^T \psi'(V_{\kappa_q^t})K(\kappa_q^t,r)\D B_q + \rho\psi(V_{\kappa_r})\right)(\psi^2)''(V_{s})\bigg]  (s-u)^{H-\halft}\du \nonumber\\
    &\lesssim \int_0^r \Big( 1+ \sqrt{\Delta}K(\kappa_u+\Delta,u) \Big) (s-u)^{H-\halft}\du \label{eq:Xi_531}\\
    &\lesssim (s-r)^{H-\half} + \sqrt{\Delta}\sum_{\ell=0}^{\tau(r)-2} \int_{t_\ell}^{t_{\ell+1}} K(t_{\ell+1},u)(s-u)^{H-\halft} \du\\
    & \qquad\qquad \qquad + \sqrt{\Delta}\int_{\kappa_r}^r K(\kappa_r+\Delta,u)(s-u)^{H-\halft}\du\\
    &=: \XXi_{5311}(s,r) +\XXi_{5312}(s,r)+\XXi_{5313}(s,r).
\end{align}
$\mathbf{(\XXi_{5311})}$ 
To begin with, we have, by Lemma~\ref{lemma:bound_KDeltaK},
\begin{align*} 
 \int_0^{t_i}\int_{t_i}^t \XXi_{5311}(s,r) \ds K(t_i,r)\dr
    & =\int_0^{t_i} \int_{t_i}^t(s-r)^{H-\half} \ds K(t_i,r)\dr \\
    & \lesssim \int_0^{t_i} \Big[ (t_i-r)^{H+\half} - (t-r)^{H+\half} \Big]  K(t_i,r)\dr\\
    &= \diamondsuit_{H+\half}(\Delta) = \Delta.
    \end{align*} 

$\mathbf{(\XXi_{5312})}$ 
We have, for $\ell\le \tau(r)-2$ and $s\in[t_i,t]$, 
$$
\int_{t_\ell}^{t_{\ell+1}} K(t_{\ell+1},u)(s-u)^{H-\halft} \du
\le 
(t_i-t_{\ell+1})^{H-\halft} \int_{t_\ell}^{t_{\ell+1}} K(t_{\ell+1},u)\du
\le \Delta^{2H-1} (i-\ell-1)^{H-\halft}.
$$
We first deal with the interval $[t_{i-1},t_i]$, and exploiting the previous lines, we have
\begin{align*}
    \int_{t_{i-1}}^{t_i}\left[ \int_{t_i}^t \XXi_{5312}(s,r) \ds \right] K(t_i,r)\dr 
    &\le \sqrt{\Delta} \Delta^{2H-1} \int_{t_{i-1}}^{t_i} \int_{t_i}^t \ds\left[\sum_{\ell=0}^{\tau(r)-2} (i-\ell-1)^{H-\halft}\right] K(t_i,r)\dr\\
    &\le \Delta^{2H+\half} \int_{t_{i-1}}^{t_i} K(t_i,r) \dr \left[\sum_{\ell=0}^{i-2} (i-\ell-1)^{H-\halft}\right]
    \lesssim \Delta^{3H+1},
\end{align*}
since the sum is bounded, as it can be seen from~\eqref{eq:sum_iell} for $\alpha=0$.
We observe that, on $[0,t_{i-1}]$, 
\begin{align*}
    \int_0^{t_{i-1}}
    \left[ \int_{t_i}^t \XXi_{5312}(s,r)\ds \right]K(t_i,r)\dr 
    \le \sqrt{\Delta} \Delta^{2H-1} \int_0^{{ t_{i-1}}}\int_{t_i}^t \ds\left[\sum_{\ell=0}^{\tau(r)-2} (i-\ell-1)^{H-\halft}\right] K(t_i,r)\dr.
\end{align*}
The sum is estimated as the one above, albeit with the advantage that $\tau(r)\le i-1$:
\begin{align*}
    \sum_{\ell=0}^{\tau(r)-2} (i-\ell-1)^{H-\halft}
    \lesssim (i-\tau(r))^{H-\half} - (i-1)^{H-\half} 
    = \Delta^{\half-H} \Big[ (t_{i-1}-{ \kappa_r})^{H-\half} - 
    t_{i-1}^{H-\half}\Big],
\end{align*}
which entails 
\begin{align*}
    \int_0^{t_{i-1}}\left( \int_{t_i}^t \XXi_{5312}(s,r)\ds \right)K(t_i,r)\dr 
    &\lesssim \Delta^{1+H} \int_0^{t_{i-1}}  \Big[ (t_{i-1}-{ \kappa_r})^{H-\half} - 
    t_{i-1}^{H-\half}\Big] K(t_i,r) \dr\\
    &\lesssim \Delta^{1+H} \int_0^{t_{i-1}} (t_{i-1}-r)^{2H-1} \dr\lesssim \Delta^{1+H}.
\end{align*}
$\mathbf{(\XXi_{5313})}$ For the last term
\begin{align*}
    \sqrt{\Delta} \int_0^{t_i} \int_{t_i}^t & \XXi_{5313}(s,r) \ds K(t_i,r)\dr
    \lesssim \sqrt{\Delta} \int_0^{t_i} \int_{t_i}^t (s-r)^{H-\halft} \ds \int_{\kappa_r}^r K(r,u)\du K(t_i,r)\dr\\
    &\lesssim \Delta^{H+1} \int_0^{t_i} \Big[ (t_i-r)^{H-\half} -(t-r)^{H-\half}\Big]    K(t_i,r)\dr\\
    & \lesssim \Delta^{H+1} \diamondsuit_{H-\half}(\Delta)
     = \Delta^{3H+1} .
\end{align*}
\end{proof}

\begin{proof}[Proof of Lemma~\ref{lem:Bound532}]
We compute
\begin{align*}
    &\Df_u^W \left(\int_{\kappa_r+2\Delta}^T \psi'(V_{\kappa_q^t})K(\kappa_q^t,r)\D B_q +  \rho\psi(V_{\kappa_r})\right)\\
& = \int_{\kappa_r+2\Delta}^t \psi''(V_{\kappa_q})K(\kappa_q,u)K(\kappa_q,r)\D B_q \\
 & \qquad + \int_t^T \psi''(V_{q})K(q,u)K(q,r)\D B_q
    + \rho  {\psi'(V_{\kappa_u^t})K(\kappa_u^t,r)}
    +\rho \psi'(V_{\kappa_r})K(\kappa_r,u),
\end{align*}
where $K(\kappa_u,r)=0$ since $u\le r$  {and we recall that $\kappa^t_q=\kappa_q \one_{q\le t} + q\one_{q>t}$}. 
By Cauchy-Schwarz inequality and It\^o isometry, 
\begin{align*}
    \abs{\XXi_{532}(s,r)} 
    &\lesssim \overbrace{\bigg\lvert\int_0^r \EE\left[ \phi^{(3)}(X^{t, \overline{X}_t}_T)(\psi^2)''(V_s)\int_{\kappa_r+2\Delta}^t \psi''(V_{\kappa_q}) K(\kappa_q,u) K(\kappa_q,r) \D B_q \right](s-u)^{H-\halft}\du\bigg\lvert}^{{=: \XXi_{5321}(s,r)}} \\
    & + \bigg\lvert\int_0^r \bigg\{ \EE\left[\int_t^T \psi''(V_{q})^2 K(q,u)^2 K(q,r)^2 \D q\right] + \EE\left[\psi'(V_{\kappa_r})^2 K(\kappa_r,u)^2 \right] \bigg\}^\half (s-u)^{H-\halft}\du\bigg\lvert \\
    &\lesssim \XXi_{5321}(s,r) + \int_0^r \bigg\{\int_t^T K(t,u)^2 K(q,t)^2\D q + K(\kappa_r,u)^2 \bigg\}^\half (s-u)^{H-\halft}\du \\
    &\lesssim \XXi_{5321}(s,r) + \int_0^r K(t,u)  (s-u)^{H-\halft}\du+ \int_0^r K(\kappa_r,u)  (s-u)^{H-\halft}\du \\
    &=: \XXi_{5321}(s,r) + \XXi_{5322}(s,r)+ \XXi_{5323}(s,r).
\end{align*}
Notice that $\XXi_{5321}(s,r)$ is zero if $\kappa_r+2\Delta \ge t$.
$\mathbf{(\XXi_{5322})}$. 
Since $u\le r\le s\le t$, then $K(t,u)\le K(s,r)$  and~$\int_0^r (s-u)^{H-\halft}\du \le K(s,r)$. 
Thus
\begin{align*}
    \int_0^{t_i} \Big(\int_{t_i}^t \int_0^r 
    & K(t,u)(s-u)^{H-\halft}\du \ds\Big) K(t_i,r)\dr
    \lesssim \int_0^{t_i} \left( \int_{t_i}^t K(s,r)^2 \ds\right) K(t_i,r)\dr\\
    &\lesssim \int_0^{t_i} K(t_i,r) \big( (t-r)^{2H} - (t_i-r)^{2H}\big) \dr = \diamondsuit_{2H}(\Delta)=\bigstar(\Delta),
\end{align*}
where in the last step we have exploited Lemma~\ref{lemma:bound_KDeltaK} for $\alpha=2H$.

$\mathbf{(\XXi_{5323})}$.
For all $j=0,\cdots,i-1$, let 
\begin{equation}\label{eq:xi_j}
\xi_j:=\int_{t_i}^t \int_0^{t_j}  K(t_j,u)  (s-u)^{H-\halft}\du \ds.
\end{equation}
We write
\begin{align}\label{Eq:fin_est_Xi5322}
    \int_0^{t_i}&  \left(\int_{t_i}^t \XXi_{5323}(s,r) \ds\right) K(t_i,r)\dr \\
    & =\sum_{j=0}^{i-1} \left(\int_{t_j}^{t_{j+1}} K(t_i,r)\dr\right) \int_{t_i}^t \int_0^{t_j} K(t_j,u)(s-u)^{H-\halft}\du\ds  \\\nonumber
    &\lesssim \sum_{j=0}^{i-2} \Delta^{H+\half} (i-j-1)^{H-\half}  \xi_j
    + \int_{t_{i-1}}^{t_i}K(t_i,r)\dr\, \xi_{i-1}\\
    & \lesssim \Delta^{H+\half} \sum_{j=0}^{i-2}  (i-j-1)^{H-\half} \xi_j+ \Delta^{H+\half}\xi_{i-1}.\nonumber
\end{align}

$\mathbf{(\XXi_{5321})}$.
Another application of the IBP formula to $ \XXi_{5321}(s,r)$ yields
\begin{align}\label{eq_expl_XI5321}
     \XXi_{5321}(s,r)
     &= \int_0^r \int_{\kappa_r+2\Delta}^t \EE\left[ \Df_q^B \left(\phi^{(3)}(X^{t, \overline{X}_t}_T)(\psi^2)''(V_s) \right)\psi''(V_{\kappa_q})\right] K(\kappa_q,u) K(\kappa_q,r) \D q\, (s-u)^{H-\halft}\du ,
\end{align}
where we compute, applying Lemma~\ref{lemma:BoundDiscreteIntegral} to the first term and noticing that the other expectations are uniformly bounded,
\begin{align*}
    \EE\bigg[ \Df_q^B &\left(\phi^{(3)}(X^{t, \overline{X}_t}_T)(\psi^2)''(V_s) \right)\psi''(V_{\kappa_q})\bigg]\\
    & =  \rho \EE\left[(\psi^2)''(V_s)\phi^{(4)}(X^{t, \overline{X}_t}_T)\psi''(V_{\kappa_q})\int_q^t\psi'(V_{\kappa_\ell})K(\kappa_\ell,q)\D B_\ell \right]\\
    &  +  \rho \EE\left[(\psi^2)''(V_s)\phi^{(4)}(X^{t,\overline{X}_t}_T)\psi''(V_{\kappa_q})\int_t^T\psi'(V_{\ell})K(\ell,q)\D B_\ell \right] \\
    &   +   \EE\left[(\psi^2)''(V_s)\phi^{(4)}(X^{t, \overline{X}_t}_T) \psi(V_{\kappa_q}) \psi''(V_{\kappa_q})\right] 
    +   \rho  \EE\left[\phi^{(3)}(X^{t, \overline{X}_t}_T)(\psi^2)^{(3)}(V_s) \psi''(V_{\kappa_q})\right] K(s,q)\\
    &\lesssim (1+K(\kappa_q+\Delta,q)\sqrt{\Delta})\one_{q<t_i}  + 1 + K(s,q).
\end{align*}
This implies, since $1\lesssim K(s,q)$,
\begin{align}
     \XXi_{5321}(s,r)
    & \lesssim \int_0^r  \left(\int_{\kappa_r+2\Delta}^t  \bigg[K(\kappa_q+\Delta,q)\sqrt{\Delta}\one_{q<t_i} +  K(s,q) \bigg] K(\kappa_q,u) K(\kappa_q,r)\D q \right)(s-u)^{H-\halft} \du \\
     & =: \XXi_{53211}(s,r)+  \XXi_{53212}(s,r).
\end{align}
Since, as mentioned before, the integral over the interval $(\kappa_r+2\Delta,t)$ is null if $r\ge t_{i-1}$, we write the following decomposition for the term we want to estimate:
\begin{align}
    \int_0^{t_i} \left(\int_{t_i}^t \XXi_{5321}(s,r) \ds\right) K(t_i,r)\dr
    \lesssim \int_0^{t_{i-1}} \left(\int_{t_i}^t \big(\XXi_{53211}(s,r)+  \XXi_{53212}(s,r)\big) \ds\right) K(t_i,r)\dr,
\end{align}
With the definition of $\xi_j$ in~\eqref{eq:xi_j}. The following estimates hold:
\begin{lemma}\label{lemma:53211}
$\displaystyle \int_0^{t_{i-1}} \left(\int_{t_i}^t \XXi_{53211}(s,r)\ds\right) K(t_i,r)\dr
    \lesssim \Delta^{3H+1} \sum_{j=0}^{i-2} (i-j-1)^{2H} \xi_{j+1}$.
\end{lemma}
\begin{lemma}\label{lemma:53212}
    $\displaystyle\int_0^{t_{i-1}} \left(\int_{t_i}^t \XXi_{53212}(s,r)\ds\right) K(t_i,r)\dr \lesssim \Delta^{3H+\half
    } \sum_{j=0}^{i-2} (i-j-1)^{3H-\half}\xi_{j+1}$
\end{lemma}
\begin{proof}[Proof of Lemma~\ref{lemma:53211}]
With similar arguments as above and observing that~$\int_{\kappa_r+2\Delta}^{t_i}=0$ for all~$r\ge t_{i-2}$, we get 
\begin{align*}
    & \int_0^{t_{i-1}}\int_{t_i}^t\int_0^r (s-u)^{H-\halft} \Bigg\{ \int_{\kappa_r+2\Delta}^t K(\kappa_q,u) K(\kappa_q,r) K(\kappa_q+\Delta,q)\sqrt{\Delta}\one_{q<t_i} \D q \Bigg\} \du \ds K(t_i,r) \dr\\
    & = \sqrt{\Delta}\sum_{j=0}^{i-3} \int_{t_j}^{t_{j+1}} K(t_i,r)  \int_{t_i}^t \int_0^{r}  (s-u)^{H-\halft}\Bigg\{ \sum_{k=j+2}^{i-1}\int_{t_{k}}^{t_{k+1}} K(\kappa_q,u)  K(\kappa_q,r) K(\kappa_q+\Delta,q)\D q  \Bigg\} \du \ds \dr \\
    &=\sqrt{\Delta}\sum_{j=0}^{i-3} \int_{t_j}^{t_{j+1}} K(t_i,r) \sum_{k=j+2}^{i-1} K(t_k,r) \left(\int_{t_i}^t \int_0^{r} (s-u)^{H-\halft} K(t_k,u) \left(\int_{t_{k}}^{t_{k+1}} K(t_{k+1},q)\dq\right)\du\ds \right) \dr\\
    &\lesssim \Delta^{1+H} \sum_{j=0}^{i-3} \left(\int_{t_j}^{t_{j+1}} K(t_i,r) \sum_{k=j+2}^{i-1} K(t_k,r) \dr\right) \left(\int_{t_i}^t \int_0^{t_{j+1}} (s-u)^{H-\halft} K(t_{j+1},u) \du\ds \right) \\
    &=:\Delta^{1+H} \sum_{j=0}^{i-3} \varpi_j \xi_{j+1} 
    \lesssim \Delta^{1+3H} \sum_{j=0}^{i-3} (i-j-1)^{2H} \xi_{j+1},
\end{align*}
where, in the very last step,
we exploited the fact that $\varpi_j$ is bounded in the following way
\begin{align}\label{eq_bound_AAJ}
    \varpi_j
    & := \sum_{k=j+2}^{i-1}\int_{t_j}^{t_{j+1}} K(t_i,r)    K(t_k,r) \dr
    \lesssim K(t_i,t_{j+1})\sum_{k=j+2}^{i-1}   \int_{t_j}^{t_{j+1}}    K(t_k,r) \dr\\
    & \lesssim K(t_i,t_{j+1}) (\Delta (i-j-1))^{H+\half}
    \lesssim (\Delta (i-j-1))^{2H}.
\end{align}

\end{proof}
\begin{proof}[Proof of Lemma~\ref{lemma:53212}]
In a similar fashion as the proof of Lemma~\ref{lemma:53211}, we have~$K(t_k,u)\le K(t_{j+1},u)$ and, since $s\in[t_i,t]$ for $i>k+1$, we have  
$$
    \int_{t_{k}}^{t_{k+1}\wedge t}K(s,q)\D q \lesssim \Delta^{H+1/2} (i-k-1)^{H-1/2},
$$ 
while, for $k=i-1$ and $k=i$, we have  
$$
    \int_{t_{k}}^{t_{k+1}\wedge t}K(s,q)\D q \lesssim \Delta^{H+1/2}.
$$ 
Recalling that $\kappa_r=t_{\tau(r)-1}$ we obtain
\begin{align}
    & \int_0^{t_{i-1}}\int_{t_i}^t\int_0^r (s-u)^{H-\halft} \Bigg\{ \int_{\kappa_r+2\Delta}^t K(\kappa_q,u) K(\kappa_q,r) K(s,q) \D q \Bigg\} \du \ds K(t_i,r) \dr\nonumber\\
    & = \int_{0}^{t_{i-1}} K(t_i,r)  \int_{t_i}^t \int_0^{r} (s-u)^{H-\halft}  \Bigg\{\sum_{k=\tau(r)+1}^{i} K(t_k,u)  K(t_k,r) \left(\int_{t_{k}}^{t_{k+1}\wedge t}K(s,q)\D q\right)  \Bigg\}\du \ds  \dr\nonumber\\
    &\lesssim \sum_{j=0}^{i-2} \int_{t_j}^{t_{j+1}} K(t_i,r) \sum_{k=j+1}^{i} K(t_k,r)  \Bigg\{ \int_{t_i}^t \int_0^{r} (s-u)^{H-\halft} K(t_{j+1},u)   \left(\int_{t_{k}}^{t_{k+1}\wedge t}K(s,q)\D q\right)   \du \ds \Bigg\} \dr\nonumber \\
    &\lesssim \Delta^{H+\half} \sum_{j=0}^{i-3} \int_{t_j}^{t_{j+1}} K(t_i,r) \sum_{k=j+1}^{i-2} K(t_k,r) (i-k-1)^{H-\half}\dr \Bigg\{ \int_{t_i}^t \int_0^{t_{j+1}} (s-u)^{H-\halft} K(t_{j+1},u)    \du \ds \Bigg\} \nonumber\\
    &\qquad + \Delta^{H+\half} \sum_{j=0}^{i-2} \int_{t_j}^{t_{j+1}} K(t_i,r) \Big(K(t_{i-1},r)+ K(t_i,r) \Big)\dr\Bigg\{ \int_{t_i}^t \int_0^{t_{j+1}} (s-u)^{H-\halft} K(t_{j+1},u)    \du \ds \Bigg\}\nonumber\\
    & = \Delta^{H+\half} \sum_{j=0}^{i-3} \int_{t_j}^{t_{j+1}} K(t_i,r) \sum_{k=j+1}^{i-2} K(t_k,r) (i-k-1)^{H-\half}\dr \xi_{j+1} \nonumber\\
    &\qquad + \Delta^{H+\half} \sum_{j=0}^{i-2} \int_{t_j}^{t_{j+1}} K(t_i,r) \Big(K(t_{i-1},r)+ K(t_i,r) \Big)\dr \xi_{j+1},
    \label{eq:proof_53212_est0}
\end{align}
where in the sum the term $j=i-2$ vanishes because~$\sum_{k=j+1}^{i}K(t_k,r)=K(t_i,r)+K(t_{i-1},r)$.

Furthermore, using the same technique as in~\eqref{eq:sum_Beta_function} (with $a=b=H-\half$) to compute the sum, for $j<i-3$,
\begin{align}
\sum_{k=j+1}^{i-2} & \int_{t_j}^{t_{j+1}} K(t_i,r)  K(t_k,r)\dr  (i-k-1)^{H-\half} \nonumber\\
&\le K(t_i,t_{j+1})\sum_{k=j+2}^{i-2}(k-j-1)^{H-\half}\Delta^{H+\half} (i-k-1)^{H-\half} +  K(t_i,t_{j+1}) \Delta^{H+\half} \nonumber\\
&\lesssim K(t_i,t_{j+1})\Delta^{H+\half} (i-j-3)^{2H}\nonumber\\
&\lesssim \Delta^{2H} (i-j-1)^{3H-\half},
\label{eq:proof_53212_est1}
\end{align}
whereas for $j=i-3$ we have $\int_{t_{i-3}}^{t_{i-2}} K(t_i,r)K(t_{i-2},r)\dr\lesssim \Delta^{2H}$. 
Regarding the second term, since $K(t_i,r)\le K(t_{i-1},r)$ we consider
\begin{align}
    \sum_{j=0}^{i-2} \int_{t_j}^{t_{j+1}} K(t_i,r) K(t_{i-1},r)\dr \xi_{j+1}
    &= \sum_{j=0}^{i-3} \int_{t_j}^{t_{j+1}} K(t_i,r) K(t_{i-1},r)\dr\xi_{j+1} + \int_{t_{i-2}}^{t_{i-1}} K(t_i,r) K(t_{i-1},r)\dr \xi_{i-1}\nonumber\\
    &\lesssim \sum_{j=0}^{i-3}(i-j-2)^{2H-1}\Delta^{2H}\xi_{j+1} 
    + \Delta^{2H}\xi_{i-1}.
    \label{eq:proof_53212_est2}
\end{align}
Since $H<1/2$ and $j\le i-3$, we have $(i-j-2)^{2H-1}\le 2^{1-2H} (i-j-1)^{3H-1/2}$. Plugging estimates~\eqref{eq:proof_53212_est1} and \eqref{eq:proof_53212_est2} into \eqref{eq:proof_53212_est0} yields 
\begin{align*}
    &\int_0^{t_{i-1}} \left(\int_{t_i}^t \XXi_{53212}(s,r)\ds\right) K(t_i,r)\dr\\
    &\lesssim \Delta^{3H+\half} \left(\sum_{j=0}^{i-4} (i-j-1)^{3H-\half}\xi_{j+1} +\xi_{i-2} + \sum_{j=0}^{i-3}(i-j-1)^{3H-\half} \xi_{j+1} + \xi_{i-1}\right)
\end{align*}
yields the claim. 
\end{proof}

We must also observe that, again for $j\le i-3$,
\begin{align}
\xi_{j+1}
 & = \sum_{\ell=0}^{j} \int_{t_\ell}^{t_{\ell+1}} \int_{t_i}^t K(t_{j+1},u)(s-u)^{H-\halft}\ds \du \nonumber \\
&\le\sum_{\ell=0}^{j} \int_{t_\ell}^{t_{\ell+1}} \int_{t_i}^t K(t_{j+1},u)(t_i-t_{\ell+1})^{H-\halft}\ds \du \nonumber \\
&= \Delta^{H-\half} \sum_{\ell=0}^{j} (i-\ell-1)^{H-\halft} \int_{t_{\ell}}^{t_{\ell+1}} K({ t_{j+1}},u)\du,
\label{eq:est_xij}
\end{align}
whereas
\begin{align*}
\xi_{i-1}&=\sum_{\ell=0}^{i-2}\int_{t_\ell}^{t_{\ell+1}} K(t_{i-1},u)\big[(t_i-u)^{H-\half} - (t-u)^{H-\half}\big] \du \lesssim \Delta^{2H},
\end{align*}
as they both boil down to~\eqref{eq:decomposition_bound_KDeltaK} in the proof of Lemma~\ref{lemma:bound_KDeltaK}.

Finally, putting together the upper bound~\eqref{Eq:fin_est_Xi5322} and Lemmas~\ref{lemma:53211}, \ref{lemma:53212}, we obtain

\begin{align*}
    &\int_0^{t_i} \int_{t_i}^t \abs{\XXi_{532}(s,r)} \ds K(t_i,r)\dr\\
 &     \lesssim \int_0^{t_i} \left(\int_{t_i}^t \XXi_{5321}(s,r) \ds\right) K(t_i,r)\dr + \int_0^{t_i} \left(\int_{t_i}^t \XXi_{5322}(s,r) \ds\right) K(t_i,r)\dr\\
& \lesssim \Delta^{3H+1} \sum_{j=0}^{i-2} (i-j-1)^{2H} \xi_{j+1} + 
 \Delta^{3H+\half} \sum_{j=0}^{i-2}  (i-j-1)^{3H-\half} \xi_{j+1}\\
 & \lesssim \Delta^{2H+\half} \sum_{j=0}^{i-2} (i-j-1)^{H-\half} \xi_{j+1} + 
 \Delta^{H+\half} \sum_{j=0}^{i-2}  (i-j-1)^{H-\half} \xi_{j+1}\\
 &\lesssim 
 \Delta^{2H} \sum_{j=0}^{i-3}  (i-j-1)^{H-\half} \sum_{\ell=0}^{j} (i-\ell-1)^{H-\halft} \int_{t_{\ell}}^{t_{\ell+1}} K({ t_{j+1}},u)\du + \Delta^{3H+\half}
    \lesssim  \bigstar(\Delta),
\end{align*}
where we used $\Delta^\alpha(i-j-1)^\alpha\le T^\alpha$ with $\alpha=H-\half$ and~$\alpha=2H$, the estimate~\eqref{eq:est_xij} and (a mild variation of) Lemma~\ref{lemma:doublesum} to conclude.

\end{proof}

\section{Proof of Proposition~\ref{prop:Btwo}}\label{sec:Btwo}
The second term to deal with is defined as follows:
\begin{align*}
    \frB_2(t)
    & = \EE\Big[\big(\psi(V_t)-\psi(V_{t_i})\big) \langle \partial_\omega (\partial_x \bar{u}_t),K^t \rangle \Big]\\
    & = \EE\bigg[\big(\psi(V_t)-\psi(V_{t_i})\big)\EE_t\Big[\phi^{(2)}(\XTb)\int_t^T \psi'(V_s)K(s,t)\D B_s\Big]\bigg].
\end{align*}
We pursue the same approach as for~$\frB_1$, using Clark-Ocone formula and integration by parts. Thus, in view of this, {for $r \leq t$,} we define
\begin{align*}
    { \Vv_r}
    & :=\Df_r^W \langle \partial_\omega (\partial_x \bar{u}_t),K^t \rangle\\
    & = \EE_t \Bigg[ \phi^{(3)}(\XTb)\bigg(\int_r^T \psi'(V_{\kappa^t_s})K(\kappa^t_s,r)\D B_s + \rho\psi(V_{\kappa_r})\bigg)\int_t^T \psi'(V_s)K(s,t)\D B_s\Bigg]\\
    &\quad + \EE_t \Bigg[ \phi^{(2)}(\XTb)\int_t^T \psi'(V_{s})K(s,r)K(s,t)\D B_s \Bigg]
     =: \Vv_r^{(1)} + \Vv_r^{(2)}.
\end{align*}
We start by discussing the term $\Vv_2$, which requires one more application of IBP:
\begin{align}\label{eq:IBP_VTwo_r}
    \Vv_r^{(2)} & = \EE_t \Bigg[ \phi^{(2)}(\XTb)\int_t^T \psi'(V_{s})K(s,r)K(s,t)\D B_s \Bigg]\nonumber\\
    & =\EE_t\left[\int_t^T \phi^{(3)}(\XTb)\Df_s^B \XTb\psi'(V_{s})K(s,r)K(s,t)\ds\right]\nonumber\\
    &=\EE_t\left[\int_t^T \phi^{(3)}(\XTb)\bigg\{\int_s^T \psi'(V_v)K(v,s)\D B_v + \psi(V_s)\bigg\}\psi'(V_{s})K(s,r)K(s,t)\ds\right].
\end{align}
With a slight abuse of notations, let $\varphi_r(t):=\EE_r[\psi'(V_t)]$
(in the previous section this denoted the derivative of~$\psi^2$). 
Similarly to $\frB_1$, Clark-Ocone and integration by parts yield
\begin{align*}
    & \frB_2(t) = \EE\big[ \langle \partial_\omega (\partial_x \bar{u}_t),K^t \rangle\big] \EE[\psi(V_t)-\psi(V_{t_i})]
    + \EE\left[\int_0^t \Vv_r \EE_r\Big[ \psi'(V_t)K(t,r) - \psi'(V_{t_i}) K(t_i,r) \Big] \dr\right]\\
    & = \EE\big[ \langle \partial_\omega (\partial_x \bar{u}_t),K^t \rangle\big] \EE[\psi(V_t)-\psi(V_{t_i})]
    + \EE\left[\int_0^t  \Vv_r^{(1)} \EE_r\Big[ \psi'(V_t)K(t,r) - \psi'(V_{t_i}) K(t_i,r) \Big] \dr\right] \\
    & + \EE\left[\int_0^t  \Vv_r^{(2)} \EE_r\Big[ \psi'(V_t)K(t,r) - \psi'(V_{t_i}) K(t_i,r) \Big] \dr\right] \\
    &\lesssim  \EE\big[ \langle \partial_\omega (\partial_x \bar{u}_t),K^t \rangle\big] \EE[\psi(V_t)-\psi(V_{t_i})]
    + \EE\left[\int_0^t  \Vv_r^{(1)} \EE_r\Big[ \psi'(V_t)K(t,r) - \psi'(V_{t_i}) K(t_i,r) \Big] \dr\right] \\
    & + \abs{\EE \left[\Vv_2(t) \varphi_t(t)\right]\int_0^t \big( K(t,r)-K(t_i,r)\big) \dr} 
    + 
    \abs{\EE\left[ \Vv_2(t) \int_0^t \big(\varphi_r(t)-\varphi_t(t)\big) \big(K(t,r)-K(t_i,r)\big)\dr\right]} \\
    &+ \abs{\EE \left[\int_0^t \varphi_r(t) \big(\Vv_r^{(2)} -\Vv_2(t)\big) \big(K(t,r)-K(t_i,r)\big)\dr \right]}
    + \abs{\EE\left[\int_0^{t_i} \Vv_r^{(2)} \big(\varphi_r(t)-\varphi_r(t_i)\big)  K(t_i,r)\dr\right]}\\
    &=  \sum_{k=1}^6 \abs{\Upsilon_k}.
\end{align*}
The first and third terms are similar to $\XXi_1$ and~$\XXi_2$ and the following mimics Lemma~\ref{lemma:Xi1}:
\begin{lemma}
    $\abs{\Upsilon_1}+\abs{\Upsilon_3} \lesssim t^{2H}-t_i^{2H}$.
\end{lemma}

\begin{lemma}
    $\abs{\Upsilon_2} \lesssim \bigstar(\Delta)$.
\end{lemma}
\begin{proof}[Sketch of proof]
The first term can be treated as $\Uu_r$ (in the discussion for $\frB_1$) since it only differs by a smooth $L^p$ random variable. 
For the sake of preventing this proof from being too long, we do not give full details but highlight crucial points. 
Whenever IBP is necessary, an additional term is generated compared to $\frB_1$. 
It is however almost identical to the term that differentiates the other integral on the interval $(t,T)$, hence is treated in the same way. 
Indeed, since $K(q,t)=0$ for $q<t$ we have
\begin{align*}
   \Df_q^W \int_t^T \psi'(V_s)K(s,t)\D B_s = \int_t^T \psi''(V_s) K(s,q)K(s,t)\D B_s. 
\end{align*}
This same remark holds for higher-order Malliavin derivatives too. 
When comes the time to bound the stochastic terms, the additional multiplication by the extra integral does not matter because all the terms are bounded in $L^p$.
\end{proof}

Now, we present the following technical lemma, which is used multiple times.
\begin{lemma}\label{lem:ZVTwoBound}
    For all $Y\in L^2$, and any $0\leq r\leq t_i\leq t$, 
    $\displaystyle \EE[|Y\Vv_r^{(2)}|]\leq \int_{t}^{T}K(s,r)K(s,t)\D s$.
\end{lemma}
\begin{proof}
Straightforward computations and Cauchy-Schwarz inequality yield
\begin{align*}
    & \EE[|Y \Vv_r^{(2)}|]\\
    & = \EE\Bigg[\left|Y \int_t^T \EE_t\left[ \phi^{(3)}(\XTb)\bigg\{\int_s^T \psi'(V_u)K(u,s)\D B_u + \psi(V_s)\bigg\}\psi'(V_{s})\right] K(s,r)K(s,t)\ds\right|\Bigg]\\
    & = \EE\Bigg[\left|Y \sup_{s \in [t,T]}\EE_t\left[ \phi^{(3)}(\XTb)\bigg\{\int_s^T \psi'(V_u)K(u,s)\D B_u + \psi(V_s)\bigg\}\psi'(V_{s})\right] \right|\Bigg] \int_t^T K(s,r)K(s,t)\ds\\ 
    & \lesssim \int_t^T K(s,r)K(s,t)\ds.
\end{align*}
\end{proof}

\begin{lemma}
    $\displaystyle \abs{\Upsilon_4}\lesssim \bigstar(\Delta)$.
\end{lemma}
\begin{proof}
As in~\eqref{eq:CO_varphi}-\eqref{eq:IBP_Uu}, we compute~$\EE\big[\Vv_2(t)(\varphi_r(t)-\varphi_t(t))\big]$ by  Clark-Ocone and IBP:
\begin{align*}
    & \Upsilon_4 
    = \int_0^t \EE\left[\int_r^t \Df_q^W \Vv_2(t) \EE_q[\psi''(V_t)]K(t,q)\dq \right] \big(K(t,r)-K(t_i,r)\big) \dr \\
    &= \EE\int_0^t \int_r^t \Bigg\{ \EE_t\left[\int_t^T \phi^{(4)}(\XTb)
    \Df_q^W \XTb
    \Df_s^B \XTb \psi'(V_{s})K(s,t)^2\ds\right]  \\
    &+ \EE_t\left[\int_t^T \phi^{(3)}(\XTb)
    \bigg\{ \int_s^T \psi''(V_v) K(v,q) K(v,s) \D B_v + \psi'(V_s)K(s,q)
    \bigg\} \psi'(V_{s})K(s,t)^2\ds\right]\\
    & + \EE_t\left[\int_t^T \phi^{(3)}(\XTb)
    \Df_s^{B} \XTb \psi''(V_{s})K(s,q)K(s,t)^2\ds\right]\Bigg\}\EE_q[\psi''(V_{t})]  K(t,q)\dq [K(t,r)-K(t_i,r)]\dr.
    \end{align*}
    where, using representation~\eqref{eq:IBP_VTwo_r} and $q\le t$
    \begin{align*}
        &\Df_q^W \Vv_2(t)
        = \EE_t\left[\int_t^T \phi^{(4)}(\XTb)
    \Df_q^W \XTb
    \Df_s^B \XTb \psi'(V_{s})K(s,r)K(s,t)\ds\right] \\
    &\quad + \EE_t\left[\int_t^T \phi^{(3)}(\XTb)
    \bigg\{ \int_s^T \psi''(V_v) K(v,q) K(v,s) \D B_v + +\psi'(V_s)K(s,q)
    \bigg\} \psi'(V_{s})K(s,t)^2\ds\right]\\
    &\quad +  \EE_t\left[\int_t^T \phi^{(3)}(\XTb)
    \Df_s^{{B}} \XTb \psi''(V_{s})K(s,q)K(s,t)^2\ds\right].
    \end{align*}
In the first term, since $s \in [t,T]$, $\Df_s^{ {B}} \XTb$ has bounded moments but the same does not hold for $\Df_q^{{W}} \XTb$ as $0\leq r \leq  q \leq t$. Using Lemma~\ref{lemma:BoundDiscreteIntegral} with $Y= \phi^{(4)}(\XTb)    \Df_s^B \XTb \psi'(V_{s})$ then
\begin{align*}
\EE\left[\abs{\phi^{(4)}(\XTb) 
    \Df_s^B \XTb \psi'(V_{s})} \abs{\Df_q^W \XTb}\right] \lesssim \left(1+\sqrt{\Delta} K(\kappa_q+\Delta,q)\right)\one_{q< t_i}.
\end{align*}
We also apply Cauchy-Schwarz inequality to the second line as follows:
\begin{align}\label{eq:CS_forKK}
    &\EE\left[\abs{Z}\abs{\int_s^T\psi''(V_v)K(v,q)K(v,s)\D B_v}\right]^2\\
    &\le \EE\left[\abs{Z}^2\right]
    \EE\left[\int_s^T \abs{\psi''(V_v)K(v,q)K(v,s)}^2\dv\right]\\ 
    &\le \EE\left[\abs{Z}^2\right] \sup_{v\in[s,T]}\EE\left[\abs{\psi''(V_v)}^2\right] K(s,q)^2\int_s^T K(v,s)^2\dv
    \lesssim K(s,q)^2,\nonumber
\end{align}
where $Z=\phi^{(3)}(\XTb)\psi'(V_s) {\EE_q[\psi''(V_{t})] }$.
Exploiting the fact that~$1\le K(s,q)\le K(t,q)$, we obtain 
\begin{align*}
    { \Upsilon_4}&\lesssim 
    \int_0^t \int_r^t \bigg\{ \int_t^T \Big(1+\sqrt{\Delta} K(\kappa_q+\Delta,q)\one_{q<t_i}\Big) K(s,t)^2\ds \bigg\}  K(t,q)\dq\, \Big|K(t,r)-K(t_i,r)\Big|\dr\\
    &\quad +3  \int_0^t \int_r^t \bigg\{\int_t^T K(s,q) K(s,t)^2 \ds \bigg\}  K(t,q)\dq\, \Big|K(t,r)-K(t_i,r)\Big|\dr\\
    &\lesssim   \int_0^{t_i} \int_r^{t_i}  \sqrt{\Delta} K(\kappa_q+\Delta,q)   K(t,q)\dq\, \Big|K(t,r)-K(t_i,r)\Big|\dr\\
    & \quad + \int_0^t \int_r^t   K(t,q)^2\dq\, \Big|K(t,r)-K(t_i,r)\Big|\dr\\
    &=: { \Upsilon_{41} + \Upsilon_{42}}.
\end{align*}
By integration and Lemma~\ref{lemma:Gassiat_log}, we obtain 
$$
{ \Upsilon_{42}} \lesssim \int_0^t (t-r)^{2H}\Big(K(t_i,r)-K(t,r)\Big)\dr
\lesssim \bigstar(\Delta).
$$
For $r\ge t_{i-1}$, we have $ \int_r^{t_i} K(\kappa_q+\Delta,q) K(t,q)\dq\le \int_r^{t_i} K(t_i,q)^2\dq\lesssim \Delta^{2H}$.
Then, for~$q\in[t_\ell,t_{\ell+1})$, we have~$K(t,q)\le \big(\Delta(i-\ell-1)\big)^{H-\half}$, so that, for $r<t_{i-1}$,
\begin{align*}
    \int_r^{t_i} & K(\kappa_q+\Delta,q) K(t,q)\dq
    = \sum_{\ell=\tau(r)-1}^{i-2} \int_{t_\ell \vee r}^{t_{\ell+1}} K(t_{\ell+1},q)K(t,q)\dq + \int_{t_{i-1}}^{t_i} K(t_{i},q)K(t,q)\dq  \\
    &\lesssim \sum_{\ell=\tau(r)-1}^{i-2} \big(\Delta(i-\ell-1)\big)^{H-\half} \int_{t_\ell \vee r}^{t_{\ell+1}} K(t_{\ell+1},q)\dq + \int_{t_{i-1}}^{t_i} K(t_{i},q)^2\dq \\
    &\lesssim \Delta^{2H}\sum_{\ell=\tau(r)-1}^{i-2} (i-\ell-1)^{H-\half} + \Delta^{2H}\\
    &\lesssim \Delta^{2H}(i-\tau(r))^{H+\half}+ \Delta^{2H}
    \lesssim \Delta^{H-\half} (t_i-\kappa_r)^{H+\half} +\Delta^{2H}.
\end{align*}
Finally, we conclude thanks to monotonicity of integration and Lemma~\ref{lemma:Bound_kappar_DeltaK} 
\begin{align*}
    { \Upsilon_{41}}
    &\lesssim \Delta^H \int_0^{t_{i}} (t_{i}-\kappa_r)^{H+\half}\Big[K(t_i,r)-K(t,r)\Big]\dr + \Delta^{2H+\half} \int_{t_{i-1}}^{t_i} \Big[K(t_i,r)-K(t,r)\Big]\dr \\
    &\lesssim \Delta^H\bigstar(\Delta) + \Delta^{1+3H}.
\end{align*}
\end{proof}
We are left to bound $\Upsilon_5$ and $\Upsilon_6$.
The first one is dealt with the following lemma.
\begin{lemma}
    $\abs{\Upsilon_5}\lesssim \bigstar(\Delta)$.
\end{lemma}
\begin{proof}
For all $Y\in L^2$ and $r\le t$, with the same computations as in Lemma~\ref{lem:ZVTwoBound},
\begin{align*}
\EE Y\abs{\Vv_r^{(2)}-\Vv_2(t)}
& \lesssim \int_t^T K(s,t) 
\big[K(s,t) - K(s,r)\big]\ds \\
& \le \int_t^T \big( K(s,t)^2 - K(s,r)^2\big) \ds
 \lesssim (t-r)^{2H},
\end{align*}
so that, by Lemma~\ref{lemma:Gassiat_log},
$\displaystyle
\abs{\Upsilon_2} \lesssim \abs{\int_0^t (t-r)^{2H} \big(K(t,r)-K(t_i,r)\big)\dr } \lesssim \bigstar(\Delta)$.
\end{proof}
The last term is a tad more tedious and is detailed in the proof of the next lemma:
\begin{lemma}
    $\abs{\Upsilon_6}\lesssim \bigstar(\Delta)$.
\end{lemma}
\begin{proof}
    Using Lemma~\ref{lemma:Delta_varphi}, we rewrite and decompose $\Upsilon_6$ as
\begin{align*}
    \Upsilon_6 &= \left(H-\half\right)\EE\int_0^{t_i} \Vv_r^{(2)} \left( \int_{t_i}^t \EE_r[\psi''(V_{q})]\int_0^r (q-u)^{H-\halft}\D W_u \dq \right) K(t_i,r)\dr \\
    &\qquad + 
    \EE\int_0^{t_i} \Vv_r^{(2)}  \left( \int_{t_i}^t \half\EE_r[\psi^{(3)}(V_{q})] (q-r)^{2H-1} \dq \right) K(t_i,r)\dr 
    =: \Upsilon_{61}+\Upsilon_{62}.
\end{align*}

$\mathbf{(\Upsilon_{62}})$ An application of Lemma~\ref{lem:ZVTwoBound}, with $Y= \sup_{q\in[t_i,t]} \EE_r[\psi^{(3)}(V_{q})]$, yields
\begin{align*}
   \EE[|Y\Vv_r^{(2)}|]\lesssim \int_t^T K(s,r)K(s,t)\ds \le \int_t^T K(s,t)^2\ds =\frac{(T-t)^{2H}}{2H}, 
\end{align*}
that in turn, by Lemma~\ref{lemma:bound_KDeltaK}, entails
\begin{align*}
    \abs{\Upsilon_{62}}
    & \lesssim \int_0^{t_i} \left(\int_{t_i}^t (q-r)^{2H-1} \dq \right) K(t_i,r)\dr\\
    & \lesssim \int_0^{t_i} \left|(t-r)^{2H} - (t_i-r)^{2H}\right| K(t_i,r)\dr
    \lesssim \bigstar(\Delta).
\end{align*}
$\mathbf{(\Upsilon_{61}})$ This term requires IBP with respect to $W$. In particular, we rewrite it as
\begin{align*}
    \Upsilon_{61} = \left(H-\half\right)
    \int_0^{t_i} \int_{t_i}^t \int_0^r  \EE\Big[\Df^W_u \Big(\Vv_r^{(2)} \EE_r[\psi''(V_{q})]\Big) \Big] (q-u)^{H-\halft} \du \dq K(t_i,r)\dr.
\end{align*}
For all $u\in[0,r)$, we recall that
\begin{equation}\label{eq:DuXT}
    \Df_u^W \XTb = \int_u^T \psi'(V_{\kappa^t_v})K(\kappa_v^t,u)\D B_v + \rho\psi(V_{\kappa^t_u}),
\end{equation}
and, combined with the equivalent representation of $\Vv_r^{(2)}$ in~\eqref{eq:IBP_VTwo_r}, we obtain  {for all~$0\le u<r<t_i<q\le t$}
\begin{align}\label{eq:DV2}
    & \Df_u^W \Big(\Vv_r^{(2)} \EE_r[\psi''(V_{q})]\Big) \\
    &= \EE_t\left[\int_t^T \phi^{(4)}(\XTb)
    \Df_u^W \XTb
    \Df_s^B \XTb \psi'(V_{s})K(s,r)K(s,t)\ds\right]\EE_r[\psi''(V_{q})]\\
    &\quad+ \EE_t\left[\int_t^T \phi^{(3)}(\XTb) \left(\int_s^T \psi''(V_v) K(v,u) K(v,s) \D B_v \right)\psi'(V_{s})K(s,r)K(s,t)\ds\right]\EE_r[\psi''(V_{q})]\\
     &\quad+ \EE_t\left[\int_t^T \phi^{(3)}(\XTb) \psi'(V_s)K(s,u)\psi'(V_{s})K(s,r)K(s,t)\ds\right]\EE_r[\psi''(V_{q})]\\
    &\quad + \EE_t\left[\int_t^T \phi^{(3)}(\XTb)
    \Df_s^{B} \XTb \psi''(V_{s})K(s,u)K(s,r)K(s,t)\ds\right]\EE_r[\psi''(V_{q})] \label{eq:DV2_line3}\\
    &\quad + \EE_t\left[\int_t^T \phi^{(3)}(\XTb)
    \Df_s^{B} \XTb \psi'(V_{s}) K(s,r)K(s,t)\ds\right]\EE_r[\psi^{(3)}(V_{q})] K(q,u),\label{eq:DV2_line4}
\end{align}
where in the second term~$K(u,v)=0$ since $u\le r\le t\le s\le v$.\\
In order to bound $\Df_u^{W} (\Vv_r^{(2)} \EE_r[\psi''(V_{q})])$ as per~\eqref{eq:DV2}, 
we apply Lemma~\ref{lemma:BoundDiscreteIntegral} to the first line~\eqref{eq:DV2} with 
$
    Y
    =\phi^{(4)}(\XTb)
    \Df_s^B \XTb \psi'(V_{s})\EE_r[\psi''(V_{q})] 
$, which entails
\begin{align*}
    \EE\left[\abs{\phi^{(4)}(\XTb)\EE_r[\psi''(V_{q})] 
    \Df_s^B \XTb \psi'(V_{s})} \abs{\Df_u^W \XTb}\right] \lesssim \big(1+\sqrt{\Delta} K(\kappa_u+\Delta,u)\big).
\end{align*}
We also apply Cauchy-Schwarz inequality to the second line~\eqref{eq:DV2} as in~\eqref{eq:CS_forKK}, yielding
\begin{align*}
\EE\left[\abs{\phi^{(3)}(\XTb)\psi'(V_s)\EE_r[\psi''(V_q)]}\abs{\int_s^T\psi''(V_v)K(v,u)K(v,s)\D B_v}\right]^2\lesssim K(s,u)^2.
\end{align*}
For the last three lines we simply observe that the stochastic terms are uniformly bounded in~$L^p$ and $K(s,u)\le K(q,u)$, hence we obtain
\begin{align*}
& \abs{\EE\Big[\Df_u \Big(\Vv_r^{(2)} \EE_r[\psi''(V_{q})]\Big) \Big]}\\
& \lesssim \int_t^T \left\{1+\sqrt{\Delta} K(\kappa_u+\Delta,u)\right\}
K(s,r) K(s,t)  \ds 
    +4 K(q,u)\int_t^T  K(s,r) K(s,t)  \ds \\
&=: \Upsilon_{611}(u,r) + \Upsilon_{612}(u,r,q).
\end{align*}
$\mathbf{(\Upsilon_{611}})$
Exploiting again $K(s,r)\le K(s,t)$, we get that~$ \Upsilon_{611}(u,r)\lesssim 1+\sqrt{\Delta} K(\kappa_u+\Delta,u)$. 
This term then boils down to~\eqref{eq:Xi_531}, which gives us
\begin{align}
&\int_0^{t_i}\int_{t_i}^t \int_0^r \Upsilon_{611}(u,r)(q-u)^{H-\halft}\du\D q K(t_i,r)\dr  
\lesssim \int_{t_i}^t\int_0^r \Big(1+\sqrt{\Delta}K(\kappa_u+\Delta,u)\Big) (q-u)^{H-\halft}\du \dq,
\end{align}
which can bounded by $\Delta^{1+H}$ as in Equation~\eqref{eq:Xi_531} in the proof of Lemma~\ref{lem:Bound531}.

$\mathbf{(\Upsilon_{612}})$ Since $K(s,r)\le K(s,t)$, by Lemma~\ref{lemma:bound_KDeltaK}, we obtain
\begin{align*}
& \int_0^{t_i} \int_{t_i}^t \int_0^r \Upsilon_{612}(u,r,q) (q-u)^{H-\halft} \du \dq K(t_i,r)\dr\\
&\lesssim \int_0^{t_i} \int_{t_i}^t \int_0^r \left(K(q,u)\int_t^T K(s,t)^2\ds\right) (q-u)^{H-\halft}\du \dq K(t_i,r)\dr \\
&\lesssim \int_0^{t_i} \int_{t_i}^t (q-r)^{2H-1} \dq K(t_i,r)\dr
    \lesssim \bigstar(\Delta).
\end{align*}
\end{proof}

\section{Proofs of Section~\ref{sec:RvolApplication}}\label{sec:proof_RvolApplication}

\subsection{Proof of Lemma~\ref{lemma:Bounds_XV}}\label{sec:Proof_Bounds_XV}
Let~$p>0$.
First, $V$ is a Gaussian process hence~$\EE\left[\E^{p \norm{V}_{\TT}}\right]$ is finite~\cite[Lemma 6.13]{jacquier2021rough}, proving the first bound of~\eqref{eq:Bounds_XV}. 
Now, by Gaussian computations,
$$
\EE\left[\E^{pV_s^{t,\omega}}\right] 
= \E^{p\omega_t} \EE\left[ \exp\left(p \int_t^s K(t,r)\,\D W_r\right)\right]
= \exp\left(p\omega_t + \frac{p^2 (s-t)^{2H}}{4H}\right),
$$
yielding the second estimate.
For the last one, by BDG and Jensen's inequalities, we have
\begin{align*}
    \EE\left[\sup_{s \in [t,T]}\abs{X_s^{t,x,\omega}}^p\right]
    & \le 3^{p-1} \EE\left[\abs{x}^p +\sup_{s\in[t,T]}\left(\int_t^s \psi(r,V_r^{t,\omega})\,\D B_r\right)^p + \zeta^p \left(\int_0^T\psi(r,V_r^{t,\omega})^2\dr\right)^p\right]\\
    &\le  3^{p-1}    \EE\left[
\abs{x}^p +\bdg_p \left(\int_t^T \psi(r,V_r^{t,\omega})^{2} \dr\right)^{p/2} + \zeta^p T^{p-1}\int_t^T \abs{\psi(r,V_r^{t,\omega})}^{2p} \dr \right] \\
    &\le 3^{p-1} \EE\left[\abs{x}^p +\bdg_p T^{p/2-1} \int_0^T\abs{\psi(r,V_r)}^p\ds + \zeta^p T^{p-1} \int_0^T \abs{\psi(r,V_r)}^{2p}\dr\right],
\end{align*}
where~$\bdg_p$ is the BDG constant. Furthermore, Assumption~\ref{assu:phipsi}(ii) gives
$$
\EE\left[\int_0^T \abs{\psi(r,V_r)}^{2p}\dr\right]
\le \sup_{r\in\TT} \EE\left[\E^{2p\kpsi\left(r+V^{t,\omega}_r\right)}\right].
$$
This, combined with the second estimate, yields the claim.
By assumption there exists $\ell>0$ such that
$G(X_t,V_t) \lesssim 1+ \abs{X_t}^\ell + \E^{\ell V_t}$.
From~\eqref{eq:Bounds_XV} that~$\EE\left[\norm{G(X,V)}_{\TT}\right]$ is then finite.

\subsection{Proof of Proposition~\ref{prop:SpatialDerivatives}} \label{sec:Proof_SpatialDerivatives}\footnote{The authors would like to thank Jean-François Chassagneux for pointing out this nice method of proof.}
    For all~$(s,(t,x,\omega))\in\TT\times\overline{\Gamma}$, such that $t <s$, recall~$V^{t,\omega}_s$ and~$ X_T^{t,x,\omega}$ defined in~\eqref{eq:Xtxomega}. \\
    
    $\bullet$ We start by considering the derivatives in $x$ and for clarity we note~$X^{x}_T=X_T^{t,x,\omega}$, with~$(t,\omega)$ fixed. For any~$\ep>0$, notice that~$X^{x+\ep}_T-X^x_T=\ep$ hence
    \begin{equation}\label{eq:Diff_proof_derux}
    \phi(X_T^{x+\ep})-\phi(X_T^x)
    = (X^{x+\ep}_T-X^x_T)\int_0^1 \phi'(\lambda X^{x+\ep}_T + (1-\lambda) X^{x}_T) \D \lambda
    =\ep \int_0^1 \phi'(X^{x}_T + \lambda \ep) \D \lambda.
    \end{equation}
    Since $\phi'\in \Cc^1$ it is also locally Lipschitz continuous which entails
    \begin{align*}
        \EE_{t,\omega}\left[ \frac{\phi(X_T^{x+\ep})-\phi(X_T^x)}{\ep} - \phi'(X^x_T)\right]
        &=\EE_{t,\omega}\left[ \int_0^1 \phi'(X^{x}_T + \lambda \ep) - \phi'(X^{x}_T )\D \lambda\right] \\
        &\leq\EE_{t,\omega}\left[ \sup_{\ep'\in(0,\ep)}\abs{\phi''(X^{x+\ep'}_T)}\right] \int_0^1 \lambda\ep \D \lambda,
    \end{align*}
    which tends to zero as~$\ep\downarrow0$ by Assumption~\ref{assu:phipsi}(i) and~\eqref{eq:Bounds_XV}, as
    \begin{align*}
        \EE_{t,\omega}\left[ \sup_{\ep'\in(0,\ep)}\abs{\phi''(X^{x+\ep'}_T)}\right]
        & = \EE_{t,\omega}\left[ \sup_{\ep'\in(0,\ep)}\abs{\phi''(X^{x}_T +\ep')}\right]\\
        & \lesssim \EE_{t,\omega}\left[ \sup_{\ep'\in(0,\ep)}(1+|X^{x}_T +\ep'|^\kphi)\right]\\
        &\lesssim 1 + \sup_{\ep'\in(0,\ep)} (\ep')^\kphi + \EE_{t,\omega}\left[ |X^{x}_T |^\kphi\right] 
        < \infty.
    \end{align*}

    This proves the expression of the first derivative. 
    The second space derivative is proved in the same fashion since~$\phi''$ also belongs to~$\Cc^1$ by assumption.\\

    $\bullet$
        We turn to the Fréchet derivative; by a mild abuse of notation we will write~$\eta$ instead of~$\eta\one_{[t,T]}$ in the remainder of this proof. 
    We fix $(t,x)$ and write~\eqref{eq:Xtxomega} as~$X^\omega_T$. For clarity, we only show the proof in the case~$\zeta=0$, the general case being analogous. Recall the definition of the pathwise derivative in~\eqref{eq:PathwiseDerDef}, thus the proof amounts to showing that
    $$
    \lim_{\norm{\eta}_{\TT}\downarrow0} \frac{u(t,x,\omega+\eta)-u(t,x,\omega)-\EE_{t,x,\omega}\left[\phi'(X_T) \left\{\int_t^T \psi'(s,V_s) \eta_s \,\D B_s\right\}\right]}{\norm{\eta}_{\TT}} =0.
    $$
    We then write
   \begin{equation}
        \phi(X_T^{\omega+\eta})-\phi(X_T^\omega)
    =\delta_{\eta} X^\omega_T \,\wphi(\eta) ,
    \label{eq:Diff_proof_Dphi}
   \end{equation}
    where~$\delta_{\eta} X^\omega_T:= X^{\omega+\eta}_T-X^\omega_T$ and $\wphi(\eta):=\int_0^1 \phi'(X^{\omega}_T + \lambda \delta_{\eta} X^\omega_T) \D \lambda$. Therefore,
     \begin{align}
        &\frac{1}{\norm{\eta}_{\TT}}\EE_{t,x}\left[ \phi(X_T^{\omega+\eta})-\phi(X_T^\omega) - \phi'(X^\omega_T)\int_t^T \psi'(s,V^\omega_s)\eta_s\,\D B_s\right]\nonumber \\
        &=\frac{1}{\norm{\eta}_{\TT}}\EE_{t,x}\left[ \wphi(\eta) \left(\delta_{\eta} X^\omega_T- \int_t^T \psi'(s,V^\omega_s) \eta_s\,\D B_s\right)\right] \label{eq:Diff_proof_deruw_one}\\
        & \qquad + \frac{1}{\norm{\eta}_{\TT}}\EE_{t,x}\left[\left(\wphi(\eta)-\phi'(X^\omega_T)\right)\int_t^T \psi'(s,V^\omega_s)\eta_s\,\D B_s \right].
        \label{eq:Diff_proof_deruw_two}
    \end{align}   
    To conclude, we need to show that the terms on the right-hand side tend to zero as $\norm{\eta}_{\TT}$ goes to zero. We can thus restrict to $\norm{\eta}_{\TT}\le 1$. Let us start by considering the first one.
    On the one hand, for all~$\lambda\in(0,1)$, Assumption~\ref{assu:phipsi} gives
    \begin{align}
        \EE_{t,x}\big[\phi'(\lambda X^\omega_T + (1-\lambda) X^{\omega+\eta}_T)^2 \big] 
        &\lesssim 1+ \EE_{t,x}\left[\abs{ X^{\omega}_T}^{2\kphi} +  \abs{X^{\omega+\eta}_T}^{2\kphi}\right],
        \label{eq:Diff_proof_deruw_1}
    \end{align}
    which is bounded by~\eqref{eq:Bounds_XV} and in turn yields the finiteness of the second moment of $\widetilde{\phi}(\eta)$.
    Moreover, we notice that~$V^{\omega+\eta}_s-V^{\omega}_s=\eta_s$ which yields
    \begin{align}\label{eq:DeltaXomega}
        \delta_{\eta} X^\omega_T 
        &= \int_t^T \psi(s,V^{\omega+\eta}_s)-\psi(s,V^{\omega}_s) \,\D B_s \\
        &= \int_t^T\left( (V^{\omega+\eta}_s-V^{\omega}_s)\int_0^1 \psi'\big(s,V^{\omega}_s + \lambda (V^{\omega+\eta}_s-V^{\omega}_s) \big) \,\D \lambda\right) \D B_s \\
        & = \int_t^T \eta_s \wpsi_s(\eta) \D B_s, \nonumber
    \end{align}
    where~$\wpsi_s(\eta):= \int_0^1 \psi'\big(s,V^{\omega}_s + \lambda \eta_s\big) \,\D \lambda$. 
    Now, the local Lipschitz continuity of $\psi'(s, \cdot)$ gives
    \begin{equation}\label{eq:est_dif_psi}
       \wpsi_s(\eta)- \psi'(s,V^\omega_s) 
    = \int_0^1 \Big[\psi'(s,V^\omega_s+\lambda\eta_s) -\psi'(s,V^\omega_s)\Big] \D \lambda 
    \le \half\sup_{\norm{\tau}_{\TT}\le1}\abs{\psi''(s,V^\omega_s+ \tau_s)}  |\eta_s| .
    \end{equation}
    For all $p>1$, again by Assumption~\ref{assu:phipsi} and~\eqref{eq:Bounds_XV}, we have
    \begin{align}
       \EE_{t,x} \left[\sup_{\norm{\tau}_{\TT}\le1}\abs{\psi''(s,V^\omega_s+\tau_s)}^p\right]
       & \lesssim \EE_{t,x} \left[\sup_{\norm{\tau}_{\TT}\le1}(1+ \E^{\kpsi(V^\omega_s+ \tau_s)})^p\right]
       \lesssim 1 + \sup_{\norm{\tau}_{\TT}\le1} \E^{\kpsi p \tau_s }\EE_{t,x} \left[\E^{\kpsi p V^\omega_s }\right]\\ \nonumber
       & \lesssim 1 + \E^{\kpsi p }\EE \left[\E^{\kpsi p V^\omega_s }\right]
       <\infty.
    \end{align} 
    This allows us to show, exploiting~\eqref{eq:DeltaXomega} and It\^{o}'s isometry first and then~\eqref{eq:est_dif_psi}, 
\begin{align}
\frac{1}{\norm{\eta }_{\TT}^2}
& \EE_{t,x}\left[ \left(\delta_{\eta} X^\omega_T - \int_t^T \psi'(s,V^\omega_s) \eta_s\,\D B_s\right)^2\right]\\
        &= \frac{1}{\norm{\eta }_{\TT}^2}\EE_{t,x}\left[ \left(\int_t^T\big( \wpsi_s(\eta)- \psi'(s,V^\omega_s) \big)\eta_s\,\D B_s\right)^2\right] \nonumber\\
        &=\frac{1}{\norm{\eta }_{\TT}^2}\EE_{t,x}\left[ \int_t^T\big( \wpsi_s(\eta)- \psi'(s,V^\omega_s) \big)^2\eta_s^2\ds\right]\nonumber\\
        &\le \norm{\eta }_{\TT}^2\sup_{s\in\TT}\frac{\EE_{t,x}[\sup_{\norm{\tau}_{\TT}\le1}\abs{\psi''(s,V^\omega_s+ \tau_s)}^2](T-t)}{4}  , \qquad \qquad
        \label{eq:Diff_proof_deruw_2}
    \end{align}
    which goes to zero as~$\norm{\eta }_{\TT}$ tends to zero. 
    By virtue of Cauchy-Schwarz inequality, { the finiteness of the second moment of $\widetilde{\phi}(\eta)$, as noticed in} the sentence below~\eqref{eq:Diff_proof_deruw_1}, and~\eqref{eq:Diff_proof_deruw_2}, we conclude that the first term~\eqref{eq:Diff_proof_deruw_one} tends to zero.

    For the second term, by definition of $\widetilde{\phi}$, the local Lipschitz continuity of~$\phi''$, Assumption~\ref{assu:phipsi} and~\eqref{eq:Bounds_XV}, we have
    \begin{align}\label{eq:diff_phi'}
        \wphi(\eta)-\phi'(X^\omega_T) 
        & = \int_0^1 \big( \phi'(X^\omega_T+ \lambda\delta_{\eta} X^\omega_T) - \phi'(X^\omega_T)\big) \D\lambda\\ \nonumber
        &\le \sup_{\alpha\in(0,1)}\abs{\phi''\left((1-\alpha)X^{\omega}_T+\alpha X^{\omega +\eta}_T\right)} \int_0^1 \lambda\abs{\delta_{\eta} X^\omega_T} \D\lambda \\ \nonumber
        &\lesssim \left(1 + \abs{X^{\omega}_T}^\kphi+ \abs{ X^{\omega +\eta}_T}^\kphi\right) \abs{\delta_{\eta} X^\omega_T}.
    \end{align}
    By~\eqref{eq:DeltaXomega}, BDG and H\"{o}lder's inequality, we have
    \begin{align}
        \EE_{t,x}\left[(\delta_{\eta} X^\omega_T)^4\right] 
        \lesssim \int_t^T \eta_s^4 \EE_{t,x}\left[\wpsi_s(\eta)^4\right]\ds,
    \end{align}
    while Assumption~\ref{assu:phipsi}(ii) yields
    \begin{align*}
\EE_{t,x}\left[\psi'(s,V^\omega_s+\lambda \eta_s)^4\right] 
& \lesssim 1 + \E^{4\kpsi T} \EE_{t,x}\left[\exp\left( 4\kpsi (V^\omega_s+\lambda \eta_s)\right)\right] \\
& \lesssim 1+\E^{4\kpsi\left( \omega_s+\lambda\eta_s\right)} \EE_{t,x}\left[\E^{4\kpsi I^t_s}\right],
    \end{align*}
    which is uniformly bounded for $s,t\in\TT$ since~$I_t^s$ is Gaussian and, as a consequence, so is $\EE_{t,x}[\wpsi_s(\eta)^4]$.  {Therefore $\EE\abs{\delta_\eta X^\omega_T}^p$ tends to zero as $\norm{\eta}_{\TT}\to0$.}
    Using Cauchy-Schwarz and $(1 + \abs{X^{\omega}_T}^\kphi+ \abs{ X^{\omega +\eta}_T}^\kphi )\in L^4$we conclude that 
    $$
\lim_{\norm{\eta}_{\TT}\downarrow0}\EE_{t,x}\left[\left( \wphi(\eta)-\phi'(X^\omega_T) \right)^2\right]=0.
    $$
    On the other hand, similar computations yield
    $$
    \sup_{t\in\TT}\EE_{t,x}\left[\left(\int_t^T \psi'(s,V^\omega_s) \eta_s \,\D B_s\right)^2\right] <\infty.
    $$
    A new application of Cauchy-Schwarz and the last two displays show that the second term~\eqref{eq:Diff_proof_deruw_two} tends to zero as well, concluding the proof of the first pathwise derivative formula. 
    The crossed derivative~$\partial_{\omega x} u$ is computed analogously since~$\phi''$ is also Lipschitz continuous.\\ 

    $\bullet$ For the second pathwise derivative, we again consider only the case $\zeta=0$, as the other follow with the same method.
    First of all, let us recall that we use the same notations as in the previous bullet point with $(t,x)$ fixed.
    For all~$\eta^{(1)},\eta^{(2)}\in\Ww_t,\omega\in\Ww$, let us define
    $$
    \Psi_t^i(\omega) := \int_t^T \psi'(s,V^{\omega}_s) \eta^{(i)}_s \D B_s,
    $$
    for $i=1,2$, 
    so that we are interested in the difference
    \begin{align}
       & \phi'(X^{\omega+\eta^{(2)}}_T) \Psi_t^1(\omega+\eta^{(2)}) -\phi'(X^{\omega}_T)\Psi_t^1(\omega)\\
       & \qquad\qquad\qquad- \left\{\phi''(X_T^\omega) \Psi_t^1(\omega)\Psi_t^2(\omega) + \phi'(X_T^\omega) \int_t^T \psi''(s,V_s^\omega) \eta^{(1)}_s \eta^{(2)}_s \,\D B_s \right\}\nonumber\\
       &\qquad \qquad =  \left\{ \phi'(X^{\omega+\eta^{(2)}}_T)  -\phi'(X^{\omega}_T) -\phi''(X_T^\omega) \Psi_t^2(\omega)\right\} \Psi_t^1(\omega) \label{eq:Diff_proof_deruww_1}\\
       & \qquad\qquad\qquad+ \phi'(X^{\omega+\eta^{(2)}}_T)\left\{\Psi_t^1(\omega+\eta^{(2)}) -\Psi_t^1(\omega) -\int_t^T \psi''(s,V_s^\omega) \eta^{(1)}_s \eta^{(2)}_s \,\D B_s\right\} \label{eq:Diff_proof_deruww_2}\\
       & \qquad\qquad\qquad+ \int_t^T \psi''(s,V_s^\omega) \eta^{(1)}_s \eta^{(2)}_s \,\D B_s \left\{ \phi'(X^{\omega+\eta^{(2)}}_T)-\phi'(X^{\omega}_T)\right\} .\label{eq:Diff_proof_deruww_3}
    \end{align}
    Once again, we need to take expectations,  {divide by~$\norm{\eta^{(2)}}_{\TT}$ and show that the limit as~$\norm{\eta^{(2)}}_{\TT}\downarrow0$ is zero. }
    The first term~\eqref{eq:Diff_proof_deruww_1} can be dealt with in the same way as~\eqref{eq:Diff_proof_deruw_two} with~$\phi',\psi'$ replaced by~$\phi'',\psi''$ and one more application of Cauchy-Schwarz to separate the term in brackets and~$\Psi_t^1(\omega)$.  The third term~\eqref{eq:Diff_proof_deruww_3} 
    is studied similarly to~\eqref{eq:Diff_proof_Dphi}  exploiting~\eqref{eq:DeltaXomega}. 
    We focus our attention on the second term~\eqref{eq:Diff_proof_deruww_2}. 
    As in~\eqref{eq:DeltaXomega}, we start by writing 
    $$
  \Psi^1_t(\omega+\eta^{(2)}) - \Psi^1_t(\omega) 
 = \int_t^T \big(\psi'(s,V^{\omega+\eta^{(2)}}_s)-\psi'(s,V^{\omega}_s)\big) \eta^{(1)}_s \,\D B_s
 = \int_t^T  \wpsi'_s(\eta^{(2)}) \eta^{(1)}_s\eta^{(2)}_s\,\D B_s,
    $$
    where $\wpsi'_s(\eta^{(2)}):= \int_0^1 \psi''(s,V^\omega_s+\lambda\eta^{(2)}_s)\D\lambda$. Therefore, 
    \begin{align}\label{eq:psi_est_incr}
        \Psi^1_t(\omega+\eta^{(2)}) - \Psi^1_t(\omega) & -\int_t^T \psi''(s,V_s^\omega) \eta^{(1)}_s \eta^{(2)}_s \D B_s \nonumber\\
        & = \int_t^T \left(  \wpsi'_s(\eta^{(2)})- \psi''(s,V^\omega_s)\right) \eta^{(1)}_s\eta^{(2)}_s \D B_s \nonumber\\
        & = \int_t^T \left( \int_0^1 \left(\psi''(s,V^\omega_s+\lambda\eta^{(2)}_s)- \psi''(s,V^\omega_s)\right)\D \lambda \right) \eta^{(1)}_s \eta^{(2)}_s \D B_s.
    \end{align}
    Now, as in~\eqref{eq:Diff_proof_deruw_2}, It\^o's isometry for~\eqref{eq:psi_est_incr} and the local Lipschitz continuity of~$\psi''$ yield
    \begin{align*}
        \lim_{\norm{\eta^{(2)}}_{\TT}\downarrow0} \frac{1}{\norm{\eta^{(2)}}_{\TT}} \EE_{t,x}\left[\left(  \Psi^1_t(\omega+\eta^{(2)}) - \Psi^1_t(\omega) -\int_t^T \psi''(s,V_s^\omega) \eta^{(1)}_s \eta^{(2)}_s \,\D B_s \right)^2\right] =0,
    \end{align*}
    which concludes the proof.


\subsection{Proof of Proposition~\ref{prop:SingularDerivative}}\label{sec:Proof_SingularDerivative}
We recall that the pathwise derivative with respect to a singular direction is the limit of smooth directions, as in~\eqref{eq:SpatialDerivatives}.
Hence, we introduce the smooth approximation~$K^{\delta,t}:=K(\cdot\vee(t+\delta),t)$ which belongs to~$\Ww_t$, such that Proposition~\ref{prop:SpatialDerivatives} holds with $\eta=K^{\delta,t}$ and then pass to the limit~$\delta\downarrow0$.
We show that the conclusions of Proposition~\ref{prop:SpatialDerivatives} still hold for the second pathwise derivative, and in the case~$\zeta=0$, as the first derivative and the other cases follow in a straightforward way using the same techniques.  \\

$\bullet$ Let us start with the first term of~\eqref{eq:SingularDerivative}. Note first that~$K^{\delta,t}-K^t=0$ on~$[t+\delta,T]$.  We use the identity $a^2-b^2=(a+b)(a-b)$ and apply Cauchy-Schwarz inequality 
\begin{align}
&\EE_{t,x,\omega}\left[\phi'(X_T) \left(\int_t^{t+\delta} \psi'(s,V_s) K(t+\delta,t)\,\D B_s\right)^2\right]\\
 &\qquad\qquad -\EE_{t,x,\omega}\left[\phi'(X_T) \left(\int_t^{t+\delta} \psi'(s,V_s) K(s,t)\,\D B_s\right)^2\right] \label{eq:DeltaToZero1} \\
    &\le \sqrt{2} \EE_{t,x,\omega}\left[\phi'(X_T)^2\left\{\left(\int_t^{t+\delta} \psi'(s,V_s) K(t+\delta,t)\,\D B_s\right)^2 + \left(\int_t^{t+\delta} \psi'(s,V_s) K(s,t)\,\D B_s\right)^2\right\}\right]^\half \label{eq:DeltaToZero2} \\
    &\qquad \cdot \EE_{t,x,\omega}\left[\left(\int_t^{t+\delta} \psi'(s,V_s) \big\{K(t+\delta,t)-K(s,t)\big\}\,\D B_s\right)^2\right]^\half.\label{eq:DeltaToZero3} 
\end{align}
We study the second term in~\eqref{eq:DeltaToZero2}.
For all~$p\ge2$, 
BDG and Holder's inequalities yield
\begin{align}
    &\EE_{t,x,\omega}\left[\abs{\int_t^{t+\delta} K(s,t) \psi'(s,V_s)\,\D B_s}^p\right] \\
    &\le \bdg_p \EE_{t,x,\omega}\left[\abs{\int_t^{t+\delta}K(s,t)^2 \psi'(s,V_s)^2\ds}^{\frac{p}{2}}\right] \nonumber \\
    &= \bdg_p\EE_{t,x,\omega}\left[\abs{\int_t^{t+\delta}\big(K(s,t)^{\frac{4}{p}} \psi'(s,V_s)^2\big) K(s,t)^\frac{2p-4}{p}\ds}^{\frac{p}{2}}\right] \nonumber  \\
    &\le \bdg_p\EE_{t,x,\omega}\left[ \abs{\int_t^{t+\delta}\big(K(s,t)^{\frac{4}{p}} \psi'(s,V_s)^2 \big)^{\frac{p}{2}} \dr} \cdot\abs{\int_t^{t+\delta}K(s,t)^{\frac{2p-4}{p} \frac{p}{p-2}} \ds }^{\frac{p}{2}\frac{p-2}{p}}\right] \nonumber \\
    &\le \bdg_p \int_t^{t+\delta}K(s,t)^2 \EE_{t,x,\omega}[\abs{\psi'(s,V_s)}^p] \ds \left( \int_t^{t+\delta}K(s,t)^2\ds \right)^{\frac{p-2}{2}}.
\label{eq:DeltaToZeroMoment}
\end{align}
Since $\EE_{t,x,\omega}[\abs{\psi'(s,V_s)}^p]= \EE[\abs{\psi'(s,V_s^{t,\omega})}^p] $ is uniformly bounded in $s$, the above is bounded by some constant times~$(\int_t^{t+\delta}K(s,t)^2\ds)^{p/2}\le C \delta^{pH}$.
A similar reasoning holds replacing~$K(s,t)$ with~$K(t+\delta,t)$, hence Cauchy-Schwarz ensures that~\eqref{eq:DeltaToZero2} is bounded. It\^o's isometry, together with Assumption~\ref{assu:phipsi} and~\eqref{eq:lemma_bound_G}, yields
\begin{align*}
&\EE_{t,x,\omega}\left[\left(\int_t^{t+\delta} \psi'(s,V_s) \big\{K(t+\delta,t)-K(s,t)\big\}\,\D B_s\right)^2\right] \\
    & = \EE\left[\int_t^{t+\delta} \psi'(s,V_s)^2\big\{K(t+\delta,t)-K(s,t)\big\}^2\ds\right]\\
    & \qquad \qquad \lesssim \sup_{s\in\TT} \EE[\psi'(s,V_s)^2] \int_t^{t+\delta} \big\{K(t+\delta,t)-K(s,t)\big\}^2\ds
    \lesssim \delta^{2H}.
\end{align*}

$\bullet$ The other term to deal with is, for~$\delta>0$,
\begin{align*}
& \EE_{t,x,\omega}\left[\phi'(X_T) \int_t^T \psi''(s,V_s) K(s\vee (t+\delta),t)^2 \,\D B_s\right]\\
& = \EE_{t,x,\omega}\left[ \int_t^T \langle\bm{\Df}_s \phi'(X_T),\bm{\rho}\rangle \psi''(s,V_s) K(s\vee (t+\delta),t)^2 \ds \right],
\end{align*}
using Malliavin integration by parts~\eqref{eq:IBP_B}. Otherwise, when taking~$\delta\downarrow 0$, we would obtain the ill-defined integral~$\int_t^T K(s,t)^2\,\D B_s$. Notice that, for~$r\le s$, $\Df_s^W V_r=K(s,r)$ and~$\Df_s^{\overline{W}} V_r=0$. 
Therefore, $\bm{\Df}_s \phi'(X_T) =  \phi''(X_T)\big( \Df_s^W X_T,  \Df^{\overline{W}}_s X_T\big)^\top$ where
\begin{align}\label{eq:DX}
\Df^{\overline{W}}_s X_T 
&=  \overline\rho \psi(s,V_s), \\ \nonumber
\Df_s^W X_T &= \rho \psi(s,V_s) + \int_s^T \Df_s \psi(r,V_r) \,\D B_r + \zeta \int_0^T \Df_s^W \psi(r,V_r)^2 \dr \\ \nonumber
&= \rho \psi(s,V_s) + \int_s^T K(r,s)\psi'(r,V_r) \,\D B_r + \zeta \int_s^T K(r,s) (\psi^2)'(r,V_r) \dr.
\end{align}
Since~$\psi,\psi^2,(\psi^2)'$ have subexponential growth, 
the following holds for all~$p>1$ by Cauchy--Schwarz inequality
\begin{align*}
    &\EE_{t,x,\omega}\left[|\psi(s,V_s)|^p\right] 
    + \EE_{t,x,\omega}\left[\abs{\int_s^T K(r,s)(\psi^2)'(r,V_r)\dr }^p\right] \\ 
    & \qquad \lesssim \EE_{t,x,\omega}\left[|\psi(s,V_s)|^p\right] 
    + \left(\int_s^T K(r,s)^2\dr \right)^{p/2}\EE_{t,x,\omega}\left[\abs{\int_s^T (\psi^2)'(r,V_r)^{2}\dr }^p\right] \\
    & \qquad \lesssim 1+ \E^{p\kpsi T} { \sup_{s\in\TT } }\EE_{t,x,\omega}\left[\E^{p\kpsi  V_s}\right]
    + \E^{p\kappa_\psi^2 T} \left(\int_s^T K(r,s)^2 \dr \right)^{p/2}  \EE_{t,x,\omega}\left[{ \sup_{s\in\TT } }\, \E^{p\kappa_\psi^2  V_s}\right]  .
\end{align*}
By virtue of~\eqref{eq:Bounds_XV} and~$K\in L^2(\TT)$, these terms are uniformly bounded in~$L^p$ so that the first and the last terms on the right-hand side of~\eqref{eq:DX} are in ~$L^p$.
The same technique as in~\eqref{eq:DeltaToZeroMoment} shows the same holds for the second term in the second equation in~\eqref{eq:DX}; this entails
\begin{equation}\label{eq:Bound_DX}
\sup_{s\in\TT} \EE\left[\abs{\langle\bm{\Df}_s\phi'(X_T), \bm{\rho}\rangle}^p\right]<\infty.
\end{equation}
Hence the proof follows from the estimate
\begin{align*}
    &\abs{\EE_{t,x,\omega}\left[ \int_t^T \langle\bm{\Df}_s \phi'(X_T),\bm{\rho}\rangle \psi''(s,V_s) \Big\{K(s\vee (t+\delta),t)^2-K(s,t)^2\big\} \ds \right]} \\
    &\le \sup_{s\in\TT}\EE_{t,x,\omega}\left[\langle\bm{\Df}_s \phi'(X_T),\bm{\rho}\rangle \psi''(s,V_s)\right] \abs{\int_t^{t+\delta} \Big\{K(t+\delta,t)^2-K(s,t)^2\big\} \ds} \lesssim \delta^{2H}.
\end{align*}


\subsection{Proof of Proposition~\ref{prop:Cplusalpha}}\label{sec:Proof_Cplusalpha}
As in Definition~\ref{def:Cplusalpha}, let~$(t,x,\omega)\in\overline{\Gamma}$, $\eta,\eta^{(1)},\eta^{(2)}\in\Ww_t$ with supports in~$[t,t+\delta]$ for some small~$\delta>0$. Essentially, the estimates in~$x,\omega$ are verified thanks to growth conditions on~$\phi,\psi$ and the bounds~\eqref{eq:Bounds_XV}. The presence of~$\eta$ in a Riemann integral (resp. stochastic integral) leads to a bound proportional to~$\delta$ (resp. $\sqrt{\delta}$). This justifies the estimates of Definition~\ref{def:Cplusalpha} with~$\alpha=\half$. Finally the continuity of the derivatives is a consequence of the regularity of~$\phi$ and~$\psi$. \\

\emph{(i)} 
Since~$\phi'$ and~$\phi''$ both have polynomial growth, it is clear from~\eqref{eq:Bounds_XV} that~$\partial_x u$ and~$\partial_{xx} u$ have $G$-growth as introduced in Definition~\ref{def:Cplusalpha}. 
Let us note these derivatives do not require the factor~$\delta^\alpha$ as they are not in the direction of a singular kernel. 
Turning to the pathwise derivative~$\langle \partial_\omega u,\eta\rangle$ given in Proposition~\ref{prop:SpatialDerivatives}, we apply It\^{o}'s isometry after H\"{o}lder's inequality and by the growth of~$\phi,\psi$ and~\eqref{eq:Bounds_XV}, there exist~$G_1,G_2\in\Xx$ such that 
\begin{align}\label{eq:est_first_derivative}
\EE\left[\phi'(X_T^{t,x,\omega}) \int_t^T \psi'(s,V_s^{t,\omega}) \eta_s\,\D B_s\right] 
&\le \EE\left[\phi'(X_T^{t,x,\omega})^2\right]^\half \EE\left[\left(\int_t^{t+\delta} \psi'(s,V_s^{t,\omega}) \eta_s\,\D B_s\right)^2\right]^\half \\
&\lesssim \norm{G_1(x,\omega)}_{\TT} \EE\left[\int_t^{t+\delta} \psi'(s,V_s^{t,\omega})^2 \eta_s^2\ds\right]^\half \nonumber\\
&\lesssim \norm{G_1(x,\omega)}_{\TT} \sup_{s\in\TT} \EE\left[ \psi'(s,V_s^{t,\omega})^2 \right]^\half \,\norm{\eta}_\TT \sqrt{\delta} \nonumber\\
&\lesssim \norm{G_1(x,\omega)}_{\TT} \norm{G_2(x,\omega)}_{\TT} \,\norm{\eta}_\TT \sqrt{\delta} \nonumber.
\end{align}
Here~$\norm{G(x,\omega)}_{\TT}=\sup_{t\in\TT} \abs{G(x,\omega_t)}$. 
Clearly, $G_1 G_2$  is also in~$\Xx$ by Cauchy-Schwarz inequality, and the second term of~$\langle \partial_\omega u,\eta\rangle$ can be dealt with in the same way by applying Jensen's inequality instead of It\^o's isometry. This shows the $G$-growth with~$\alpha=\half$. The case of~$\langle \partial_\omega(\partial_x u),\eta\rangle$ is identical.

For the second pathwise derivative, we first apply the Malliavin IBP as in the representation~\eqref{eq:SingularDerivative} and then~\eqref{eq:Bound_DX} entails
\begin{align}\label{eq:est_second_derivative_1}
    \EE\left[\phi'(X_T^{t,x,\omega})\int_t^T \psi''(s,V^{t,\omega}_s) \eta^{(1)}_s \eta^{(2)}_s \D B_s \right]
    &=\int_t^T \EE\left[\langle\bm{\Df}_s \phi'(X_T^{t,x,\omega}),\bm{\rho}\rangle \psi''(s,V^{t,\omega}_s)\right] \eta^{(1)}_s \eta^{(2)}_s \ds \\
    &\lesssim \int_t^T \eta^{(1)}_s \eta_s^{(2)}\ds
    \lesssim \norm{\eta^{(1)}}_{\TT}\norm{\eta^{(2)}}_{\TT} \delta_1\land\delta_2 
    \lesssim \norm{\eta^{(1)}}_{\TT}\norm{\eta^{(2)}}_{\TT} \sqrt{\delta_1}\sqrt{\delta_2} . \nonumber
\end{align}
Similar computations show that for some~$G\in\Xx$
\begin{align}\label{eq:est_second_derivative2}
\EE\left[\zeta \int_t^T (\psi_s^2)''(V_s^{t,\omega}) \eta^{(1)}_s \eta^{(2)}_s \,\D s \right] \lesssim \norm{G(x,\omega)}_{\TT}^2 \,\norm{\eta^{(1)}}_{\TT}\norm{\eta^{(2)}}_{\TT} \sqrt{\delta_1}\sqrt{\delta_2}.
\end{align}
 For the other terms, we apply H\"older's inequality with~$\frac{1}{p_1} + \frac{1}{p_2} + \frac{1}{p_3}=1$ to separate the three factors and there exist~$G\in\Xx$ such that
\begin{align*}
    & \EE\left[\phi''(X_T^{t,x,\omega}) \left(\int_t^T \psi'(s,V_s^{t,\omega}) \eta^{(1)}_s \,\D B_s \right)\left(\int_t^T \psi'(s,V_s^{t,\omega}) \eta^{(2)}_s \,\D B_s \right) \right] \\
    &\le \norm{\phi''(X_T^{t,x,\omega})}_{L^{p_1}} \norm{\int_t^T \psi'(s,V_s^{t,\omega}) \eta^{(1)}_s \,\D B_s}_{L^{p_2}} \norm{\int_t^T \psi'(s,V_s^{t,\omega}) \eta^{(2)}_s \,\D B_s}_{L^{p_3}}\\
    & \lesssim \norm{G(x,\omega)}_{\TT}\norm{\eta^{(1)}}_{\TT} \norm{\eta^{(2)}}_{\TT} \sqrt{\delta_1}\sqrt{\delta_2},
\end{align*}
where the last inequality follows from BDG inequality and taking the supremum of~$\eta^{(1)},\eta^{(2)}$.
The same technique of splitting into three different factors shows that the other terms also yield the same estimates. 
This concludes Definition~\ref{def:Cplusalpha}(i) with $\alpha=\half$.
\\

\emph{(ii)} For the regularity, we consider~$(t,x,\omega),(t,x',\omega')\in\overline{\Gamma}$. We focus on $\omega$-continuity as $x$-continuity is easier and follows with similar arguments. Hence we fix~$(t,x)$, abbreviate~$X^{t,x,\omega}$ by~$X^\omega$ and denote the~$L^2$ norm under~$\EE_{t,x}$ as~$\norm{\cdot}_{L^2_{t,x}}$:
\begin{align}\label{eq:est_first_derivative_reg}
\EE_{t,x} & \left[ \phi'(X_T^\omega) \int_t^T \psi'(s,V_s^\omega)\eta_s\,\D B_s\right]  - \EE_{t,x}\left[ \phi'(X_T^{\omega'}) \int_t^T \psi'(s,V_s^{\omega'})\eta_s \,\D B_s\right]\\
&= \EE_{t,x}\left[ \left(\phi'(X_T^\omega)-\phi'(X_T^{\omega'})\right) \int_t^T \psi'(s,V_s^\omega)\eta_s\,\D B_s\right] \nonumber\\
& \qquad  + \EE_{t,x}\left[ \phi'(X_T^{\omega'}) \left(\int_t^T \psi'(s,V_s^\omega)\eta_s\,\D B_s-\int_t^T \psi'(s,V_s^{\omega'})\eta_s\,\D B_s\right)\right] \nonumber\\
&\le \norm{\phi'(X_T^\omega)-\phi'(X_T^{\omega'})}_{L^2_{t,x}} \norm{\int_t^{t+\delta}\psi'(s,V_s^\omega)\eta_s\,\D B_s}_{L^2_{t,x}} \nonumber\\
& \qquad +\norm{\phi'(X_T^{\omega'})}_{L^2_{t,x}} \norm{\int_t^{t+\delta} \left(\psi'(s,V_s^\omega)- \psi'(s,V_s^{\omega'})\right)\eta_s\,\D B_s}_{L^2_{t,x}}. \nonumber
\end{align}
Previous computations showed there exists~$G\in\Xx$ such that
\begin{align*}
\norm{\phi'(X_T^{\omega'})}_{L^2_{t,x}}
& \lesssim \norm{G(x,\omega')}_{\TT},\\
\norm{\int_t^{t+\delta}\psi'(s,V_s^\omega)\eta_s\,\D B_s}_{L^2_{t,x}}
 & \lesssim \norm{G(x,\omega)}_{\TT} \norm{\eta}_{\TT} \sqrt{\delta}.
\end{align*}
Starting with a decomposition similar to the one in~\eqref{eq:DeltaXomega}, then applying It\^{o}'s isometry and the fact that $\eta$ is uniformly bounded over time together with Assumption~\ref{assu:phipsi}, we obtain
\begin{align*}
& \norm{\int_t^{t+\delta} \left(\psi'(s,V_s^\omega)- \psi'(s,V_s^{\omega'})\right)\eta_s\,\D B_s}_{L^2_{t,x}}\\
    & = \norm{\int_t^{t+\delta}\eta_s(V_s^\omega-V_s^{\omega'})\int_0^1 \psi''(s,\lambda V_s^\omega+ (1-\lambda)V_s^{\omega'})\D \lambda \,\D B_s}_{L^2_{t,x}}\\
    & = \EE_{t,x} \left[ \int_t^{t+\delta}\eta_s^2(\omega_s-\omega'_s)^2\left(\int_0^1 \psi''(s,\lambda V_s^\omega+ (1-\lambda)V_s^{\omega'}) \D \lambda\right)^2 \,\D s\right]^{\half}\\
    & \lesssim \norm{\eta}_{\TT}\norm{\omega-\omega'}_\TT \EE \left[ \int_t^{t+\delta}\int_0^1 (1+\E^{\kappa_\psi(\lambda V_s^\omega+ (1-\lambda)V_s^{\omega'})})^2 \D \lambda  \,\D s\right]^{\half}\\
    & \lesssim \norm{\eta}_{\TT}\norm{\omega-\omega'}_\TT \left( 1+\int_t^{t+\delta}\EE \left[ \E^{2\kappa_\psi I^t_s}\right]\int_0^1 \E^{2\kappa_\psi(\lambda \omega_s+ (1-\lambda)\omega'_s)} \D \lambda  \,\D s\right)^{\half}\\
    & \lesssim \norm{\eta}_{\TT}\norm{\omega-\omega'}_\TT \left( 1+\int_t^{t+\delta}(\E^{2\kappa_\psi \omega_s}+ \E^{2\kappa_\psi\omega'_s}) \D s\right)^{\half}\\
    &\lesssim (\norm{G(x,\omega)}_{\TT} +\norm{G(x,\omega')}_{\TT})\norm{\omega-\omega'}_{\TT} \norm{\eta}_{\TT}\sqrt{\delta},
\end{align*}
where we used the convexity of the exponential.
Following the idea of~\eqref{eq:Diff_proof_Dphi}, we also know by local Lipschitz continuity and Cauchy-Schwarz inequality
\begin{align*}
\norm{\phi'(X_T^\omega)-\phi'(X_T^{\omega'})}_{L^2_{t,x}} 
&\le \norm{X^\omega_T-X^{\omega'}_T}_{L^4_{t,x}} \norm{ \int_0^1 \phi''(\lambda X^\omega_T +(1-\lambda) X^{\omega'}_T)\D \lambda}_{L^4_{t,x}}\\
&\lesssim (\norm{G(x,\omega)}_{\TT} +\norm{G(x,\omega')}_{\TT})\norm{\omega-\omega'}_\TT .    
\end{align*}
where, to get the bound for the integral, we have exploited the fact that, by Assumption~\ref{assu:phipsi}, monotonicity and an application of Lemma~\ref{lemma:Bounds_XV}, we have
\begin{align}
    \norm{ \int_0^1 \phi''(\lambda X^\omega_T +(1-\lambda) X^{\omega'}_T)\D \lambda}_{L^4_{t,x}}
    & \lesssim \norm{ \int_0^1 (1+ (\lambda X^\omega_T+(1-\lambda) X^{\omega'}_T)^{\kappa_\phi})\D \lambda}_{L^4_{t,x}}\\
    & \lesssim 1+  \norm{ \int_0^1  (\lambda X^\omega_T)^{\kappa_\phi} +((1-\lambda) X^{\omega'}_T)^{\kappa_\phi}\D \lambda}_{L^4_{t,x}}\\
    & \lesssim 1+  \norm{ ( X^\omega_T)^{\kappa_\phi} }_{L^4_{t,x}}+ \norm{(X^{\omega'}_T)^{\kappa_\phi}}_{L^4_{t,x}}    \\
    & \lesssim 1+ ( |x|^{4\kappa_\phi} + \E^{2\kappa_\psi4\kappa_\psi\norm{\omega}_T})^{\frac{1}{4}} +( |x|^{4\kappa_\phi} + \E^{2\kappa_\psi4\kappa_\psi\norm{\omega'}_T})^{\frac{1}{4}}
    \\
    & \lesssim 1+ |x|^{4\kappa_\phi} + \E^{2\kappa_\psi4\kappa_\psi\norm{\omega}_T} +  \E^{2\kappa_\psi4\kappa_\psi\norm{\omega'}_T},
\end{align}
which corresponds to~$\norm{G(x,\omega)}_{\TT}+\norm{G(x,\omega')}_{\TT}$ in the bound above.
Once again, the other term follows from the same steps.
\\

Let us move on to the second derivative with representation~\eqref{eq:SingularDerivative}.
The first term in the representation is handled in the following way. 
Let us notice once more that without loss of generality we restrict our study to the case $\zeta=0$, as the additional terms are studied with similar computations.
\begin{align*}
    &\EE_{t,x}\left[ \int_t^T \langle\bm{\Df}_s \phi'(X_T^{\omega}),\bm{\rho}\rangle \psi''(s,V_s^{\omega}) \eta^{(1)}_s\eta^{(2)}_s \ds \right] 
    - \EE_{t,x}\left[ \int_t^T \langle\bm{\Df}_s \phi'(X_T^{\omega'}),\bm{\rho}\rangle \psi''(s,V_s^{\omega'}) \eta^{(1)}_s\eta^{(2)}_s \ds \right] \\
    &=\EE_{t,x}\Bigg[ \int_t^{t+(\delta_1\land\delta_2)} \Bigg\{ \Big(\langle\bm{\Df}_s \phi'(X_T^{\omega}),\bm{\rho}\rangle -\langle\bm{\Df}_s \phi'(X_T^{\omega'}),\bm{\rho}\rangle \Big) \psi''(s,V_s^{\omega}) \\
    & \qquad \qquad + \langle\bm{\Df}_s\phi'(X_T^{\omega'}),\bm{\rho}\rangle \left(\psi''(s,V_s^{\omega})-\psi''(s,V_s^{\omega'})\right) \Bigg\} \eta^{(1)}_s\eta^{(2)}_s \ds \Bigg] \\
    &\le \int_t^{t+(\delta_1\land\delta_2)} \norm{\langle\bm{\Df}_s \phi'(X_T^{\omega}),\bm{\rho}\rangle -\langle\bm{\Df}_s \phi'(X_T^{\omega'}),\bm{\rho}\rangle}_{L^2_{t,x}} \norm{\psi''(s,V_s^{\omega})}_{L^2_{t,x}} \ds \norm{\eta^{(1)}}_{\TT} \norm{\eta^{(2)}}_{\TT}\\
    & \quad +\int_t^{t+(\delta_1\land\delta_2)}
    \norm{\langle\bm{\Df}_s \phi'(X_T^{\omega'}),\bm{\rho}\rangle}_{L^2_{t,x}} \norm{\psi''(s,V_s^{\omega})-\psi''(s,V_s^{\omega'})}_{L^2_{t,x}} \ds \norm{\eta^{(1)}}_{\TT} \norm{\eta^{(2)}}_{\TT}.
\end{align*}
Concerning the second term on the right-hand side, we have already seen that~$ \langle\bm{\Df}_s \phi'(X_T^{\omega'}),\bm{\rho}\rangle$ is bounded in~$L^2$ and, since~$\psi''$ is $\Cc^1$, $\norm{\psi''(s,V_s^{\omega})-\psi''(s,V_s^{\omega'})}_{L^2}\lesssim \norm{G(x,\omega)}_{\TT} \norm{\omega-\omega'}_\TT$. 
On the other hand, for the first term, we see from~\eqref{eq:DX} that
\begin{align}\nonumber
    & \norm{\langle\bm{\Df}_s \phi'(X_T^{\omega}),\bm{\rho}\rangle -\langle\bm{\Df}_s \phi'(X_T^{\omega'}),\bm{\rho}\rangle}_{L^2_{t,x}} \\
    &= \norm{\rho \phi''(X^\omega_T) \Df_s X^\omega_T +\bar{\rho}  \phi''(X^\omega_T) \overline{\Df}_s X^\omega_T - \rho \phi''(X^{\omega'}_T) \Df_s X^{\omega'}_T +\bar{\rho}  \phi''(X^{\omega'}_T) \overline{\Df}_s X^{\omega'}_T}_{L^2_{t,x}}
\end{align}
and, focusing on the first term, Cauchy-Schwarz inequality yields
\begin{align} 
    &\norm{ \phi''(X^\omega_T) \Df_s X^\omega_T  -  \phi''(X^{\omega'}_T) \Df_s X^{\omega'}_T}_{L^2_{t,x}}\\
    &\le \norm{\big(\phi''(X^\omega_T)-\phi''(X^{\omega'}_T)\big) \Df_s X^\omega_T }_{L^2_{t,x}}
    + \norm{ \phi''(X^{\omega'}_T) \big( \Df_s X^{\omega}_T-\Df_s X^{\omega'}_T\big)}_{L^2_{t,x}}\\
    &\lesssim \norm{\phi''(X^\omega_T)-\phi''(X^{\omega'}_T)}_{L^4_{t,x}}\norm{ \Df_s X^\omega_T }_{L^4_{t,x}}
    + \norm{\phi''(X^{\omega'}_T)}_{L^4_{t,x}}\bigg[ \norm{\psi(s,V_s^\omega) - \psi(s,V_s^{\omega'})}_{L^4_{t,x}} \\
    & + \norm{\int_s^T K(r,s) \big( \psi'_r(V_r^\omega)-\psi'_r(V_r^{\omega'})\big) \D B_r }_{L^4_{t,x}} + \norm{\int_s^T K(r,s)\big((\psi^2_r)'(V_r^{\omega})-(\psi^2_r)'(V_r^{\omega'})\big) \dr }_{L^4_{t,x}}\bigg].
\label{eq:Time_Diff_MalliavinTerm}
\end{align}
The regularity of~$\psi,\psi',\psi''$ allow us to conclude, using similar computations as before, that this term is smaller than~$\norm{G(x,\omega)}_{\TT} \norm{\omega-\omega'}_\TT \norm{\eta^{(1)}}_{\TT} \norm{\eta^{(2)}}_{\TT} (\delta_1\land\delta_2)$.

For the next term, we separate this triple in the following way
\begin{align}\label{eq:est_second_derivative_reg}
&\EE_{t,x}\left[\phi''(X_T^{\omega}) \left(\int_t^T \psi'(s,V_s^{\omega}) \eta^{(1)}_s \,\D B_s \right)\left(\int_t^T \psi'(s,V_s^{\omega}) \eta^{(2)}_s \,\D B_s \right) \right] \\
&-\EE_{t,x}\left[\phi''(X_T^{\omega'}) \left(\int_t^T \psi'(s,V_s^{\omega'}) \eta^{(1)}_s \,\D B_s \right)\left(\int_t^T \psi'(s,V_s^{\omega'}) \eta^{(2)}_s \,\D B_s \right) \right] \nonumber\\
& = \EE_{t,x}\left[\left(\phi''(X_T^{\omega}) - \phi''(X_T^{\omega'}) \right) \left(\int_t^T \psi'(s,V_s^{\omega}) \eta^{(1)}_s \,\D B_s \right)\left(\int_t^T \psi'(s,V_s^{\omega}) \eta^{(2)}_s \,\D B_s \right) \right] \nonumber\\
&\quad + \EE_{t,x}\left[\phi''(X_T^{\omega'}) \left(\int_t^T \left(\psi'(s,V_s^{\omega})-\psi'(s,V_s^{\omega'})\right) \eta^{(1)}_s \,\D B_s \right)\left(\int_t^T \psi'(s,V_s^{\omega}) \eta^{(2)}_s \,\D B_s \right) \right] \nonumber\\
&\quad+\EE_{t,x}\left[\phi''(X_T^{\omega'}) \left(\int_t^T \psi'(s,V_s^{\omega'}) \eta^{(1)}_s \,\D B_s \right)\left(\int_t^T \left(\psi'(s,V_s^{\omega})-\psi'(s,V_s^{\omega'})\right) \eta^{(2)}_s \,\D B_s \right) \right],\nonumber
\end{align}
and then applying H\"older to each of them, the same arguments as in point~(i) show that is bounded by~$\norm{G(x,\omega)}_{\TT} \norm{\omega-\omega'}_{\TT}\norm{\eta^{(1)}}_{\TT} \norm{\eta^{(2)}}_{\TT} \sqrt{\delta_1}\sqrt{\delta_2}$. 

\bigskip
\emph{(iii)}
We are left checking the last condition, continuity of the pathwise derivatives on~$\overline{\Gamma}$. 
We just showed uniform continuity in~$\omega$,
thus time continuity is the only step left.
We only show it for the second derivative, using~\eqref{eq:SingularDerivative}. We fix $x,\omega$; as in point~(ii) of this proof,
\begin{align}
&\EE_{x,\omega}\left[\phi''(X_T^{t+\delta}) \left(\int_{t+\delta}^T \psi'(s,V_s^{t+\delta}) \eta^{(1)}_s \,\D B_s \right)\left(\int_{t+\delta}^T \psi'(s,V_s^{t+\delta}) \eta^{(2)}_s \,\D B_s \right) \right] \nonumber \\
&-\EE_{x,\omega}\left[\phi''(X_T^{t}) \left(\int_t^T \psi'(s,V_s^{t}) \eta^{(1)}_s \,\D B_s \right)\left(\int_t^T \psi'(s,V_s^{t}) \eta^{(2)}_s \,\D B_s \right) \right] \nonumber\\
&= \EE_{x,\omega}\left[ \left(\phi''(X_T^{t+\delta}) - \phi''(X_T^{t}) \right)\left(\int_{t+\delta}^T \psi'(s,V_s^{t+\delta}) \eta^{(1)}_s \,\D B_s \right)\left(\int_{t+\delta}^T \psi'(s,V_s^{t+\delta}) \eta^{(2)}_s \,\D B_s \right) \right] \qquad \label{eq:Time_Diff_Decomposition_1}\\
& + \EE_{x,\omega}\left[\phi''(X_T^{t}) \left(\int_{t+\delta}^T \psi'(s,V_s^{t+\delta}) \eta^{(1)}_s \,\D B_s 
-\int_{t}^T \psi'(s,V_s^{t}) \eta^{(1)}_s \,\D B_s\right)\left(\int_{t+\delta}^T \psi'(s,V_s^{t+\delta}) \eta^{(2)}_s \,\D B_s \right) \right] \qquad\qquad\label{eq:Time_Diff_Decomposition_2}\\
&  + \EE_{x,\omega}\left[\phi''(X_T^{t}) \left(\int_t^T \psi'(s,V_s^{t}) \eta^{(1)}_s \,\D B_s \right) \left(\int_{t+\delta}^T \psi'(s,V_s^{t+\delta}) \eta^{(2)}_s \,\D B_s 
-\int_{t}^T \psi'(s,V_s^{t}) \eta^{(2)}_s \,\D B_s\right)\right].\qquad\qquad
\label{eq:Time_Diff_Decomposition_3}
\end{align}
For the first line~\eqref{eq:Time_Diff_Decomposition_1}, we apply Cauchy-Schwarz inequality to focus on the difference
\begin{align*}
\norm{\phi''(X_T^{t+\delta}) - \phi''(X_T^{t})}_{L^2_{x,\omega}} 
&\le \norm{X_T^{t+\delta}-X_T^t}_{L^4_{x,\omega}} \norm{\int_0^1 \phi'''(\lambda X_T^{t+\delta} + (1-\lambda) X_T^{t})\,\D \lambda}_{L^4_{x,\omega}}.
\end{align*}
Furthermore, by BDG and Jensen inequalities
\begin{align*}
\norm{X_T^{t+\delta}-X_T^t}_{L^4_{x,\omega}}^4 
&= \EE_{x,\omega}\left[\left(\int_{t+\delta}^T \big(\psi(s,V_s^{t+\delta})-\psi(s,V_s^t)\big)\,\D B_s + 
\int_t^{t+\delta} \psi(s,V_s^t)\,\D B_s\right)^4\right] \\
&\le 8 \bdg_4 T \int_t^T \EE_{x,\omega}\left[\abs{\psi(s,V_s^{t+\delta})-\psi(s,V_s^t)}^4\right]\ds +
8 \delta \bdg_4  \int_t^{t+\delta} \EE_{x,\omega}\left[\psi(s,V_s)^4  \right]\ds,
\end{align*}
where the second term clearly goes to zero as~$\delta$ tends to zero. 
For the first one, we write
\begin{equation}\label{eq:Time_Diff_psi}
\EE_{x,\omega}\left[\abs{\psi(s,V_s^{t+\delta})-\psi(s,V_s^t)}^4\right] 
= \EE_{x,\omega}\left[\abs{\left(V_s^{t+\delta}-V_s^t\right) \int_0^1 \psi'\left(s,V^t_s+\lambda\left(V_s^{t+\delta}-V_s^t\right)\right)\D \lambda }^4\right],
\end{equation}
where the integral is bounded in~$L^8$ as in~\eqref{eq:Bounds_XV} and
$$
\EE_{x,\omega}\left[\abs{V_s^{t+\delta}-V_s^t}^8\right]=\EE_{x,\omega}\left[\abs{\int_t^{t+\delta} K(s,r)\,\D W_r}^8\right] 
\le \bdg_8 \left(\int_t^{t+\delta} K(s,r)^2 \dr\right)^4 
\le \bdg_8 \delta^{8H}. 
$$
Yet another application of Cauchy-Schwarz yields that~\eqref{eq:Time_Diff_Decomposition_1} goes to zero with~$\delta$.

The second and third lines~\eqref{eq:Time_Diff_Decomposition_2} and~\eqref{eq:Time_Diff_Decomposition_3} are identical hence we only deal with the second one.
We decompose further 
\begin{align*}
& \int_{t+\delta}^T \psi'(s,V_s^{t+\delta}) \eta^{(1)}_s \,\D B_s 
-\int_{t}^T \psi'(s,V_s^{t}) \eta^{(1)}_s \,\D B_s\\
 & = \int_{t+\delta}^T \big(\psi'(s,V_s^{t+\delta})-\psi'(s,V_s^{t})\big) \eta^{(1)}_s \,\D B_s 
+\int_{t}^{t+\delta}  \psi'(s,V_s^{t}) \eta^{(1)}_s \,\D B_s
\end{align*}
The integral between $t$ and~$t+\delta$ was already studied point~(ii) and we showed these tend to zero in~$L^p$ norm as~$\delta\downarrow0$. Since~$\psi'$ is also $\Cc^1$, the same computations as above show that~$\EE_{x,\omega}[\abs{\psi'(s,V_s^{t+\delta})-\psi'(s,V_s^t)}^4]$ goes to zero, which concludes the time continuity of the first term of the representation in~\eqref{eq:SingularDerivative}.

Regarding the next term,
\begin{align*}
    &\EE_{x,\omega}\left[ \int_{t+\delta}^T \langle\bm{\Df}_s \phi'(X_T^{t+\delta}),\bm{\rho}\rangle \psi''(s,V_s^{t+\delta}) K(s,t+\delta)^2 \ds
    - \int_t^T \langle\bm{\Df}_s \phi'(X_T^t),\bm{\rho}\rangle \psi''(s,V_s^t) K(s,t)^2 \ds\right] \\
    &= \EE_{x,\omega}\bigg[ \int_{t+\delta}^T \Big\{\langle\bm{\Df}_s \phi'(X_T^{t+\delta}),\bm{\rho}\rangle \psi''(s,V_s^{t+\delta}) K(s,t+\delta)^2 - \langle\bm{\Df}_s \phi'(X_T^t),\bm{\rho}\rangle \psi''(s,V_s^t) K(s,t)^2 \Big\} \ds \\
    &\qquad + \int_t^{t+\delta} \langle\bm{\Df}_s \phi'(X_T^t),\bm{\rho}\rangle \psi''(s,V_s^t) K(s,t)^2 \ds \bigg] \\
    &=\EE_{x,\omega}\left[\int_{t+\delta}^T \Big\{\langle\bm{\Df}_s \phi'(X_T^{t+\delta}),\bm{\rho}\rangle - \langle\bm{\Df}_s \phi'(X_T^{t}),\bm{\rho}\rangle\Big\} \psi''(s,V_s^{t+\delta}) K(s,t+\delta)^2 \ds \right]\\
    &\qquad +\EE_{x,\omega}\left[\int_{t+\delta}^T \langle\bm{\Df}_s \phi'(X_T^{t}),\bm{\rho}\rangle \Big\{\psi''(s,V_s^{t+\delta})-\psi''(s,V_s^{t})\Big\} K(s,t+\delta)^2 \ds\right] \\
    &\qquad +\EE_{x,\omega}\left[\int_{t+\delta}^T \langle\bm{\Df}_s \phi'(X_T^{t}),\bm{\rho}\rangle \psi''(s,V_s^{t})\Big\{ K(s,t+\delta)^2-K(s,t)^2\Big\} \ds\right] \\
    &\qquad + \EE_{x,\omega}\left[\int_t^{t+\delta} \langle\bm{\Df}_s \phi'(X_T^t),\bm{\rho}\rangle \psi''(s,V_s^t) K(s,t)^2 \ds \right].
\end{align*}
Using the expression~\eqref{eq:DX} and similar techniques as before, the first term boils down to evaluating differences of the type~\eqref{eq:Time_Diff_psi}, albeit with~$\psi'$ instead of~$\psi$, which we proved tends to zero as~$\delta\downarrow0$. The second term follows identically with~$\psi''$ this time. For the third one we apply Cauchy-Schwarz inequality to isolate the term
$$
\abs{\int_{t+\delta}^T \Big( K(s,t+\delta)^2-K(s,t)^2\Big)\ds } \lesssim \delta^{2H}.
$$
Finally, the last term is smaller in absolute value than
$$
\sup_{s\in[t,t+\delta]} \EE_{x,\omega}\left[\langle\bm{\Df}_s \phi'(X_T^t),\bm{\rho}\rangle \psi''(s,V_s^t)\right] \int_t^{t+\delta} K(s,t)^2\ds \lesssim \delta^{2H}.
$$
The other terms of~\eqref{eq:SingularDerivative} can be shown to be continuous using the same techniques. This concludes Definition~\ref{def:Cplusalpha}(iii) and hence the proof of the proposition.

\subsection{Proof of Lemma~\ref{lemma:u_regularity}}\label{app:u_reg}

In order to prove this lemma we have to check that each of the conditions in Assumption~\ref{assumption:regcoef+U} are satisfied for the rough Bergomi model. 

(i) 
Regarding the regularity of $b$ and $\sigma$, for $x=(x_1,x_2)$ and $y=(y_1,y_2)$, that of~$\psi$ is guaranteeing the existence of a function~$G$ and a suitable modulus of continuity so that
\begin{align*}
    \abs{b(s,r,x)-b(s,r,y)}+\abs{\sigma(s,r,x)-\sigma(s,r,y)} 
    & = 2|\zeta|\abs{\psi(r,x_2)^2-\psi(r,y_2)^2}+\abs{\psi(r,x_2)-\psi(r,y_2)}\\
    &\le \varrho(\abs{x-y}) G(x\vee y).
\end{align*}

(ii) a. About the regularity of the derivatives of $U$, exploiting the explicit expression for the derivatives of $u$ as in Proposition \ref{prop:SpatialDerivatives}, we have that, for any $t\in\TT$, $\ww,\ww'\in \overline{\Ww}$ such that~$\ww,\ww',\eta,\eta'\in\Ww_t$, it holds
\begin{align*}
\abs{\big\langle\partial_{\ww} U(t,\ww),\eta\big\rangle} 
&= \abs{ \big\langle\partial_{\omega_1} U(t,\ww),\eta^{(1)}\big\rangle + \big\langle\partial_{\omega_2} U(t,\ww),\eta^{(2)}\big\rangle}\\
&= \abs{ \eta^{(1)}\partial_{x} u(t,x,\omega) + \big\langle\partial_{\omega} u(t,x,\omega),\eta^{(2)}\big\rangle}\\
    &= |\eta^{(1)}|\abs{\EE_{t,x,\omega}[\phi'(X_T)]} + \abs{\EE_{t,x,\omega}\left[\phi'(X_T) \int_t^T \psi'(s,V_s) \eta^{(2)}_s \D B_s \right]}\\
    &\qquad\qquad\qquad\qquad\qquad+\zeta \abs{\EE_{t,x,\omega}\left[\phi'(X_T)\int_t^T (\psi^2)'(s,V_s) \eta^{(2)}_s \D s \right]}
    \\
    &
    \le \norm{G(\ww)}_{\TT} \norm{\eta}_{L^2(\TT)}.
\end{align*}
In the last line we have exploited the fact that, since~$\phi'$ has polynomial growth, it is clear from~\eqref{eq:Bounds_XV} that~$\partial_x u$ 
has $G$-growth as introduced in Definition~\ref{def:Cplusalpha}. For the first term of~$\langle \partial_\omega u,\eta^{(2)}\rangle$ (the second term can be dealt with in the same way by applying Jensen's inequality instead of It\^o's isometry) we have exploited a suitable modification of~\eqref{eq:est_first_derivative} that provides the following bound in terms of $\norm{\eta^{(2)}}_{L^2(\TT)}$
\begin{align*}
    \EE\left[\phi'(X_T^{t,x,\omega}) \int_t^T \psi'(s,V_s^{t,\omega}) \eta^{(2)}_s\,\D B_s\right] 
    &\le \EE\left[\phi'(X_T^{t,x,\omega})^2\right]^\half \EE\left[\left(\int_t^{T} \psi'(s,V_s^{t,\omega}) \eta^{(2)}_s\,\D B_s\right)^2\right]^\half \\
    &\lesssim \norm{G_1(x,\omega)}_{\TT} \EE\left[\int_t^{T} \psi'(s,V_s^{t,\omega})^2 (\eta^{(2)}_s)^2\ds\right]^\half \nonumber\\
    &\lesssim \norm{G_1(x,\omega)}_{\TT}  \EE\left[ \sup_{s\in\TT}\psi'(s,V_s^{t,\omega})^2 \right]^\half \,\norm{\eta^{(2)}}_{L^2(\TT)}  \nonumber\\
    &\lesssim \norm{G_1(x,\omega)}_{\TT} \norm{G_2(x,\omega)}_{\TT} \,\norm{\eta^{(2)}}_{L^2(\TT)} \nonumber.
\end{align*}

(ii) b. For what concerns the estimate of 
\begin{align}
    \abs{\big\langle\partial_{\ww\ww} U(t,\ww),(\eta,\widetilde{\eta})\big\rangle}
    & = | \eta^{(1)}\widetilde{\eta}^{(1)}\partial_{xx} u(t,x,\omega)+ \widetilde{\eta}^{(1)}\big\langle\partial_{x\omega} u(t,x,\omega),\eta^{(2)}\big\rangle \\  
    & + {\eta}^{(1)}\big\langle\partial_{x\omega} u(t,x,\omega),\widetilde{\eta}^{(2)}\big\rangle 
    + \big\langle\partial_{\omega\omega} u(t,x,\omega),(\eta^{(2)}, \widetilde{\eta}^{(2)})\big\rangle|
\end{align}
we notice that, thanks again to the explicit representation of the derivative of $u$, this follows from estimating terms that are analogous to the ones performed above and some additional ones that we can deal with as in equations \eqref{eq:est_second_derivative_1} and  \eqref{eq:est_second_derivative2} by applying Hölder inequality instead of taking the supremum to obtain an estimate in terms of the $L^2(\TT)$-norm of $\eta$.

Finally, for the other terms, we apply H\"older's inequality with~$\frac{1}{p_1} + \frac{1}{p_2} + \frac{1}{p_3}=1$ to separate the three factors and then BDG inequality so that there exist~$G\in\Xx$ such that 
\begin{align*}
    & \EE\left[\phi''(X_T^{t,x,\omega}) \left(\int_t^T \psi'(s,V_s^{t,\omega}) \eta^{(2)}_s \,\D B_s \right)\left(\int_t^T \psi'(s,V_s^{t,\omega}) \widetilde{\eta}^{(2)}_s \,\D B_s \right) \right] \\
    &\le \norm{\phi''(X_T^{t,x,\omega})}_{L^{p_1}} \norm{\int_t^T \psi'(s,V_s^{t,\omega}) \eta^{(2)}_s \,\D B_s}_{L^{p_2}} \norm{\int_t^T \psi'(s,V_s^{t,\omega}) \widetilde{\eta}^{(2)}_s \,\D B_s}_{L^{p_3}}\\
    & \lesssim \norm{G(x,\omega)}_{\TT}\norm{\eta^{(2)}}_{L^2(\TT)} \norm{\widetilde{\eta}^{(2)}}_{L^2(\TT)}.
\end{align*} 
The same technique of splitting into three different factors shows that the other terms also yield the same estimates. 

(ii) c. For what concerns $\abs{\big\langle\partial_{\ww} U(t,\ww),\eta\big\rangle - \big\langle\partial_{\ww} U(t,\ww'),\eta\big\rangle}$, as in the proof of Proposition~\ref{prop:Cplusalpha}(ii), we focus on $\omega$-continuity as $x$-continuity is easier and follows with similar arguments. 
In particular, we refer to the computations in~\eqref{eq:est_first_derivative_reg} and what follows where exploiting Hölder inequality we can replace the supremum norm for $\eta$ with a $L^2(\TT)$-norm.

(ii) d. An analogous conclusion can be obtained adapting the results in~\eqref{eq:est_second_derivative_reg} and what follows for the estimate of $\abs{\big\langle\partial_{\ww\ww} U(t,\ww),(\eta,\eta)\big\rangle - \big\langle\partial_{\ww\ww} U(t,\ww'),(\eta,\eta)\big\rangle} $.


\bibliographystyle{siam}
\bibliography{references}

@article{nualart2023error,
  title={Error distribution of the {E}uler approximation scheme for stochastic {V}olterra equations},
  author={Nualart, David and Saikia, Bhargobjyoti},
  journal={Journal of Theoretical Probability},
  volume={36},
  number={3},
  pages={1829-1876},
  year={2023},
  publisher={Springer}
}

@article{fukasawa2023limit,
  title={Limit distributions for the discretization error of stochastic {V}olterra equations with fractional kernel},
  author={Fukasawa, Masaaki and Ugai, Takuto},
  journal={The Annals of Applied Probability},
  volume={33},
  number={6B},
  pages={5071--5110},
  year={2023},
  publisher={Institute of Mathematical Statistics}
}

@article{leon2023euler,
  title={Euler scheme for {SDEs} driven by fractional {B}rownian motions: integrability and convergence in law},
  author={Le{\'o}n, Jorge and Liu, Yanghui and Tindel, Samy},
  journal={arXiv:2307.06759},
  year={2023}
}

@article{pannier2024path,
  title={A path-dependent {PDE} solver based on signature kernels},
  author={Pannier, Alexandre and Salvi, Cristopher},
  journal={arXiv:2403.11738},
  year={2024}
}

@article{dupret2022impact,
  title={Impact of rough stochastic volatility models on long-term life insurance pricing},
  author={Dupret, Jean-Loup and Barbarin, J{\'e}r{\^o}me and Hainaut, Donatien},
  journal={European Actuarial Journal},
  pages={1--41},
  year={2022},
  publisher={Springer}
}

@article{abi2019markovian,
  title={Markovian structure of the {V}olterra {H}eston model},
  author={Abi Jaber, Eduardo and El Euch, Omar},
  journal={Statistics \& Probability Letters},
  volume={149},
  pages={63--72},
  year={2019},
  publisher={Elsevier}
}

@article{euch2018perfect,
  title={Perfect hedging in rough {H}eston models},
  author={Euch, Omar El and Rosenbaum, Mathieu},
  journal={The Annals of Applied Probability},
  volume={28},
  number={6},
  pages={3813--3856},
  year={2018},
  publisher={JSTOR}
}

@article{bonesini2021functional,
  title={Functional quantization of rough volatility and applications to volatility derivatives},
  author={Bonesini, Ofelia and Callegaro, Giorgia and Jacquier, Antoine},
  journal={Quantitative Finance},
  year={2023},
volume = {23},
number = {12},
pages = {1769-1792}
}

@article{cuchiero2020generalized,
  title={Generalized {F}eller processes and {M}arkovian lifts of stochastic {V}olterra processes: the affine case},
  author={Cuchiero, Christa and Teichmann, Josef},
  journal={Journal of Evolution Equations},
  volume={20},
  number={4},
  pages={1301--1348},
  year={2020},
  publisher={Springer}
}

@article{ren2014overview,
  title={An overview of viscosity solutions of path-dependent {PDE}s},
  author={Ren, Zhenjie and Touzi, Nizar and Zhang, Jianfeng},
  journal={Stochastic Analysis and Applications 2014: In Honour of Terry Lyons},
  pages={397--453},
  year={2014},
  publisher={Springer}
}

@article{jacquier2019deep,
  title={Deep curve-dependent {PDEs} for affine rough volatility},
  author={Jacquier, Antoine and Oumgari, Mugad},
  journal={SIAM Journal on Financial Mathematics},
  volume={14},
  number={2},
  pages={353--382},
  year={2023},
}

@article{bayer2020weak,
  title={Weak error rates for option pricing under the rough {B}ergomi model},
  author={Bayer, Christian and Hall, Eric Joseph and Tempone, Ra{\'u}l},
  journal={International Journal Theoretical and Applied Finance},
  year={2022},
volume = {25},
number = {7}
}

@article{gatheral2018volatility,
  title={Volatility is rough},
  author={Gatheral, Jim and Jaisson, Thibault and Rosenbaum, Mathieu},
  journal={Quantitative finance},
  volume={18},
  number={6},
  pages={933--949},
  year={2018},
  publisher={Taylor \& Francis}
}

@article{viens2019martingale,
  title={A martingale approach for fractional {B}rownian motions and related path-dependent {PDE}s},
  author={Viens, Frederi and Zhang, Jianfeng},
  journal={The Annals of Applied Probability},
  volume={29},
  number={6},
  pages={3489--3540},
  year={2019},
  publisher={Institute of Mathematical Statistics}
}

@article{horvath2017functional,
  title={Functional central limit theorems for rough volatility},
  author={Horvath, Blanka and Jacquier, Antoine and Muguruza, Aitor and S{\o}jmark, Andreas},
  journal={Finance and Stochastics},
  volume={28},
pages={1-47},
  year={2024}
}

@article{dupire2019functional,
  title={Functional {I}t{\^o} calculus},
  author={Dupire, Bruno},
  journal={Quantitative Finance},
  volume={19},
  number={5},
  pages={721--729},
  year={2019},
  publisher={Taylor \& Francis}
}

@article{BFG16,
  title={Pricing under rough volatility},
  author={Bayer, Christian and Friz, Peter and Gatheral, Jim},
  journal={Quantitative Finance},
  volume={16},
  number={6},
  pages={887--904},
  year={2016}
}

@book{Nualart06,
    author =       {Nualart, D.},
    title =        {The {M}alliavin {C}alculus and {R}elated {T}opics},
    publisher =     {Springer},
    year =         {2006},
}

@article{gassiat2022weak,
  title={Weak error rates of numerical schemes for rough volatility},
  author={Gassiat, Paul},
  journal={SIAM Journal on Financial Mathematics},
  volume={14},
  number={2},
  pages={475--496},
  year={2023}
}

@article{wang2022path,
  title={Path dependent {F}eynman--{K}ac formula for forward backward stochastic {V}olterra integral equations},
  author={Wang, Hanxiao and Yong, Jiongmin and Zhang, Jianfeng},
  journal={Annales de l'Institut Henri Poincar{\'e}: Probabilit{\'e}s et Statistiques},
  volume={58},
  number={2},
  pages={603--638},
  year={2022}
}

@article{peng2016bsde,
  title={{BSDE}, path-dependent {PDE} and nonlinear {F}eynman-{K}ac formula},
  author={Peng, Shige and Wang, Falei},
  journal={Science China Mathematics},
  volume={59},
  number={1},
  pages={19--36},
  year={2016},
  publisher={Springer}
}

@book{zhang2017backward,
  title={Backward {S}tochastic {D}ifferential {E}quations},
  author={Zhang, Jianfeng},
  year={2017},
  publisher={Springer}
}

@article{talay1990expansion,
  title={Expansion of the global error for numerical schemes solving stochastic differential equations},
  author={Talay, Denis and Tubaro, Luciano},
  journal={Stochastic analysis and applications},
  volume={8},
  number={4},
  pages={483--509},
  year={1990},
  publisher={Taylor \& Francis}
}

@article{talay1986discretisation,
  title={Discr{\'e}tisation d'une {\'e}quation diff{\'e}rentielle stochastique et calcul approch{\'e} d'esp{\'e}rances de fonctionnelles de la solution},
  author={Talay, Denis},
  journal={ESAIM: Mathematical Modelling and Numerical Analysis},
  volume={20},
  number={1},
  pages={141--179},
  year={1986},
  publisher={EDP Sciences}
}

@incollection{fukasawa2021hedging,
  title={Hedging under rough volatility},
  author={Fukasawa, Masaaki and Horvath, Blanka and Tankov, Peter},
booktitle = {Rough Volatility},
  publisher = {SIAM},
  year = {2023},
editor = {Bayer, Christian and Friz, Peter K. and Fukasawa, Masaaki and Gatheral, Jim and Jacquier, Antoine and Rosenbaum, Mathieu}
}

@article{bayer2022weak,
  title={On the weak convergence rate in the discretization of rough volatility models},
  author={Bayer, Christian and Fukasawa, Masaaki and Nakahara, Shonosuke},
  journal={SIAM Journal on Financial Mathematics},
  volume={13},
  number={2},
  pages={66--73},
  year={2022},
  publisher={SIAM}
}

@article{friz2022weak,
  title={Weak error estimates for rough volatility models},
  author={Friz, Peter K and Salkeld, William and Wagenhofer, Thomas},
  journal={arXiv:2212.01591},
  year={2022}
}

@article{richard2021discrete,
  title={Discrete-time simulation of stochastic {V}olterra equations},
  author={Richard, Alexandre and Tan, Xiaolu and Yang, Fan},
  journal={Stochastic Processes and their Applications},
  volume={141},
  pages={109--138},
  year={2021},
  publisher={Elsevier}
}

@article{li2022numerical,
  title={Numerical methods for stochastic {V}olterra integral equations with weakly singular kernels},
  author={Hu, Yaozhong and Huang, Chengming and Li, Min},
  journal={IMA Journal of Numerical Analysis},
  volume={42},
  number={3},
  pages={2656--2683},
  year={2022},
  publisher={Oxford University Press}
}

@incollection{pages2018numerical,
  title={Numerical {P}robability},
  author={Pag{\`e}s, Gilles},
  booktitle={Universitext},
  year={2018},
  publisher={Springer}
}

@article{alfonsi2021approximation,
  title={Approximation of Stochastic {V}olterra Equations with kernels of completely monotone type},
  author={Alfonsi, Aur{\'e}lien and Kebaier, Ahmed},
  journal={Mathematics of Computation},
  volume={93},
  number={346},
  pages={643--677},
  year={2024}
}

@article{neuenkirch2016order,
  title={The order barrier for strong approximation of rough volatility models},
  author={Neuenkirch, Andreas and Shalaiko, Taras},
  journal={arXiv:1606.03854},
  year={2016}
}

@article{jacquier2021rough,
  title={Rough multifactor volatility for {SPX} and {VIX} options},
  author={Jacquier, Antoine and Muguruza, Aitor and Pannier, Alexandre},
  journal={Advances in Applied Probability},
  year={2025}
}

@article{wu2022rough,
  title={From rough to multifractal volatility: The log {S-fBM} model},
  author={Wu, Peng and Muzy, Jean-Fran{\c{c}}ois and Bacry, Emmanuel},
  journal={Physica A: Statistical Mechanics and its Applications},
  volume={604},
  pages={127919},
  year={2022},
  publisher={Elsevier}
}

@article{romer2022empirical,
  title={Empirical analysis of rough and classical stochastic volatility models to the {SPX} and {VIX} markets},
  author={R{\o}mer, Sigurd Emil},
  journal={Quantitative Finance},
  volume={22},
  number={10},
  pages={1805--1838},
  year={2022},
  publisher={Taylor \& Francis}
}

@article{bolko2022gmm,
  title={A {GMM} approach to estimate the roughness of stochastic volatility},
  author={Bolko, Anine E and Christensen, Kim and Pakkanen, Mikko S and Veliyev, Bezirgen},
  journal={Journal of Econometrics},
  volume={235},
  number={2},
  pages={745--778},
  year={2023},
  publisher={Elsevier}
}

@article{han2021hurst,
  title={The {H}urst roughness exponent and its model-free estimation},
  author={Han, Xiyue and Schied, Alexander},
  journal={arXiv:2111.10301},
  year={2021}
}

@article{bezborodov2019option,
  title={Option pricing with fractional stochastic volatility and discontinuous payoff function of polynomial growth},
  author={Bezborodov, Viktor and Di Persio, Luca and Mishura, Yuliya},
  journal={Methodology and Computing in Applied Probability},
  volume={21},
  pages={331--366},
  year={2019},
  publisher={Springer}
}

@article{clement2006duality,
  title={A duality approach for the weak approximation of stochastic differential equations},
  author={Cl{\'e}ment, Emmanuelle and Kohatsu-Higa, Arturo and Lamberton, Damien},
  journal={The Annals of Applied Probability},
  volume={16},
  number={3},
  pages={1124--1154},
  year={2006}
}

@article{gassiat2019martingale,
  title={On the martingale property in the rough Bergomi model},
  author={Gassiat, Paul}, journal={Electronic Communications in Probability}, 
    volume={24},
pages={1--9},
  year={2019}
}

@article{abi2025martingale,
  title={Martingale property and moment explosions in signature volatility models},
  author={Abi Jaber, Eduardo and Gassiat, Paul and Sotnikov, Dimitri},
  journal={arxiv:2503.17103},
  year={2025}
}
\end{document}